\documentclass{article}

\usepackage{amsmath,amsfonts,amssymb,amsthm,booktabs,color,epsfig,graphicx,hyperref,url}

\usepackage{fullpage}
\usepackage{enumerate}
\usepackage[usenames, dvipsnames]{xcolor} 
\usepackage{dsfont} 
\usepackage{float} 
\usepackage{bm} 
\makeatletter
\g@addto@macro\@floatboxreset\centering
\makeatother

\usepackage{fancyhdr}
 
\newtheorem{theorem}{Theorem}[section]
\newtheorem{lemma}[theorem]{Lemma}
\newtheorem{remark}[theorem]{Remark}

\newtheorem{example}[theorem]{Example}

\newtheorem{corollary}[theorem]{Corollary}
\newtheorem{proposition}[theorem]{Proposition}
\newtheorem{test}[theorem]{Test}

\newcommand{\E}{\mathbb{E}}
\newcommand{\N}{\mathbb{N}}
\newcommand{\R}{\mathbb{R}}
\newcommand{\Prob}{\mathbb{P}}

\newcommand{\Askript}{\mathcal{A}}
\newcommand{\Bskript}{\mathcal{B}}
\newcommand{\Cskript}{\mathcal{C}}

\newcommand{\Mskript}{\mathcal{M}}
\newcommand{\Iskript}{\mathcal{I}}
\newcommand{\Jskript}{\mathcal{J}}
\newcommand{\Qskript}{\mathcal{Q}}
\newcommand{\Rskript}{\mathcal{R}}
\newcommand{\Sskript}{\mathcal{S}}

\newcommand{\eqd}{\stackrel{d}{=}}

\newcommand{\Cov}{\operatorname{Cov}}

\newcommand{\Var}{\operatorname{Var}}
\newcommand{\sgn}{\operatorname{sign}}

\newcommand{\Mcor}{M\mathrm{cor}}

\newcommand{\ii}{\mathrm{i}}
\newcommand{\ee}{\mathrm{e}}

\newcommand{\Mfrac}[2]{#1} 

\newcommand{\hN}{\mbox{}^{\scriptscriptstyle N}\kern-1.5pt}

\usepackage{ifdraft,ifthen}
\newcommand{\draftremark}[1]{\ifdraft{%
   \textcolor{red}{// {#1} //}%
   \ifthenelse{\isodd{\value{page}}}{%
	\normalmarginpar\marginpar{\textcolor{red}{$\maltese$}}}{%
	\reversemarginpar\marginpar{\textcolor{red}{$\maltese$}}}}{}}

\newcommand{\off}[1]{}

\begin{document}

\title{Dependence and dependence structures: estimation and visualization using the unifying concept of distance multivariance}
\author{Bj\"{o}rn B\"{o}ttcher\footnote{TU Dresden, Fakult\"at Mathematik, Institut f\"{u}r Mathematische Stochastik, 01062 Dresden, Germany, email: \href{mailto:bjoern.boettcher@tu-dresden.de}{bjoern.boettcher@tu-dresden.de}}}
\date{\today}
\maketitle


\begin{abstract}
Distance multivariance is a multivariate dependence measure, which can detect dependencies between an arbitrary number of random vectors each of which can have a distinct dimension. Here we discuss several new aspects, present a concise overview and use it as the basis for several new results and concepts: In particular, we show that distance multivariance unifies (and extends) distance covariance and the Hilbert-Schmidt independence criterion HSIC, moreover also the classical linear dependence measures: covariance, Pearson's correlation and the RV coefficient appear as limiting cases. Based on distance multivariance several new measures are defined: a multicorrelation which satisfies a natural set of multivariate dependence measure axioms and $m$-multivariance which is a dependence measure yielding tests for pairwise independence and independence of higher order. These tests are computationally feasible and under very mild moment conditions they are consistent against all alternatives.
Moreover, a general visualization scheme for higher order dependencies is proposed, including consistent estimators (based on distance multivariance) for the dependence structure.

Many illustrative examples are provided. All functions for the use of distance multivariance in applications are published in the R-package \texttt{multivariance}.
\end{abstract}

\noindent \textit{Keywords:} 
\newcommand{\sep}{, }
multivariate dependence\sep 
testing independence\sep
visualizing higher order dependence\sep
test of pairwise independence\sep 
multivariate dependence measure axioms\sep distance covariance\sep Hilbert Schmidt independence criterion\sep HSIC\sep joint distance covariance.




\section{Introduction}
The detection of dependence is a common statistical task, which is crucial in many applications. There have been many methods employed and proposed, see e.g.\ \cite{JossHolm2016}, \cite{TjoesOtneStoev2018} and \cite{LiuPenaZhen2018} for recent surveys. Usually these focus on the (functional) dependence of pairs of variables. Thus when the dependence of many variables is studied the resulting networks (correlation networks, graphical models) only show the pairwise dependence. As long as pairwise dependence is present, this might be sufficient (and also for the detection of such dependencies total multivariance and $m$-multivariance provide efficient tests). But recall that pairwise independence does not imply the independence of all variables if more than two variables are considered. Thus, in particular if all variables are pairwise independent many methods of classical statistical inference would have discarded the data. Although there might be some higher order dependence present. This can only be detected directly with a multivariate dependence measure. The classical examples featuring 3-dependence are a dice in the shape of a tetrahedron with specially colored sides (see Example \ref{ex:tetra}) and certain events in multiple coin throws (Examples \ref{ex:2coins}). In Example \ref{ex:ncoins} a generalization to higher orders is presented.

To avoid misconceptions when talking about independence one should note that the term ``mutual independence'' is ambiguous, some authors use it as a synonym for pairwise independence, others for independence. For the latter also the terms ``total independence'' or ``joint independence'' are used. We use the terms: pairwise independence, its extension $m$-independence (see Section \ref{sec:defs}) and independence.

In \cite{BoetKellSchi2018,BoetKellSchi2019} the basics of distance multivariance and total distance multivariance were developed, which can be used to measure multivariate dependence. Incidentally, a variant of total distance multivariance based on the Euclidean distance was simultaneously developed in \cite{ChakZhan2019}. Moreover, distance multivariance names and extends a concept introduced in \cite{BiloNang2017}.
Here we recall and extend the main definitions and properties (Sections \ref{sec:defs} and \ref{sec:test}). In particular, the moment conditions required in \cite{BoetKellSchi2019} for the independence tests  are considerably relaxed (Theorem \ref{thm:convergence}, Tests \ref{testA} and \ref{testB}), invariance properties are explicitly discussed (Propositions \ref{prop:inv} and \ref{prop:scaleinv}) and resampling tests are introduced (Section \ref{sec:test}). Moreover, on this basis the following new results and concepts are developed:
\begin{itemize}
\item A general \textbf{scheme for the visualization of higher order dependence} which can be used with any multivariate dependence measure (Section \ref{sec:algo}). For the setting of multivariance we provide explicit \textbf{consistent estimators for the (higher order) dependence structure}. In particular the method for the clustered dependence structure is based on the fact that multivariance is a truely multivariate dependence measure: On the one hand it can measure the dependence of multiple (more than 2) random variables. On the other hand each random variable can be multivariate and each can have a distinct dimension. 
\begin{figure}[H]
\includegraphics[width = 0.25\textwidth]{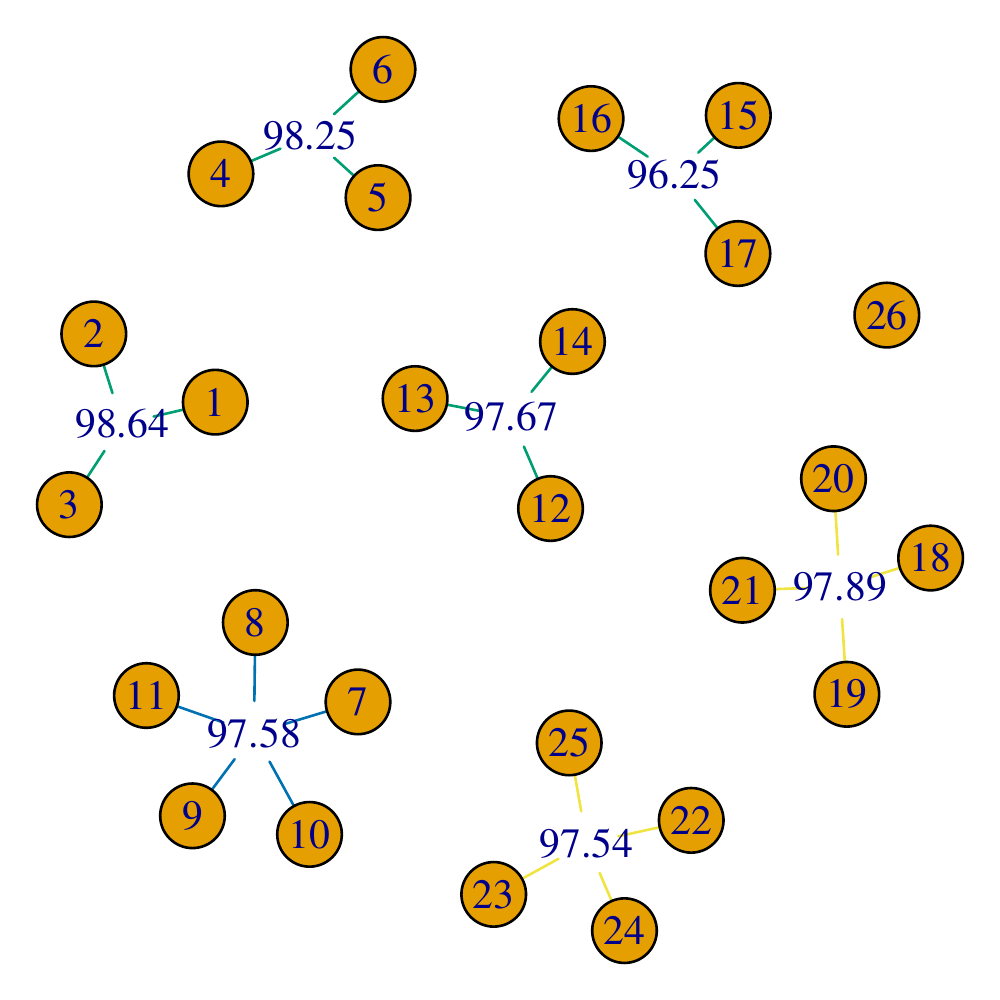}
\quad\quad\quad
\includegraphics[width = 0.25\textwidth]{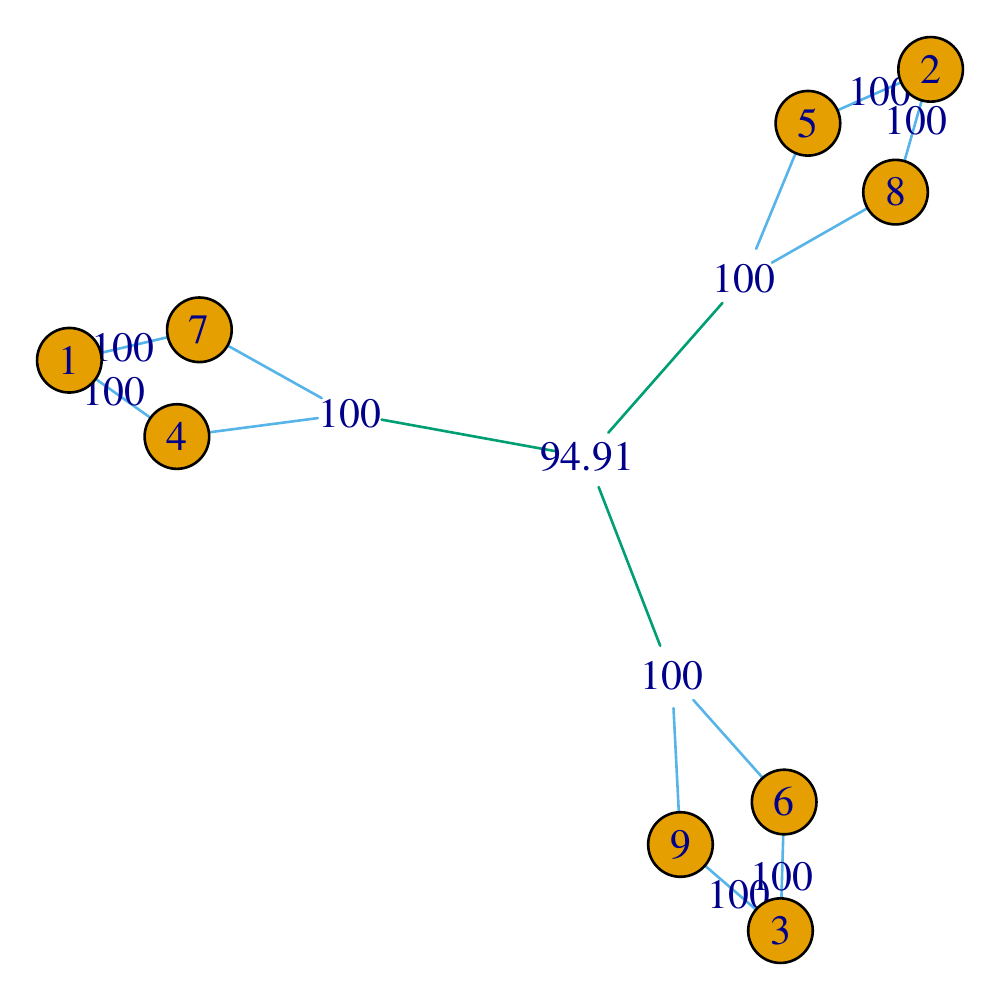}
\quad\quad\quad
\includegraphics[width = 0.25\textwidth]{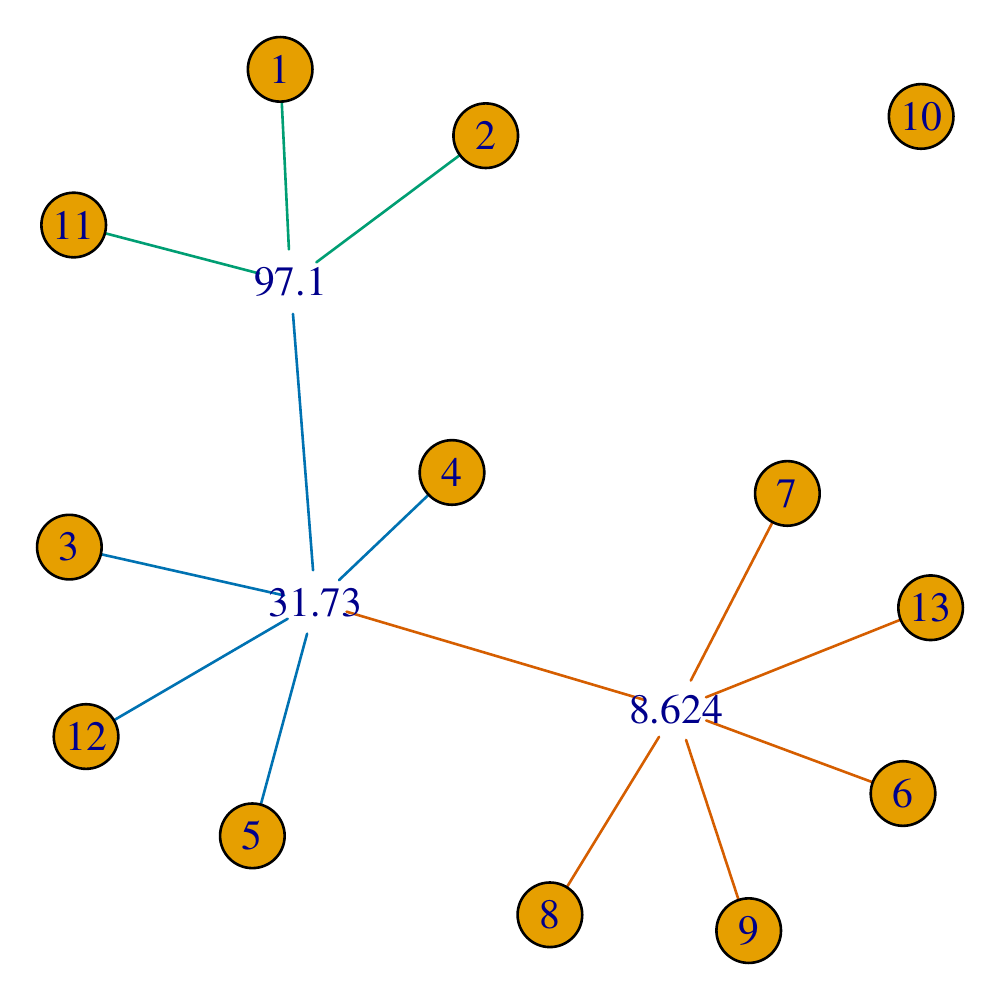}
\caption{Visualized dependence structures of Examples \ref{ex:several}, \ref{ex:star} and \ref{ex:iterated}.}		\label{fig:intro}
\end{figure}
\item \textbf{Global tests for pairwise (and higher order) independence}: Pairwise independence is a fundamental requisite e.g.\ for the law of large numbers (in its basic form; a result which is used in almost every statistical estimation). Recently in \cite{YaoZhanShao2017} a test for pairwise independence of identically distributed random variables was proposed. We derive in Section \ref{sec:mmulti} a test for pairwise (and higher order) independence which is applicable to any mixture of marginal distributions and dimensions.
\item \textbf{Unification of dependence measures} (Section \ref{sec:compare}): 
\begin{itemize}
\item In \cite{SejdSripGretFuku2013} it was shown that for independence testing methods based on reproducing kernels and methods based on distances are equivalent. They considered the bivariate setting. We present in Section \ref{sec:unify-dcov-hsic} a different, explicit and very elementary relation. Moreover, this transfers also to the setting of multiple random vectors. 
\item It was clear that the RV-coefficent structurally corresponds to distance covariance, see e.g.\ \cite{JossHolm2016}. We show in Section \ref{sec:Cov-RV-Pearson} that the RV-coefficent and, in particular, covariance (the most classical dependence measure of all!) are covered as limiting cases by the framework of multivariance.
\item Hilbert Space methods require characteristic kernels. Multivariance requires the full support of the underlying measures. We show that these assumptions are equivalent (Proposition \ref{prop:fullsupport}).
\end{itemize}
\item \textbf{Multivariate correlations}: A formalization of desired properties of dependence measures goes back at least to Reni's Axioms \cite{Reny1959}. We discuss in Section \ref{sec:axioms} a multivariate extension of the Axioms and provide several multicorrelations, e.g.\ \eqref{eq:multicorr}, \eqref{eq:Mcor} and \eqref{eq:mcor2}.
\end{itemize}
Recently also several other dependence measures for multiple random vectors were proposed, e.g.\ \cite{YaoZhanShao2017} propose tests for pairwise independence, banded independence and independence based on distance covariance or based on the approach of \cite{PfisBuehSchoPete2017}. The latter presented tests for independence of multiple random vectors using kernels. In \cite{JinMatt2018} also distance covariance (a measure for the dependence of two random vectors; introduced in \cite{SzekRizzBaki2007}) was generalized to tests of independence of multiple random vectors. All these approaches are related to distance multivariance, see Section \ref{sec:compare} for a detailed discussion. Empirical comparisons can be found in Examples \ref{ex:JinMatt2017}  and  \ref{ex:YaoZhanShao2017}. 

It is remarkable that, although the above measures are able to detect higher order dependencies, all real data examples which were provided so far feature only pairwise dependence. Certainly the predominant statistical methods cause a publication bias for such datasets. Nevertheless, we want to point out that many available datasets feature higher order dependence. Based on a data driven study we collected over 350 datasets featuring statistically significant higher order dependencies \cite{Boet2019-tech}. One of these datasets is discussed in further details in Section \ref{sec:ex-real-data}, and all of these datasets are distributed as part of various R-packages without the context of higher order dependence. This indicates that higher order dependence can be detected frequently, but what remains open (with the potential to be the basis for various new research projects) are intrinsic explanations of higher order dependence within each field of research of the underlying data. 

Besides the real data examples the presentation of this paper is complemented by a comprehensive collection of further examples (in Section \ref{sec:examp} and in the supplementary\footnote{pages \pageref{sec:supp-example} ff.\ of this manuscript} Section \ref{sec:supp-example}): illustrating higher order dependencies (Section \ref{sec:ex-higher-order}), discussing various properties of distance multivariance (Section \ref{sec:ex-properties}), comparisons to other dependence measures (Section \ref{sec:ex-compare}). Technical details and further results are collected in Section \ref{sec:appendix}. 




The R code for the evaluation of distance multivariance and the corresponding tests is provided in the R-package \texttt{multivariance} \cite{Boett2019R-2.2.0}. Finally, based on (some of) the results of this paper we have the following recommendations for questions common in independence testing:
\begin{enumerate}
\item \textit{Are at least some variables dependent? Detection of any kind of dependence:}
\begin{enumerate}
\item The global independence test based on total multivariance can be used to detect any kind of dependence, alternatively 2-multivariance can be used to test for pairwise (in)dependence. The latter and $m$-multivariance can also be used to reduce the statistical curse of dimension which total multivariance might suffer. For all settings fast distribution free (conservative) tests exist and these are applicable for large samples and a large number of random vectors. The computation of the test statistic takes in its current implementation for 100 variables with 1000 samples each (or for 1000 variables with 300 samples each) less than 2 seconds on a dated i7-6500U CPU Laptop. Slower, but approximately sharp, are the corresponding resampling tests. Faster approximately sharp tests are discussed in \cite{BersBoet2018v2}.
\item  As a complementary approach to the global tests one could perform multiple tests as suggested in \cite{BiloNang2017}. This requires $2^n-n-1$ individual tests, where $n$ denotes the number of random vectors. Hence it is only applicable for small $n$. \cite{BiloNang2017} also provides a multiple testing approach to $m$-dependence.
\end{enumerate}
\item \textit{Which variables depend on each other? Dependence structure:}
Especially if some dependence was detected the algorithm of Section \ref{sec:algo} can be used to analyze which variables depend on each other, yielding either a full or clustered dependence structure. The method is based on multiple tests, but variables are clustered (or related tuples are excluded from further tests) as soon as a positive detection occurred. This can considerably reduce the computation time in comparison to 1.(b).
\end{enumerate}

\off{In the next sections we will discuss the following topics: 
\begin{itemize}
\item Definition and properties of multivariance. (Section \ref{sec:defs})
\item Prerequisites for using multivariance: $\psi_i$, $\rho$ and the moment conditions. (Section \ref{sec:prereq})
\item Testing independence using multivariance. (Section \ref{sec:test})
\item Detecting $m$-independence. (Section \ref{sec:mmulti})
\item Detect and visualize dependence structures using multivariance. (Section \ref{sec:algo})
\item Examples. (Section \ref{sec:example})
\end{itemize}
}

\section{Distance multivariance} \label{sec:defs}
In the following distance multivariance is introduced. Some parts are essential for the (less technical) comparison to other dependence measures in Section \ref{sec:compare}, other parts are required for the introduction of $m$-multivariance (Section \ref{sec:mmulti}). Furthermore, several new results are included which make distance multivariance more accessible and applicable. Tests using distance multivariance will be discussed in Section \ref{sec:test}.

Let $X_i$ be $\R^{d_i}$ valued random variables with characteristic functions $f_{X_i}(t_i) = \E(\ee^{\ii t_i X_i})$ for $i =1,\ldots,n$. Then \textbf{distance multivariance} is defined by 
\begin{equation} \label{def:multi}
M_\rho(X_1,\ldots,X_n) := M_\rho(X_i, i\in\{1,\ldots,n\}) := \sqrt{ \int \left| \E\left(\prod_{i=1}^n (\ee^{\ii X_i t_i} - f_{X_i}(t_i))\right)\right|^2\, \rho(dt) }
\end{equation}
and \textbf{total distance multivariance} is
\begin{equation} \label{eq:totaldm}
\overline{M_\rho}(X_1,\ldots,X_n) :=  \sqrt{\sum_{\substack{1\leq i_1< \ldots < i_m \leq n\\2 \leq m \leq n}} M_{\otimes_{k=1}^m \rho_{i_k}}^2(X_{i_1},\ldots,X_{i_m})},
\end{equation}
where $\rho = \otimes_{i=1}^n \rho_i$ and each $\rho_i$ is a symmetric measure with full topological support\footnote{A measure $\rho$ has \textbf{full topological support} on $\R^{d}$ if and only if $\rho(O)>0$ for all open sets $O \subset \R^d$. See also Proposition \ref{prop:fullsupport}.} on $\R^{d_i}$ such that $\int 1 \land |t_i|^2 \,\rho_i(dt_i)<\infty$ and $t=(t_1,\ldots,t_n)$ with $t_i \in \R^{d_i}$. To simplify notation and definitions we will just use the term 'multivariance' instead of 'distance multivariance', and we will drop the subscript $\rho$ if the measure is the full measure $\rho$. 

Random variables $X_1,\ldots,X_n$ are called \textbf{$\bm{m}$-independent}, if $X_{i_1},\ldots,X_{i_m}$ are independent for any distinct $i_j \in \{1,\ldots,n\}$ for $j=1,\ldots,m.$ This concept is essential for the statement (and proof) of the following theorem. It is also the basis for the estimators for $m$-independence which will be developed in Section \ref{sec:mmulti}. 

\begin{theorem}[Characterization of independence, \protect{\cite[Theorem 3.4.]{BoetKellSchi2019}}] \label{thm:indep} For random variables $X_1,\ldots, X_n$ the following are equivalent:
\begin{enumerate}[i)]
\item $X_1,\ldots,X_n$ are independent, 
\item $M(X_1,\ldots,X_n) = 0$ and  $X_1,\ldots,X_n$ are $(n-1)$-independent,
\item $\overline{M}(X_1,\ldots,X_n) = 0$.
\end{enumerate}
\end{theorem}

For statistical applications the following representations, which require the moment condition \eqref{d1}, are very useful. Let $(X_1',\ldots,X_n')$ be an independent copy of $(X_1,\ldots,X_n)$ and $\psi_i(y_i) := \int_{\R^{d_i}\backslash \{0\}} 1-\cos(y_i\cdot t_i) \,\rho_i(dt_i)$ then
\begin{equation}
\label{eq:M} M_\rho^2(X_1,\ldots,X_n) = \E\left(\prod_{i=1}^n \Psi_i(X_i,X_i') \right) \quad \text{ and } \quad 
\overline{M_\rho}^2(X_1,\ldots,X_n) = \E\left(\prod_{i=1}^n (1+\Psi_i(X_i,X_i')) \right) - 1
\end{equation}
where
\begin{equation}
\label{eq:Psi}\Psi_i(X_i,X_i'):=-\psi_i(X_i-X_i') + \E(\psi_i(X_i-X_i')\mid X_i) + \E(\psi_i(X_i-X_i')\mid X_i') -  \E(\psi_i(X_i-X_i')).
\end{equation}

Note that in \cite{BoetKellSchi2019} a technical looking moment condition was required for the above representations, we show in Section \ref{sec:momcon} that the following more comprehensible condition is equivalent to it (for non constant random variables)
\begin{equation}
\label{d1}
\text{\textbf{finite joint $\psi$-moments}: } \text{ for all $S\subset\{1,\ldots,n\}: $ } \E \left( \prod_{i\in S} \psi_i(X_{i})\right) < \infty.
\end{equation}
A direct consequence of \eqref{eq:M} is the factorization of $M$ and $\overline M$ for independent subsets, i.e., if $(X_i)_{i\in I}$ and $(X_i)_{i\in I^c}$ are independent for some $I \subset \{1,\ldots,n\}$ then
\begin{align} \label{eq:factorization}
M(X_i, i\in I\cup I^c) &= M_{\otimes_{i\in I}\rho_i}(X_i,i\in I) \cdot M_{\otimes_{i\in I^c}\rho_i}(X_i,i\in I^c), \\
\overline M^2(X_i, i\in I\cup I^c) + 1 &=  (\overline M^2_{\otimes_{i\in I}\rho_i}(X_i,i\in I)+1) \cdot (\overline M^2_{\otimes_{i\in I^c}\rho_i}(X_i,i\in I^c)+1).
\end{align}
Furthermore, the expectations in \eqref{eq:M} yield strongly consistent (see \cite[Theorem 4.3]{BoetKellSchi2019} and Corollary \ref{cor:betaM-consistency}) and numerically feasible estimators. Hereto denote samples of $(X_1,\ldots,X_n)$ by $\bm{x}^{(k)} = (x_1^{(k)}, \ldots ,x_n^{(k)})\in \R^{d_1}\times \ldots \times \R^{d_n}$  for $k = 1,\ldots, N$. Then \textbf{sample multivariance} is defined by 
\begin{equation}
\label{def:smulti} \hN M^2(\bm{x}^{(1)},\ldots,\bm{x}^{(N)}) :=   \frac{1}{N^2} \sum_{j,k=1}^N (A_1)_{jk} \cdot \ldots \cdot (A_n)_{jk}
\end{equation}
and \textbf{sample total multivariance} is
\begin{equation}
\label{def:stotalmulti} \hN \overline{M}^2(\bm{x}^{(1)},\ldots,\bm{x}^{(N)}) :=  \left[\frac{1}{N^2} \sum_{j,k=1}^N (1+(A_1)_{jk}) \cdot \ldots \cdot (1+(A_n)_{jk})\right] - 1,
\end{equation}
where $(A_i)_{jk}$ denotes the element in the $j$-th row and $k$-th column of the doubly centered distance matrix $A_i$ defined by
\begin{equation}
\label{eq:A} A_i := -CB_iC \text{ with }B_i := \left(\psi_i(x_i^{(j)}-x_i^{(k)})\right)_{j,k = 1,\ldots,N}\text{ and } C := \left(\delta_{jk} -\frac{1}{N}\right)_{j,k = 1,\ldots, N}. 
\end{equation}  
The matrices $A_i$ are positive definite \cite[Remark 4.2.b]{BoetKellSchi2019}, since the considered distances $\psi_i(. - .)$ are given by 
\begin{equation} \label{eq:psi}
\psi_i(y_i) := \int_{\R^{d_i}\backslash \{0\}} 1-\cos(y_i\cdot t_i) \,\rho_i(dt_i) \text{ for }y_i\in \R^{d_i}.
\end{equation} 
Functions defined via \eqref{eq:psi} appear in various areas: e.g.\ they are called variogram (e.g.\ \cite{Math1963}), continuous negative definite function (e.g.\ \cite{BergFors75}) or characteristic exponent of a L\'evy process with L\'evy measure $\rho_i$ (e.g.\ \cite{Sato99}), and they are closely related to the symbol of generators of Markov processes (\cite{Jaco2001}, \cite{BoetSchiWang2013}). The choice of $\rho_i$ and $\psi_i$ is discussed in more detail in Remark \ref{rem:choosepsi}, the standard choices are the Euclidean distance $\psi_i(t_i) = |t_i|$ and for $\alpha_i \in (0,2)$ stable distances $\psi_i(t_i) = |t_i|^{\alpha_i}$ and bounded functions of the form  $\psi_i(t_i)=1- \exp(-\delta_i|t_i|^{\alpha_i})$ with $\delta_i>0.$ But also other functions like Minkowski distances are possible, various examples are given in \cite[Table 1]{BoetKellSchi2018}. 

We call a function $\psi$ \textbf{characterizing} if for any random vector $X$ the function $z\mapsto \E(\psi(X-z))$ characterizes the distribution of $X$ uniquely, or equivalently, if for finite measures $\mu$ the function $\mu \mapsto \int \psi(x-.) \mu(dx)$ is injective. The following Proposition provides a characterization of the required support property of $\rho$ in terms of $\psi$, it actually solves an open problem of \cite{BoetKellSchi2018}. Moreover it also links the setting of multivariance to other dependence measures, see Section \ref{sec:compare}.
\begin{proposition}\label{prop:fullsupport}
Let $\psi_i$ be given by \eqref{eq:psi} for a symmetric measure $\rho_i$ such that $\int 1\land|t_i|^2\,\rho_i(t_i) < \infty.$ Then $\psi_i$ is characterizing if and only if $\rho_i$ has full topological support.
\end{proposition}
\begin{proof}
The statement is a consequence of Theorem \ref{thm:rhochar} (see Section \ref{sec:appendix}). Hereto note that the distributions of two random vectors coincide if and only if their characteristic functions coincide on a dense subset, i.e., $\mu$ almost surely for a measure $\mu$ with full topological support.
\end{proof}

There are important scaled versions of the estimators in \eqref{def:smulti} and \eqref{def:stotalmulti}:
 
\begin{itemize}
\item 
\textbf{normalized sample (total) multivariance}: We write $\Mskript$ instead of $M$, if each $A_i$ in \eqref{def:smulti} and \eqref{def:stotalmulti} is replaced by 
\begin{equation} \label{eq:normalizedm}
\Askript_i := \begin{cases}\frac{1}{\hN a_i} A_i &\text{ if $\hN a_i>0$}\\0 &\text{ if $\hN a_i=0$}
\end{cases} \text{, where }\hN a_i:=\frac{1}{N^2} \sum_{j,k=1}^N \psi_i(x_i^{(j)}-x_i^{(k)}) \text{ which estimates }\E(\psi_i(X_i-X_i')).
\end{equation}
In the case of normalized sample total multivariance the sum in \eqref{def:stotalmulti} is additionally scaled by the number of summands in the definition of total multivariance \eqref{eq:totaldm}, i.e., 
\begin{equation} \label{eq:sampletm}
\hN \overline{\Mskript}^2(\bm{x}^{(1)},\ldots,\bm{x}^{(N)}) := 
 \frac{1}{2^n-n-1} \left\{ \left[\frac{1}{N^2} \sum_{j,k=1}^N \prod_{l =1}^n (1+(\Askript_l)_{jk}) \right] - 1 \right\}. 
\end{equation}
By this scaling the test statistics for multivariance and total multivariance have expectation 1 (in the case of independent variables).
\item \textbf{sample multicorrelation}:   We write $\mathcal{R}$ instead of $M$, if each $A_i$ in \eqref{def:smulti} is replaced by  
\begin{equation}
\label{eq:multicorr} 
\Bskript_i := \begin{cases}\frac{1}{\hN b_i} A_i &\text{ if $\hN b_i>0$}\\0 &\text{ if $\hN b_i=0$}
\end{cases}\text{, where } 
\hN b_i:= \left(\frac{1}{N^2}  \sum_{j,k=1}^N |(A_i)_{kl}|^n\right)^{1/n}\text{ which estimates }\left(\E(|\Psi_i(X_i,X_i')|^n)\right)^{1/n}.
\end{equation}
\item \textbf{unnormalized sample multicorrelation}: We write $\Mcor$ instead of $M$, if each $A_i$ in \eqref{def:smulti} is replaced by 
\begin{equation}
\label{eq:Mcor} 
\Cskript_i := \begin{cases}\frac{1}{\hN c_i} A_i &\text{ if $\hN c_i>0$}\\0 &\text{ if $\hN c_i=0$} \end{cases}\text{, where }
 \hN c_i:= \left(\frac{1}{N^2}  \sum_{j,k=1}^N (A_i)_{kl}^n\right)^{1/n} \text{ which estimates }\left(M^2_{\otimes_{k=1}^n \rho_i}(\underbrace{X_i,\ldots,X_i}_{n\text{-times}})\right)^{1/n}.
\end{equation}
\end{itemize}
Note that $\Mcor$ is newly introduced in this paper, see in particular Table \ref{tab:axioms} for a comparison. For even $n$ multicorrelation $\mathcal{R}$ and $\Mcor$ coincide, but for odd $n$ they differ. Only $\mathcal{R}$ is always bounded by 1, hence $\Mcor$ is called \textit{unnormalized}. But only for $\Mcor$ the value 1 has an explicit interpretation. The population versions of the above sample measures are given by scaling $\Psi_i$ in \eqref{eq:M} with the terms on the right of \eqref{eq:normalizedm}, \eqref{eq:multicorr} and \eqref{eq:Mcor}, e.g.\ normalized multivariance and normalized total multivariance are
\begin{equation} \label{eq:popMMt}
\Mskript_\rho^2(X_1,\ldots,X_n) = \frac{M^2(X_1,\ldots,X_n)}{\prod_{i=1}^n \E(\psi_i(X_i-X_i'))}\quad \text{ and }  \quad
\overline{\Mskript}_\rho^2(X_1,\ldots,X_n) =\frac{\E\left(\prod_{i=1}^n \left(1+\frac{\Psi_i(X_i,X_i')}{ \E(\psi_i(X_i-X_i'))}\right)\right)-1}{2^n-n-1},
\end{equation}
where implicitly the finiteness of the corresponding moments is assumed, i.e.,
\begin{equation} \label{eq:mom-1}
\text{\textbf{finite first $\psi$-moments}: }\text{ for all }i = 1,\ldots,n:\ \E(\psi_i(X_i))<\infty. 
\end{equation}
For the scaling of the multicorrelations one has to assume
\begin{equation} \label{eq:mom-n}
\text{\textbf{finite $\psi$-moments of order $\bm{n}$}: }\text{ for all }i = 1,\ldots,n:\ \E(\psi_i^n(X_i))<\infty.
\end{equation}
Note that the scaling factors given in \eqref{eq:multicorr} and \eqref{eq:Mcor} depend on $n$ thus the total multicorrelations do not have such a simple representation as $\overline{\Mskript}$ (or its sample version \eqref{eq:sampletm}) in fact the following holds (analogously also for $\Mcor$):
\begin{align} \label{eq:totalmcor}
\overline{\Rskript_\rho}^2(X_1,\ldots,X_n) &:=  \frac{1}{{2^n-n-1}}\sum_{\substack{1\leq i_1< \ldots < i_m \leq n\\2 \leq m \leq n}} \frac{M_{\otimes_{k=1}^m \rho_{i_k}}^2(X_{i_1},\ldots,X_{i_m})}{\prod_{k=1}^m \left(\E(|\Psi_{i_k}(X_{i_k},X_{i_k}')|^m)\right)^{1/m}} \\
\label{eq:totalmcor-bound}&\geq \frac{\E\left(\prod_{i=1}^n \left(1+\frac{\Psi_i(X_i,X_i')}{ \left(\E(|\Psi_{i}(X_{i},X_{i}')|^n)\right)^{1/n}}\right)\right)-1}{2^n-n-1}.
\end{align}
Therefore the total multicorrelations seem more of a theoretic interest, but the corresponding $m$-multicorrelations (which will be defined in Remark \ref{rem:mmulti}.\ref{remsub:mmulticor}) have efficient estimators. Moreover, also the lower bound in \eqref{eq:totalmcor-bound} has an efficient sample version analogously to \eqref{eq:sampletm}.  For a comparison of these multicorrelations see Section \ref{sec:axioms}.

The introduced dependence measures and their sample versions have in particular the following properties.
\begin{proposition}[Invariance properties of multivariance] \label{prop:inv}
The following properties hold for $M, \overline{M}, \mathcal{M}, \overline{\mathcal{M}}, \mathcal{R}, \overline{\mathcal{R}},  \Mcor,$ $ \overline{\Mcor}$ and the corresponding sample versions, to avoid redundancy we only explicitly state them for $M$: 
\begin{enumerate}
\item \textbf{trivial for single variables}, i.e., $M_{\rho_i}(X_i) = 0$ for all $i\in\{1,\ldots,n\}.$
\item \textbf{permutation invariant}, i.e.,
$
M(X_1,\ldots,X_n) =  M_{\otimes_{i=1}^n\rho_{k_i}}(X_{k_1},\ldots,X_{k_n})$ for all permutations $k_1,\ldots,k_n$ of $1,\ldots,n.$ Moreover, the sample versions are in addition invariant with respect to permutations of the samples, i.e., the equality $\hN M(x^{(1)},\ldots,x^{(N)}) = \hN M(x^{(l_1)},\ldots,x^{(l_N)})$ holds for all permutations $l_1,\ldots,l_N$ of $1,\ldots,N.$ (This should not be confused with the permutations for a resampling test, where components of the samples are permuted separately, see \eqref{eq:Rresampling}.)
\item \label{eq:Msymmetric} \textbf{symmetric in each variable}, i.e., $ M(X_1,\ldots,X_n) = M(c_1 X_1,\ldots,c_n X_n)$  for all $c_i \in \{-1,1\}.
$
\item \label{eq:Mtransinv} \textbf{translation invariant}, i.e.,  $M(X_1-r_1, \ldots, X_n-r_n) = M(X_1, \ldots, X_n)$ for all $r_i\in \R^{d_i}.$\\
Note that the latter and \eqref{eq:Msymmetric} imply that for dichotomous 0-1 coded data a swap of the coding does not change the value of the multivariance.
\item  \textbf{rotation invariant for isotropic $\bm{\psi_i}$}, i.e., if $\psi_i(x_i) = g_i(|x_i|)$ for some $g_i$ and all $i=1,\ldots,n$, then  $ M(X_1,\ldots,X_n) = M(R_1 X_1,\ldots,R_n X_n)$ for all rotations $R_i$ on $\R^{d_i}$.\\
Note that in this case, since also \eqref{eq:Msymmetric} and \eqref{eq:Mtransinv} hold, $M$ is invariant with respect to Euclidean isometries.
\end{enumerate}
\end{proposition}
\begin{proof}
For multivariance $M$ the property (a) follows by direct calculation using \eqref{eq:M} and \eqref{eq:Psi}, (b) is obvious by \eqref{eq:M}, (c) holds since $\psi_i$ is symmetric and for (d) note that the translations cancel in \eqref{eq:Psi}. Moreover, since a rotation preserves Euclidean distances also (e) holds. Total multivariance $\overline{M}$ is just a sum of multivariances, hence it inherits these properties. 

For sample multivariance $\hN M$ the same arguments apply using \eqref{def:smulti} and \eqref{eq:A}. For the sample permutation invariance in (b) note that permutations of samples correspond to permutations of rows and columns of the centered distance matrices.  Analogously also the scaling factors given in \eqref{eq:normalizedm}, \eqref{eq:multicorr} and \eqref{eq:Mcor} have these properties. Therefore they hold also for all scaled and sample versions.
\end{proof} 
Moreover, for special functions $\psi_i$ the scaled dependence measures feature one further invariance.
\begin{proposition}[{Scale invariance of scaled multivariance for ${\psi_i(x_i) = |x_i|^{\alpha_i}}$}] \label{prop:scaleinv}
Let $\psi_i(x_i) = |x_i|^{\alpha_i}$ with $\alpha_i \in (0,2)$. Then the scaled measures $\mathcal{M}, \overline{\mathcal{M}}, \mathcal{R}, \overline{\mathcal{R}},  $ $\Mcor, \overline{\Mcor}$ and the corresponding sample versions are scale invariant, that is, 
\begin{equation}
\mathcal{M} (r_1 X_1,\ldots, r_n X_n) = \mathcal{M} (X_1,\ldots,X_n) \text{ for all } r_i \in \R\backslash\{0\}.
\end{equation}
\end{proposition}
\begin{proof}
For $\psi_i(x_i)= |x_i|^{\alpha_i}$ note that $\Psi_i$ given in \eqref{eq:Psi} satisfies $\Psi_i(r_iX_i,r_iX_i') = |r_i|^{\alpha_i} \Psi_i(X_i,X_i').$ Thus multivariance is $\alpha$-homogeneous, i.e., $M(r_1X_1,\ldots,r_nX_n) =  M(X_1,\ldots,X_n) \prod_{i=1}^n |r_i|^{\alpha_i}.$ The same holds (using \eqref{eq:A}) for $\hN M$ and also for the scaling factors given in \eqref{eq:normalizedm}, \eqref{eq:multicorr} and \eqref{eq:Mcor}. Thus the factors $|r_i|^{\alpha_i}$ cancel by the scaling.
\end{proof}

The key for statistical tests based on multivariance is the following convergence result. The presented result relaxes the required moments considerably in comparison to \cite[Thm.\ 4.5, 4.10, Cor.\ 4.16, 4.18]{BoetKellSchi2019}, moreover also a new parameter $\beta$ is introduced which will be useful in Section \ref{sec:algo}.

\begin{theorem}[Asymptotics of sample multivariance] \label{thm:convergence} Let $X_i,$ $i=1,\ldots,n$ be non-constant random variables and let $\bm{X}^{(k)}, k=1,\ldots,N$ be independent copies of $\bm{X}=(X_1,\ldots,X_n)$. Let either of the following conditions hold 
\begin{equation} \label{eq:psibounded}
\text{ all }\psi_i \text{ are bounded }
\end{equation}
or 
\begin{equation} \label{eq:mom1-log}
\text{ for all }i = 1,\ldots,n:\ 
\E(\psi_i(X_i))<\infty \text{ and } \E\left[(\log(1+|X_i|^2))^{1+\varepsilon} \right] <\infty\text{ for some }\varepsilon>0.
\end{equation}
Then for any $\beta > 0$ 
\begin{align}
\label{eq:estimator-divergence} N^\beta \cdot \hN \Mskript^2(\bm{X}^{(1)},\ldots, \bm{X}^{(N)}) &\xrightarrow[N\to\infty]{a.e.} \infty& \text{ if $X_1,\ldots,X_n$ are dependent but $(n-1)$-independent,}\\
\label{eq:estimator-convergence}N \cdot \hN \Mskript^2(\bm{X}^{(1)},\ldots, \bm{X}^{(N)}) &\xrightarrow[N \to \infty]{d} Q& \quad \text{ if $X_1,\ldots,X_n$ are independent,}\\
\label{eq:totalestimator-divergence}\Mfrac{N^\beta \cdot \hN \overline{\Mskript}^2(\bm{X}^{(1)},\ldots, \bm{X}^{(N)})}{2^n-n-1}  &\xrightarrow[N\to\infty]{a.e.} \infty& \text{ if $X_1,\ldots,X_n$ are dependent,}\\
\label{eq:totalestimator-convergence}\Mfrac{N \cdot \hN \overline{\Mskript}^2(\bm{X}^{(1)},\ldots, \bm{X}^{(N)})}{2^n-n-1}  &\xrightarrow[N \to \infty]{d} \overline{Q} & \quad \text{ if $X_1,\ldots,X_n$ are independent,}
\end{align}
where $Q$ and $\overline{Q}$ are Gaussian quadratic forms with $\E Q = 1 = \E \overline{Q}.$
\end{theorem}
\begin{proof} Here we explain the main new ideas, the details are provided in Section \ref{sec:asymptoticsthm}. 

For the convergence in \eqref{eq:estimator-convergence} and \eqref{eq:totalestimator-convergence} exist at least two methods of proof: As in \cite{BoetKellSchi2019} the convergence of empirical characteristic functions can be used. For this step a slightly relaxed (but technical) version of the log moment condition (see Remark \ref{rem:logmom}) is necessary and sufficient, cf.\ \cite[Remark 4.6.b]{BoetKellSchi2019}. An alternative approach (Theorem \ref{thm:asymp-Vstat} in Section \ref{sec:appendix}) uses the theory of degenerate V-statistics, this requires moments of second order with respect to $\psi_i$, but no further condition. Thus, in particular, for bounded $\psi_i$ the latter removes the log moment condition. 

For $\beta = 1$ the divergence in \eqref{eq:estimator-divergence} and \eqref{eq:totalestimator-divergence} was proved in \cite{BoetKellSchi2019} under the condition \eqref{d1}, which ensures that the representations \eqref{eq:M} of the limits of sample (total) multivariance are well defined and finite. Using the characteristic function representation \eqref{def:multi} (which is always well defined, but possibly infinite) the divergence can be proved without \eqref{d1}, see Section \ref{sec:appendix} Lemma \ref{lem:liminfNM} ff.. Moreover, the arguments used therein work for any $\beta > 0.$
\end{proof}

\begin{remark} \label{rem:logmom} The log moment condition $\E\left[(\log(1+|X_i|^2))^{1+\varepsilon} \right]<\infty$ in \eqref{eq:mom1-log} can be slightly relaxed to \cite[Condition $(\star)$]{Csoe1985}. But the latter is practically infeasible, thus we opted for a comprehensible condition. Moreover the log moment condition is trivially satisfied, if $\psi_i$ satisfies a minimal growth, i.e., $\psi_i(x_i) \geq c \log(1+|x_i|^2)^{1+\varepsilon}$ for some $c,\varepsilon>0$. Also the condition \eqref{eq:psibounded} is stated here for clarity, in fact \eqref{eq:secondmom} is sufficient.
\end{remark} 

Note that in Theorem \ref{thm:convergence} the parameter $\beta$ was only considered in the dependent cases. In the case of independent random variables one obtains the following result.

\begin{corollary}[Strong consistency of $N^\beta$-scaled multivariance in the case of independence] \label{cor:betaM-consistency}
Let $X_i,$ $i=1,\ldots,n$ be independent random variables and let $\bm{X}^{(k)}, k=1,\ldots,N$ be independent copies of $\bm{X}=(X_1,\ldots,X_n)$. If either \eqref{eq:psibounded} or \eqref{eq:mom1-log} holds, then for any $\beta < 1$ 
\begin{align}
\label{eq:estimator-consistent}N^\beta \cdot \hN \Mskript^2(\bm{X}^{(1)},\ldots, \bm{X}^{(N)}) &\xrightarrow[N \to \infty]{a.e.} 0,\\
\label{eq:totalestimator-consistent}\Mfrac{N^\beta \cdot \hN \overline{\Mskript}^2(\bm{X}^{(1)},\ldots, \bm{X}^{(N)})}{2^n-n-1}  &\xrightarrow[N \to \infty]{a.e.} 0. 
\end{align}
\end{corollary}
\begin{proof}
The statements \eqref{eq:estimator-consistent} and \eqref{eq:totalestimator-consistent} are a direct consequence of \eqref{eq:estimator-convergence} and \eqref{eq:totalestimator-convergence}, if one considers convergence in probability instead of 'a.e.', see e.g.\ \cite[proof of Corollary 4.7]{BoetKellSchi2019} for the case $\beta= 0.$ 

For almost sure convergence one has to look at the proof(s) of \eqref{eq:estimator-convergence} and \eqref{eq:totalestimator-convergence}. Therein a key step is an application of the central limit theorem, which requires (in the given setting) exactly the factor $N$ for convergence (in distribution) to a standard normally distributed random variable. Using therein, for $N$ replaced by $N^\beta$ with $\beta < 1$, Marcinkiewicz's law of large numbers, e.g.\ \cite[Theorem 3.23]{Kall1997}, (or the law of the iterated logarithm) yields the limit 0 almost surely.
\end{proof}

The choice of $\psi_i$ is intertwined with the invariance properties (Propostions \ref{prop:inv} and \ref{prop:scaleinv}) and the moment conditions \eqref{eq:psibounded} and \eqref{eq:mom1-log}. For the population measures also condition \eqref{d1} and for the scaled measures also \eqref{eq:mom-1} and \eqref{eq:mom-n} have to be considered. In particular, it is possible to choose $\psi_i$ (or to transform the random variables) such that these conditions are satisfied regardless of the underlying distributions.

\begin{remark} \label{rem:choosepsi}
\begin{enumerate}
\item (Comments on choosing $\psi$) Based on Propositions \ref{prop:inv} and \ref{prop:scaleinv} the canonical choice for $\psi_i$ is $\psi_i(x_i)= |x_i|^{\alpha_i}$ with $\alpha_i \in (0,2)$, classically with $\alpha_i = 1$ (other choices might provide higher power in tests; a general $\alpha_i$ selection procedure is to our knowledge not yet available).

Nevertheless there are many other options for $\psi_i$, see \cite[Table 1]{BoetKellSchi2018} for various examples, and there are at least a few reasons why one might choose a $\psi_i$ which is not (a power of) the Euclidean distance:
\begin{enumerate}[i)]
\item For unbounded $\psi_i$ condition \eqref{eq:mom1-log} is required in Theorem \ref{thm:convergence}. If the existence of these moments is unknown for the underlying distribution the convergence results might not hold. Here the use of a slower growing or bounded $\psi_i$ is a safer approach, see Example \ref{ex:mom}.
\item The empirical size/power of the tests (details are given in Section \ref{sec:test}) can depend on the functions $\psi_i$ used, see Example \ref{ex:variousm}. Especially if some information on the dependence scale is known the parameter $\delta_i >0$ in $\psi_i(x_i) = 1- e^{-\delta_i|x_i|^{\alpha_i}}$ can be adapted accordingly, see \cite[Example 5.2]{BoetKellSchi2019} for an example using multivariance. Adaptive procedures for $\delta_i$ can be found in \cite{Guet2017}. 
\item A non-linear dependence of multivariance on sample distances might be desired, e.g.\ there might be application based reasons to use the Minkowski distance \cite[Section 2.4.4]{HanPeiKamb2011}. 
\end{enumerate}
An alternative approach to ensure the moment conditions is the following.
\item \label{rem:boundedrv} (Transformation to bounded random variables) 
Recall a basic result on independence: For $i=1,..,n$ let $X_i:\Omega \to \R^{d_i}$  be random variables and $f_i:\R^{d_i} \to D_i\subset  \R^{s_i}$ be measurable functions, then:
\begin{equation*}
X_i, i=1,\ldots,n \text{ are independent } \quad \Rightarrow \quad f_i(X_i), i=1,\ldots,n \text{ are independent.}
\end{equation*}
Moreover, if $d_i = s_i$ and $f_i$ are bijective then also the converse implication holds. Thus one way to ensure all moment conditions in Theorem \ref{thm:convergence} -- and preserve the (in)dependence -- is to transform the random variables by bounded (bounded $D_i$) bijective functions $f_i$. But beware that with this approach the multivariance is neither translation invariant nor homogeneous, cf.\ Example \ref{ex:mom}.
\end{enumerate}\end{remark}

\section{Comparison of multivariance to other dependence measures}\label{sec:compare}
In this section we compare multivariance to other dependence measures for random vectors $X_i\in \R^{d_i}.$ We only consider dependence measures which are closely related, in the sense that they are also based on characteristic functions or appear as special cases.
In the papers introducing and discussing these measures comparisons with further dependence measures can be found. 

Recall that multivariance (squared), $M^2(X_1,\ldots,X_n)$, is structurally of the form 
\begin{equation}\label{eq:mv-structure}
 \int \left|\E\left(\prod_{i=1}^{n} (\ee^{\ii X_it_i} - f_{X_i}(t_i))\right) \right|^2\, \rho(dt) = \E\left[ \prod_{i=1}^n \Psi_i(X_i,X_i')\right] 
\end{equation}
with $\Psi_i$ given in \eqref{eq:Psi} and $(X_1',\ldots,X_n')$ is an independent copy of $(X_1,\ldots,X_n)$.

\subsection{Classical covariance, Pearson's correlation and the RV coefficient are limiting cases of multivariance} \label{sec:Cov-RV-Pearson}
Let $n=2$ and $\psi_i(x_i) = |x_i|^2$. Note that $|.|^2$ is not characterizing in the sense of Proposition \ref{prop:fullsupport}. It actually does not correspond to a L\'evy measure, cf.\ \cite[Table 1]{BoetKellSchi2018}. Thus the characteristic function representation (left hand side of \eqref{eq:mv-structure}) does not hold and a value 0 of the right hand side does not characterize independence. Nevertheless, $|.|^2$  is a continuous negative definite function and it is the limit for $\alpha_i \uparrow 2$ of $|.|^{\alpha_i}$ which are valid for multivariance. Moreover, the right hand side of \eqref{eq:mv-structure} is also for $|.|^2$  well defined, and it corresponds to classical linear dependence measures: Hereto denote by $X_{i,k}$ the components of the vectors $X_i$, i.e., $X_i = (X_{i,1},\ldots,X_{i,d_i})$ where $X_{i,k} \in \R$. By direct (but extensive calculation) the expectation representation in \eqref{eq:mv-structure} of $M^2(X_1,X_2)$ with $\psi_i(x_i) = |x_i|^2 = \sum_{k=1}^{d_i} x_{i,k}^2$ simplifies to
\begin{equation}
\sum_{k=1}^{d_1} \sum_{l=1}^{d_2} (2\Cov(X_{1,k},X_{2,l}))^2.
\end{equation}
Especially for $d_1 = d_2 = 1$ the (absolute value of) classical covariance is recovered.
Note that for $n=2$ and $\psi_i(.) = |.|^2$ the scaling constants \eqref{eq:normalizedm}, \eqref{eq:multicorr} and \eqref{eq:Mcor} become $2\Var(X_i)$, thus normalized multivariance coincides in this setting with both multicorrelations and with the absolute value of classical correlation. For arbitrary $d_1$ and $d_2$ the multicorrelations (squared) also coincide with the extension of correlation to random vectors developed in \cite{Esco1973}. The corresponding sample versions $\hN \Mskript, \hN \Rskript$ and $\hN \Mcor$ coincide for $d_1=d_2=1$ with (the absolute value of) Pearson's correlation coefficient and $\hN \Rskript^2$ and $\hN \Mcor^2$ coincide for arbitrary $d_1$ and $d_2$ with the RV coefficient of \cite{RobeEsco1976} (see also \cite{JossHolm2016}).

Note that also for $n>2$ the right hand side of \eqref{eq:mv-structure} with $\psi_i(x_i) = |x_i|^2$ is a well defined expression, which can be understood as an extension of covariance, Pearson's correlation and the RV coefficient to more than two random vectors.

\subsection{Multivariance unifies distance covariance and HSIC} \label{sec:unify-dcov-hsic}
In the case of two random variables (that is, $n=2$) multivariance coincides with generalized distance covariance \cite{BoetKellSchi2018} and the following (simplified) representations hold (using \cite[Eq.\ (30)]{BoetKellSchi2018}, direct calculations, \cite[Eq.\ (3.2)]{JossHolm2016} and the notation $\overline{\psi} = 1-\psi$)
\begin{align}
 M^2(X_1,X_2) &= \iint |f_{(X_1,X_2)}(t_1,t_2) - f_{X_1}(t_1)f_{X_2}(t_2)|^2 \, \rho_1(dt_1)\rho_2(dt_2) = \E\left[ \prod_{i=1}^2 \Psi_i(X_i,X_i')\right]\\
\label{eq:hsic} &= \E\left[\prod_{i=1}^2 (\overline{\psi_i}(X_i-X_i')) \right] - 2 \E\left[ \prod_{i=1}^2 \E[\overline{\psi_i}(X_i-X_i')\mid X_i]\right] +  \prod_{i=1}^2 \E\left[\overline{\psi_i}(X_i-X_i')\right]\\
\label{eq:dcov} &= \E\left[\prod_{i=1}^2 \psi_i(X_i-X_i') \right] - 2 \E\left[ \prod_{i=1}^2 \E[\psi_i(X_i-X_i')\mid X_i]\right] +  \prod_{i=1}^2 \E\left[\psi_i(X_i-X_i')\right]\\
&= \Cov\left(\psi_1(X_1-X_1'), \psi_2(X_2-X_2')\right) - 2 \Cov\left( \psi_1(X_1-X_1'), \psi_2(X_2-X_2'')\right). 
\end{align}
The last line is included to emphasize that further interesting representations exist - this one actually provides a characterization of independence using (the classical linear dependence measure) covariance. Other equivalent representations are Brownian distance covariance \cite[]{SzekRizz2009} (for $\psi(.) = |.|$) and its generalization Gaussian distance covariance \cite[Section 7]{BoetKellSchi2018}.

Note that \eqref{eq:dcov} is for $\psi_i(x_i) = |x_i|^{\alpha_i}$ distance covariance \cite[]{SzekRizzBaki2007} and \eqref{eq:hsic} is for $\psi_i(x_i) = 1 - e^{-\delta_i\tilde \psi_i(x)}$ (where $\tilde\psi_i$ can be any continuous real-valued negative definite function, e.g.\ $|.|^{\alpha_i}$, and $\delta_i>0$) the Hilbert Schmidt Independence Criterion (HSIC, \cite{GretFukuTeoSongScho2008}) with kernel $k_i(x,y) = e^{-\delta_i\tilde \psi_i(x-y)}$.\footnote{HSIC (and dHSIC in \cite{PfisBuehSchoPete2017}) require bounded, continuous, symmetric, positive definite kernels $k_i$. If $k_i$ is additionally translation invariant, then $k_i(x_i,x_i') = k_i(x_i-x_i',0) =: \phi(x_i-x_i')$ and $\phi$ is a continuous positive definite function. For the non translation invariant case see Section \ref{sec:extensions}. Moreover, note that we assume here $\phi(0)=1$ to avoid distracting constants in the presentation. HSIC and dHSIC additionally require that the kernel is characterizing, which is by Proposition \ref{prop:fullsupport} equivalent to the full support property of $\rho.$} For the latter just note that for any continuous positive definite function $\phi$ the function $\phi(0)-\phi$ is continuous negative definite (cf.\ \cite[Corollary 3.6.10]{Jaco2001}), i.e., it fits into the framework of multivariance. The equivalence of kernel based approaches and distance based approaches was noted in \cite{SejdSripGretFuku2013}, see also \cite{ShenVoge2018} for a recent discussion. But note that the approach in \cite{SejdSripGretFuku2013} to the correspondence of kernels and distance functions only works for the case $n=2$, whereas the above correspondence also extends to the multivariate setting.

In other words, in the case $n=2$ multivariance with bounded measures $\rho_i$ coincides with HSIC and special cases of unbounded $\rho_i$ yield distance covariance. Therefore, in general, multivariance is an extension of these measures to more than two variables. But note that there is also at least one alternative extension as we will discuss in the next section.

As discussed in Remark \ref{rem:choosepsi}, the cases with bounded measures have the advantage that most moment conditions are trivially satisfied and that in the case of HSIC the parameters $\delta_i$ provide a somehow natural bandwidth selection parameter. In contrast, using unbounded measures $\rho_i$ corresponding to $|.|^{\alpha_i}$ provide (scaled) measures with superior invariance properties (Propositions \ref{prop:inv} and \ref{prop:scaleinv}). Note, that also in this case the parameters $\alpha_i$ offer some variability. 

As a side remark, note that by the above it is straight forward that multivariance with $\tilde \psi_i$ is the derivative (in the bandwidth parameter at $\delta_i=0$) of multivariance corresponding to $1 - e^{-\delta_i\tilde \psi_i(x)}$, this relation of distance covariance and HSIC was noted in \cite{BiloNang2017}. Incidentally, it is also the key for relating L\'evy processes to their generators, e.g.\ see the introduction of \cite{BoetSchiWang2013}.

Finally note that also the other measures discussed in the next section reduce for the case $n=2$ to the above setting, thus they are included (or closely related as \cite{JinMatt2018}, which considers a joint measure $\rho$ without product structure).

\subsection{Independence of more than two random vectors} \label{sec:compare-3}

As a consequence of Theorem \ref{thm:indep} the multivariances of all subfamilies of the variables $X_1,\ldots,X_n$ characterize jointly their independence. In fact, this was suggested in \cite{BiloNang2017} as an approach to independence via multiple testing, i.e., via computing the p-value for each of these $2^n-n-1$ multivariances separately. The approach is complementary to the global test using total multivariance.

In \cite{BiloNang2017} multivariance is considered in disguise: expanding the integrand of \eqref{eq:mv-structure} and using the linearity of the expectation yields $\E\left(\prod_{i=1}^{n} (\ee^{\ii X_it_i} - f_{X_i}(t_i))\right) = \sum_{S\subset\{1,\ldots n\}} \E(\prod_{i\in S} (\ee^{\ii X_it_i}) \prod_{i \in S^c}  (- f_{X_i}(t_i))$. This representation of the product is also called M\"obius transformation of the characteristic functions. Without the characteristic function representation (with $\psi_i$ based on kernels $k_i$) the multivariance of 3 random variables appeared before under the name ``(complete) Lancaster interaction'' in \cite{SejdGretBerg2013}. 

Other popular multivariate dependence measures based on characteristic functions are of the following form, which is here stated using our setting (with  the notation $\overline{\psi} = 1-\psi$ and $\rho_i(\R^{d_i}) = 1$; to reformulate it for positive definite kernels use the correspondence provided in Section \ref{sec:unify-dcov-hsic}): 
\begin{subequations} \label{eq:kan-structure}
\begin{align} 
\label{eq:kan-structurecf} \int \Bigg|\E\left[\prod_{i=1}^{n}  \ee^{\ii X_it_i}\right] - \prod_{i=1}^{n} f_{X_i}(t_i) \Bigg|^2\,& \rho(dt) \\
\label{eq:kan-structureE}= & \ \E\left[\prod_{i=1}^n (\overline{\psi_i}(X_i-X_i')) \right] - 2 \E\left[ \prod_{i=1}^n \E[\overline{\psi_i}(X_i-X_i')\mid X_i]\right] +  \prod_{i=1}^n \E\left[\overline{\psi_i}(X_i-X_i')\right].
\end{align}
\end{subequations}
Such dependence measures go back at least to \cite[(1.3)]{Kank1995}. It is important to note that the equality in \eqref{eq:kan-structure} does not hold in general for unbounded measures $\rho_i$, e.g.\ for $n=3$, $X_1,X_2$ dependent (satisfying \eqref{d1}) and $X_3$ constant the term \eqref{eq:kan-structurecf} is infinite but \eqref{eq:kan-structureE} is finite. Nevertheless, dependence measures of type \eqref{eq:kan-structure} for $\rho = \otimes_{i=1}^n\rho_i$ with bounded and unbounded $\rho_i$ were recently discussed in \cite{FanMichPeneSalo2017} (in the unbounded case \cite[Lemma 1a]{FanMichPeneSalo2017} only provides a rather complicated sample version, which actually corresponds to \eqref{eq:kan-structureE}, a proof can be found in Section \ref{sec:fanstruct}), for finite $\rho_i$ representation \eqref{eq:kan-structure} corresponds to the also recently introduced measure dHSIC of \cite{PfisBuehSchoPete2017} and for an unbounded (joint measure) $\rho$ associated to $\psi(.)=|.|$ it was considered in \cite{JinMatt2018} (in this case \eqref{eq:kan-structureE} has a slightly different form).

The above illustrates that also for measures derived via \eqref{eq:kan-structure} various approaches can be unified using the framework of continuous negative definite functions and L\'evy measures.

To compare \eqref{eq:kan-structure} with multivariance, note that in \cite[Section 3.5]{BoetKellSchi2019} it was shown that for any given multivariance there exist special kernels (beyond the restrictions of the above papers) which turn \eqref{eq:kan-structureE} into multivariance. With the usual kernels the following holds:  $ \E\left(\prod_{i=1}^{n} (\ee^{\ii X_it_i} - f_{X_i}(t_i))\right) = \E\left[\prod_{i=1}^{n}  \ee^{\ii X_it_i}\right] - \prod_{i=1}^{n} f_{X_i}(t_i)$ if the given random variables are $(n-1)$-independent \cite[Corollary 3.3]{BoetKellSchi2018}. Thus the left hand sides of \eqref{eq:mv-structure} and \eqref{eq:kan-structure} coincide in the case of $(n-1)$-independence. Without $(n-1)$-independence multivariance does not characterize independence, but total multivariance $\overline M (X_1,\ldots,X_n)$, given by 
\begin{equation}\label{eq:tm-struct}
\sum_{\substack{S\subset\{1,\dots,n\}\\ |S|>1}} \int \left|\E\left(\prod_{i\in S} (\ee^{\ii X_it_i} - f_{X_i}(t_i))\right) \right|^2\, \underset{i\in S}{\otimes} \rho_i(dt_i) = \E\left(\prod_{i=1}^n (\Psi_i(X_i,X_i') + 1)\right) - 1,
\end{equation}
does characterize independence.

The approach \eqref{eq:kan-structure} and total multivariance \eqref{eq:tm-struct} require similar moment conditions\footnote{Based on the method of proof and based on the focus of the papers (sample or population versions; bounded or unbounded $\psi_i$) the stated conditions differ. But it seems a reasonable guess that these can be unified to those of multivariance, cf.\ the discussion in the proof of Theorem \ref{thm:convergence}.} (the variant in \cite{JinMatt2018} requires a joint first moment) and the computational complexity of the sample versions is similar (the variant in \cite{JinMatt2018} has a higher complexity, but they also provide an approximate estimator with the same complexity). Total multivariance needs one product of doubly centered distance matrices whereas \eqref{eq:kan-structureE} needs three products of different distance matrices (which actually coincide with those used for the double centering). Nevertheless, both approaches differ: In the Section \ref{sec:diffHSICtm} we calculate explicitly the difference of the population measures for the case $n=3$, indicating that it is by no means theoretically obvious which approach might be more advantageous. Here certainly further investigations are required. A practical difference is the fact that the current implementation of dHSIC \cite[]{PfisPete2019} requires $N>2n$, for multivariance there is no such an explicit restriction.

Generally, papers on dependence measures differ not only in their measures, but also in their methods of testing. For the approach \eqref{eq:kan-structure} various methods have been proposed, of which the resampling method seems most popular. For multivariance we introduce the resampling method in Section \ref{sec:test}. But there are also further (and faster) methods available for multivariance: Distribution free tests are used in \cite{BoetKellSchi2019} (see Theorem \ref{thm:testqf}) and in \cite{BersBoet2018v2} tests based on moments of the finite sample or limit distribution and/or using eigenvalues of the associated Gaussian process are developed, see also \cite{Guet2017}.

\subsection{Pairwise independence}

In Section \ref{sec:mmulti} we introduce $m$-multivariance. In particular, $2$-multivariance provides a global test for pairwise independence without any condition (when using bounded $\psi_i$) or under the mild moment condition \eqref{eq:mom1-log}, see Test \ref{testD}. A related approach to pairwise independence using distance covariance was developed in \cite{YaoZhanShao2017}, but in contrast it required assumptions which are necessary for applications of a (generalized) central limit theorem. The methods are compared in Example \ref{ex:YaoZhanShao2017}.

\subsection{Generalizations} \label{sec:extensions}

The setting of HSIC and also extensions of distance covariance are applicable to more general spaces than $\R^d$. In this settings the representation via characteristic functions and the characterization of independence (might) fail. Nevertheless, the representations given in \eqref{eq:M} can canonically be extend to negative definite kernels $n(x_i,x_i')$ replacing $\psi(x_i-x_i').$ Thus it seems a natural guess that the key properties required for testing can be recovered in this setting, but to our knowledge this has not been studied yet.

For distance covariance exist also further modifications, like the affinely invariant distance correlation in \cite{Edel2015}. Also this extension seems possible for multivariance. It is only defined for random vectors with non singular covariance matrices and in this setting it would be a candidate to satisfy the set of axioms given in the next section \cite[Example 3]{MoriSzek2018}.

\subsection{Axiomatic classification of dependence measures} \label{sec:axioms}

R\'enyi \cite{Reny1959} proposed in 1959 a set of axioms which a dependence measure should satisfy. These have been challenged over the years, most recently e.g.\ in \cite{MoriSzek2018}. They propose ``four simple axioms'' which a dependence measure $d$ should satisfy, and distance correlation is called the ``simplest and most appealing'' measure which satisfies these axioms. All axioms were proposed for pairwise comparisons of random variables or vectors. We present here a multivariate extension to $n$ non-constant random vectors (constants are removed to avoid technical difficulties, cf.\ \cite{MoriSzek2018}):
\begin{itemize}
\item[(A1)] \textit{characterization of independence}: $d(X_1,\ldots,X_n) = 0$ if and only if $X_i$ are independent.
\item[(A2)] \textit{invariance}: $d(X_1,\ldots,X_n) = d(S_1(X_1),\ldots,S_n(X_n))$ for all similarity transforms\footnote{A \textbf{similarity transform} is any combination of translations, rotations, and reflections and non zero scalings (using the same scaling factor for all components of a vector), cf.\ \cite{MoriSzek2018}.} $S_i$.
\item[(A3)] \textit{reference value}: $d(X_1,\ldots,X_n)= 1$ if $X_1,\ldots,X_n$ are related by similarity transforms (see \eqref{eq:similarityrel} for details).
\item[(A4)] \textit{continuity}: $d(X_1^{(k)},\ldots,X_n^{(k)}) \xrightarrow{k\to \infty} d(X_1,\ldots,X_n)$, if $(X_1^{(k)},\ldots,X_n^{(k)})$ converges in distribution to $(X_1,\ldots,X_n)$ (under a uniform moment condition, which ensures the finiteness of the measures).
\end{itemize}
Note that \cite{MoriSzek2018} uses a further common axiom --- \textit{normalization}: $d(...) \in[0,1]$ --- which was only indirectly assumed and (A3) was stronger: it contained ``if and only if'' with a seemingly more restrictive relation which actually forced explicitly the dimensions of the random vectors to be identical. Note that in the related (original) axiom \cite[Axiom E]{Reny1959} also only the ``if'' part was required and a footnote explicitly advised against strengthening it.

In the setting of multivariance we say that random variables $X_i$ and $X_k$ are related by similarity transforms $S_i$ and $S_k$ if
\begin{equation}\label{eq:similarityrel}
\psi_i(S_i(X_i)-S_i(X_i')) = \psi_k(S_k(X_k) - S_k(X_k')).
\end{equation}

A prerequisite for the continuity (A4) is the finiteness of the measure $d$, cf.\ \cite{MoriSzek2018}. Thus all considerations for (normalized) multivariance are under the moment condition \eqref{d1} and for the multicorrelation we have to assume \eqref{eq:mom-n}.   
Based on Propositions \ref{prop:inv} and \ref{prop:scaleinv} the invariance with respect to similarity transforms holds for $\psi(x) = |x|^\alpha$, and it seems (cf.\ \cite[Section 5]{BoetKellSchi2018}) that for other unbounded and all bounded $\psi$ the invariance fails. Therefore we only consider $\psi(x) = |x|^\alpha$. Table \ref{tab:axioms} indicates which axioms are satisfied by the measures, all properties follow by direct calculations (the continuity uses the dominated convergence theorem; for the normalization a generalized H\"older inequality is used, see also \cite[Proposition 4.13]{BoetKellSchi2019}). For multicorrelation the properties vary as the number of variables is even or odd, and $\mathcal{R}$ yields always a measure with values in $[0,1]$ whereas $\Mcor$ yields always the reference value 1 for variables related by similarity transforms. Note that for a multivariate normal distribution the value of total distance multivariance is (for the special case $\psi(x)=|x|$) linked to its correlation by \cite[Proposition 2]{ChakZhan2019}.

By Table \ref{tab:axioms} the four axioms are simultaneously satisfied by $\overline{\Mcor}$. But recall that $\overline{\Rskript}$ and $\overline{\Mcor}$ lack efficient sample versions. In the sample setting also $N\cdot \hN\Mskript^2$ and $N\cdot \hN\overline\Mskript^2$ provide statistical interpretable values (indirectly: via the corresponding p-value; yielding also a rough direct interpretation: they are positive and their expectation is 1 for independent random variables. Thus values much larger than one hint at dependence). Moreover, normalized multivariance requires only the moment condition \eqref{d1} whereas multicorrelation requires the more restrictive condition \eqref{eq:mom-n}. Finally, note that in the case $n=2$ the multicorrelations coincide. Thus, in particular, $\Mcor_2$ (defined in Section \ref{sec:mmulti}) provides a measure with efficient sample estimator. For this measure a value of 0 only characterizes pairwise independence, but the value 1 occurs if and only if the random variables are related by similarity transforms.

A first discussion of the behavior of (total) multivariance when one enlarges the family of random variables can be found in \cite[Proposition 3.7, Remark 3.8]{BoetKellSchi2019}, which translates directly to multicorrelation. 

\newcommand{\yes}{\checkmark}
\newcommand{\no}{---}

\begin{table}
\caption{Dependence measure axioms which are satisfied by (variants) of (total) multi\-variance for $\psi_i(x_i) = |x_i|^{\alpha_i}$ with $\alpha_i \in (0,2)$.}
\centering
\begin{tabular}{|l|c|c|c|c|c|}
	\hline
	 \phantom{multiva} axioms &      (A1)       &    (A2)    &      (A3)       &    (A4)    &               \\
	             & characterization of independence & invariance &  reference value & continuity &  normalization\\

 \hline 
	multivariance                             &                 &            &                 &            &               \\
	$M$                                       &      $n=2$      &    \no     &       \no       &    \yes    &      \no      \\
	$\overline{M}$                            &      \yes       &    \no     &       \no       &    \yes    &      \no      \\ \hline
	normalized                  &                 &            &                 &            &               \\
multivariance & & & & & \\
	$\Mskript$                                &      $n=2$      &    \yes    &       \no       &    \yes    &      \no      \\
	$\overline{\Mskript}$                     &      \yes       &    \yes    &       \no       &    \yes    &      \no      \\ \hline
	multicorrelation                          &                 &            &                 &            &               \\
	$\mathcal{R}$                             &      $n=2$      &    \yes    &    $n$ even     &    \yes    &     \yes      \\
	$\overline{\mathcal{R}}$                  &      \yes       &    \yes    &     $n= 2$      &    \yes    &     \yes      \\
	$\Mcor$                                   &      $n=2$      &    \yes    &      \yes       &    \yes    &   $n$ even    \\
	$\overline{\Mcor}$                        &      \yes       &    \yes    &      \yes       &    \yes    &    $n = 2$    \\ \hline
	2-multivariance   & char. of pairwise independence &            &  (A3) and iff   &            &               \\
(Section \ref{sec:mmulti}) & & & & & \\
 	$M_2$                                     &      \yes       &    \no     &       \no       &    \yes    &      \no      \\

	$\Mskript_2$                              &      \yes       &    \yes    &       \no       &    \yes    &      \no      \\
	$ \Mcor_2\ (= \mathcal{R}_2 )$                 &      \yes       &    \yes    &      \yes       &    \yes    &     \yes      \\ \hline
\end{tabular}
\label{tab:axioms}
\end{table}

\section{Testing independence using multivariance} \label{sec:test}

In this section we extend the discussion of \cite[Section 4.5]{BoetKellSchi2019}. We use the notation of Section \ref{sec:defs}, in particular $\bm{x}^{(k)}= (x_1^{(k)},\ldots,x_n^{(k)})$ are samples of independent copies of $(X_1,\ldots,X_n).$ Based on Theorem \ref{thm:convergence}, and recalling the fact that constant random variables are always independent, the following structure of a test for independence is obvious.

\begin{test}[Test for $n$-independence, given $(n-1)$-independence] \label{testA} 
Let $\psi_i$ be bounded or \eqref{eq:mom1-log} be satisfied. Then a test for independence is given by: Reject $n$-independence if $X_1,\ldots,X_n$ are $(n-1)$-independent and 
\begin{equation} \label{eq:testA}
N \cdot \hN \Mskript^2(\bm{x}^{(1)},\ldots,\bm{x}^{(N)}) > R.
\end{equation}
The value $R$ will be discussed below.
\end{test}
\begin{remark} \label{rem:testA}
Note that also without the assumption of $(n-1)$-independence \eqref{eq:testA} provides a test for independence for which the type I error can be controlled by the choice of $R$, since the distribution of the test statistic under the hypothesis of independence is known, see \eqref{eq:estimator-convergence}. But in this case it is unknown if the test statistic diverges if the hypothesis does not hold. Thus one can not control the Type II error and it will not be consistent against all alternatives (regardless of the satisfied moment conditions). A trivial example hereto would be the case where one random variable is constant, and thus the test statistic is always 0. But note that with the assumption of $(n-1)$-independence this problem does not appear, since the $(n-1)$-independence implies (given that at least one random variable is constant) that the random variables are independent.
\end{remark}
Analogous to Test \ref{testA}, using total multivariance instead of multivariance, one gets the test for independence.
\begin{test}[Test for ($n$-)independence] \label{testB} Let $\psi_i$ be bounded or \eqref{eq:mom1-log} be satisfied. Then a test for independence is given by: Reject independence if
\begin{equation} \label{eq:testtotal}
\Mfrac{N \cdot \hN \overline{\Mskript}^2(\bm{x}^{(1)},\ldots,\bm{x}^{(N)})}{2^n-n-1} > R.
\end{equation}
\end{test}

To get a test with significance level $\alpha \in (0,1)$  the natural choice for $R$ in \eqref{eq:testA} and \eqref{eq:testtotal} is the $(1-\alpha)$-quantile of the (limiting) distributions of the test statistics under $H_0$, i.e., assuming that the $X_i$ are independent. To find this distribution explicitly or at least to have good estimates is non trivial, see \cite{BersBoet2018v2} for an extensive discussion. As a starting point, one can follow \cite[Theorem 6]{SzekRizzBaki2007} where a general estimate for quadratic forms of Gaussian random variables given in \cite{SzekBaki2003} is used to construct a test for independence based on distance covariance. In our setting this directly yields the following result.

\begin{theorem}[Rejection level for the distribution-free tests] \label{thm:testqf}
Let $\alpha \in (0,0.215]$. Then Test \ref{testA} and \ref{testB} with
\begin{equation} \label{def:R}
R := F_{\chi_1^2}^{-1}(1-\alpha)
\end{equation}
are (conservative) tests with significance level $\alpha$. Here $F_{\chi_1^2}$ is the distribution function of the Chi-squared distribution with one degree of freedom.
\end{theorem}
In the case of univariate Bernoulli random variables the significance level $\alpha$ is achieved (in the limit) by Test \ref{testA} with $R$ given in \eqref{def:R}, see \cite[Remark 4.27]{BersBoet2018v2}. But for other cases it might be very conservative, e.g.\ Example \ref{ex:variousm} (Figure \ref{fig:mdist-marginal}). Recall that total multivariance is the sum of $2^n-n-1$ distance multivariances (this is the number of summands in \eqref{eq:totaldm}). Thus one distance multivariance with a large value might be averaged out by many small summands, see Example \ref{ex:averaging}. Hereto $m$-multivariance (which will be introduced in the next section) provides an intermediate remedy. It is also the sum of multivariances, but it has less summands. Thus the 'averaging out' (also known as 'statistical curse of dimension') will be still present but less dramatic.

Note that $R$ in Theorem \ref{thm:testqf} is provided by a general estimate for quadratic forms. It yields in general conservative tests, since it does not consider the specific underlying (marginal) distributions. Less conservative tests can be constructed if the distributions are known or by estimating these distributions. The latter can be done by a resampling approach or by a spectral approach, similarly to the case of distance covariance (see \cite[Section 7.3.]{SejdSripGretFuku2013}). Methods related to the spectral approach are developed in \cite{BersBoet2018v2}.

In the following the resampling approach for $\mathcal{M}$ is introduced. The procedure is certainly standard to experts, never the less it seems important to recall it (to avoid ambiguity): Suppose we are given i.i.d.\  samples\footnote{Here we use a common abuse of terminology: An \textit{independent sample} is a sample based on independent random variables. Analogously, an i.i.d.\  (independent and identical distributed) sample, is a sample of i.i.d.\  random variables. Moreover, note that here the random variables are in general random vectors with possibly dependent components.} $\bm{x}^{(1)},\ldots,\bm{x}^{(N)}$ with unknown dependence, i.e., for each $i$ the dependence of the components $x^{(i)}_1,\ldots,x^{(i)}_n$ is unknown. Now, resampling each component separately yields (almost) independent components. Thus Test \ref{testA} (respectively Test \ref{testB} with $\overline{\Mskript}$) becomes a \textbf{resampling test} (resampling without replacement / permutation test) with $L \in \N$ replications using the rejection level $R$ given by
\begin{equation}\label{eq:Rresampling}
R_{\text{rs}} := Q_{1-\alpha}\left(\left\{N\cdot \hN \Mskript^2\left(x_1^{(p^{(l)}_1(i))},\ldots, x_n^{(p^{(l)}_n(i))},\ i=1,\ldots,N \right), \quad l = 1,\ldots,L \right\}\right)
\end{equation}
where each $p^{(l)}_k(1),\ldots, p^{(l)}_k(N)$ is a random permutation of $1,\ldots,N$ (and these are i.i.d.\ for $k=1,\ldots,n$ and $l = 1,\ldots,L$) and $\bm{x}^{(i)} = (x^{(i)}_1,\ldots,x^{(i)}_n)$ are the samples given for the test. Here $Q_{1-\alpha}(S)$ denotes the empirical $(1-\alpha)$-quantile of the samples in the set $S$. Instead of random permutations one could allow $p^{(l)}_k(1),\ldots, p^{(l)}_k(N)$ to be any sample of $1,\ldots,N$, this would also be a resampling test (resampling with replacement / bootstrap test), but note that the permutation test can be implemented more efficiently.  Similarly, Test \ref{testA} (respectively Test \ref{testB} with $\overline{\Mskript}$) becomes a \textbf{Monte Carlo test} with $L \in N$ replications using 
\begin{equation}\label{eq:RMC}
R_{\text{MC}} := Q_{1-\alpha}\left(\left\{N \cdot \hN \Mskript^2(x_1^{(i,l)},\ldots, x_n^{(i,l)},\ i=1,\ldots,N ), \quad l = 1,\ldots,L \right\}\right)
\end{equation}
where $x_k^{(i,l)}, k=1,\ldots,n, i=1,\ldots,N, l=1,\ldots,L$ are independent samples and for each fixed $k$ the $x_k^{(i,l)},  i=1,\ldots,N, l=1,\ldots,L$ are i.i.d.\  samples of $X_k$.

\begin{remark}
In \cite[Section 3.2]{PfisBuehSchoPete2017} two related resampling tests are introduced for dHSIC. But note that they use slightly different terminology, i.e., therein the 
'permutation test' considers samples as in \eqref{eq:Rresampling} but instead of random permutations \textit{all} permutations are considered. For the 'bootstrap test' they use \textit{all} resamplings of the sample distribution of each variable. This yields $L = (N!)^n$ and $L = N^{Nd}$, respectively. Which is infeasible even for relatively small $N$, thus in \cite[Section 4.2]{PfisBuehSchoPete2017} they also use randomly selected samples instead of \textit{all} samples, and they call the resulting estimators 'Monte-Carlo approximations' of the estimators. 
\end{remark}

\section{\lowercase{$m$}-multivariance} \label{sec:mmulti}
Pairwise independence is the prime requirement for various fundamental tools in stochastics, e.g.\ the classical law of large numbers. Especially when working with many variables ($n$ large) a multiple testing approach might not be feasible. Thus a global test for pairwise independence has many applications, see also the motivation in \cite{YaoZhanShao2017}. Here we construct such a test, together with further generalizations which allow the successive testing of 2-independence, 3-independence, etc.

Define for $m\in\{2,\ldots,n\}$ the \textbf{\boldmath$m$-multivariance} $M_{m}$ by
\begin{equation} \label{def:mmulti}
M_{m,\rho}^2(X_1,\ldots,X_n) :=  \sum_{1\leq i_1 < \ldots < i_m \leq n} M_{\otimes_{k=1}^m \rho_{i_k}}^2(X_{i_1},\ldots,X_{i_m}).
\end{equation}

Instantly Theorem \ref{thm:indep} yields the following characterization.

\begin{proposition}[Characterization of $m$-independence] \label{prop:mindep}
For random variables\\ $X_1,\ldots, X_n$ the following are equivalent:
\begin{enumerate}[i)]
\item $X_1,\ldots,X_n$ are $m$-independent,
\item $M_{m,\rho}(X_1,\ldots,X_n) = 0$ and $X_1,\ldots,X_n$ are $(m-1)$-independent.
\end{enumerate}
In particular, $M_2(X_1,\ldots,X_n) = 0$ characterizes pairwise independence.
\end{proposition}

As in the case of multivariance, using \eqref{def:smulti}, a strongly consistent estimator for $M_m$ is the \textbf{sample $\mathbf{m}$-multivariance}
\begin{align} \label{eq:mindepempirical}
\hN M_m(\bm{x}^{(1)},\ldots,\bm{x}^{(N)}) &= \sqrt{\sum_{1\leq i_1 < \ldots < i_m \leq n} \frac{1}{N^2} \sum_{j,k=1}^N (A_{i_1})_{jk} \cdot \ldots \cdot (A_{i_m})_{jk} }.
\end{align}
Analogous to the case of normalized (total) multivariance   the \textbf{normalized sample $\mathbf{m}$-multivariance} $\hN\Mskript_m$ is given by
\begin{equation} \label{eq:nsm-multi}
\hN \Mskript_m^2(\bm{x}^{(1)},\ldots,\bm{x}^{(N)}) = \genfrac(){0pt}{}{n}{m}^{-1}\sum_{1\leq i_1 < \ldots < i_m \leq n} \frac{1}{N^2} \sum_{j,k=1}^N (\Askript_{i_1})_{jk} \cdot \ldots \cdot (\Askript_{i_m})_{jk},
\end{equation}
where $\Askript_i$ are the normalized matrices defined in \eqref{eq:normalizedm}. For (sample) $m$-multivariance the invariance properties (Propositions \ref{prop:inv} and \ref{prop:scaleinv}) hold analogously. To ensure that the expectation representation of $m$-multivariance (analogous to \eqref{eq:M}) is finite the following condition (which is weaker than \eqref{d1}) is required:
\begin{align}
 \text{\textbf{finite joint $\psi$-moments for families of size $m$}: } 
\label{d1m}\text{ for all $S\subset\{1,\ldots,n\} $ with $|S| \leq m$: } \E \left( \prod_{i\in S} \psi_i(X_{i})\right) < \infty.
\end{align}

Note that the sum $\sum_{1\leq i_1 < \ldots < i_m \leq n}$ has $ \genfrac(){0pt}{}{n}{m} $ summands, which might be a lot to compute. These sums can be simplified using the multinomial theorem, $(A_1 + \ldots + A_n)^m = \sum_{k_1 + \ldots +k_n = m} \frac{m!}{k_1!\cdot\ldots\cdot k_n!} A_1^{k_1}\cdot\ldots\cdot A_n^{k_n}.$ In particular, for $m=2,3$ the following expressions of sample $m$-multivariance are easier to evaluate (analogous representations hold for the normalized sample $m$-multivariance): 
\begin{align}
\label{def:2mmulti} 
\hN M_2(\bm{x}^{(1)},\ldots,\bm{x}^{(N)}) &= \sqrt{ \frac{1}{2} \frac{1}{N^2}\sum_{k,l=1}^N \left( \left((A_1 + \ldots + A_n)_{kl}\right)^2 - \sum_{i=1}^n \left((A_i)_{kl}\right)^2 \right) },\\
\label{def:3mmulti}
\hN M_3(\bm{x}^{(1)},\ldots,\bm{x}^{(N)})
& = \sqrt{ \frac{1}{3} \frac{1}{N^2}\sum_{k,l=1}^N \left( \left(\left(\sum_{i=1}^n A_i\right)_{kl}\right)^3 - 3 \left(\sum_{i=1}^n A_i\right)_{kl} \sum_{i=1}^n \left((A_i)_{kl}\right)^2 + 2 \sum_{i=1}^n \left((A_i)_{kl}\right)^3 \right) }.
\end{align}
Thus at least for small $m$ these estimators are easy to compute and -- analogous to the case of (total) multivariance -- these can be used to test $m$-independence by the next results.

\begin{theorem}(Asymptotics of sample $m$-multivariance)\label{thm:convergencem} Let $X_i, i=1,\ldots,n$ be non-constant random variables and let $\bm{X}^{(k)}, k=1,\ldots,N$ be independent copies of $(X_1,\ldots,X_n)$. If either the $\psi_i$ are bounded or \eqref{eq:mom1-log} holds, then for $m \leq n$
\begin{align}
\label{eq:convergencem}\Mfrac{N \cdot \hN \Mskript_m^2(\bm{X}^{(1)},\ldots, \bm{X}^{(N)})}{{n \choose m}}  &\xrightarrow[N \to \infty]{d} Q  &\text{ if $X_1,\ldots,X_n$ are $m$-independent,}\\
\label{eq:mestimator-divergence} \Mfrac{N \cdot \hN \Mskript_m^2(\bm{X}^{(1)},\ldots, \bm{X}^{(N)})}{{n \choose m}}  &\xrightarrow[N\to\infty]{a.e.} \infty &\text{ if $X_1,\ldots,X_n$ are $m$-dependent but $(m-1)$-independent.}
\end{align}
where $Q$ is a Gaussian quadratic form with $\E Q = 1$.
\end{theorem}
\begin{proof}
Let the assumptions of Theorem \ref{thm:convergencem} be satisfied. 
Then \eqref{eq:convergencem} holds, since in this case \eqref{eq:estimator-convergence} implies the convergence of each of the $\genfrac(){0pt}{}{n}{m}$ summands of \eqref{eq:nsm-multi} to a Gaussian quadratic form with expectation $1.$ Thus, due to the normalizing factor in \eqref{eq:nsm-multi}, the limiting distribution has expectation 1. Further note that given \eqref{eq:mom1-log} all these quadratic forms can be expressed as a stochastic integral with respect to the same process, cf.\ \cite[Eq.\ (S.15)]{BoetKellSchi2019-supp}. This yields (by the same arguments as in the case of total multivariance \cite[Section 4.3]{BoetKellSchi2019}) that the limiting distribution is in fact the distribution of a Gaussian quadratic form. If all $\psi_i$ are bounded, a proof analogous to the one for the convergence of total multivariance in Theorem \ref{thm:asymp-Vstat} shows the result.

The divergence \eqref{eq:mestimator-divergence} follows by \eqref{eq:estimator-divergence}. The latter implies under the given assumptions that at least one summand of \eqref{eq:nsm-multi} diverges.
\end{proof}

Analogous to the case of (total) multivariance the above theorem immediately yields a test for $m$-independence which is (under the given moment conditions) consistent against all alternatives.

\begin{test}[Test for $m$-independence, given $(m-1)$-independence] \label{testC} 
If either the $\psi_i$ are bounded or \eqref{eq:mom1-log} holds, then a test for $m$-independence is given by:
Reject $m$-independence if $X_1,\ldots,X_n$ are $(m-1)$-independent and 
\begin{equation} \label{eq:testmm}
N \cdot \hN \Mskript_m^2(\bm{x}^{(1)},\ldots,\bm{x}^{(N)}) > R,
\end{equation}
with $R$ as discussed in Section \ref{sec:test}. (Note that one has to replace $\Mskript$ by $\Mskript_m$ in \eqref{eq:Rresampling} and \eqref{eq:RMC} to get $R$ for the resampling test and the Monte Carlo test, respectively.)
\end{test}

For a test of $m$-independence (without controllable type II error) one can drop in Test \ref{testC} the assumption of $(m-1)$-independence, cf.\ Remark \ref{rem:testA}.

As a special case, for $m=2$, the Test \ref{testC} becomes a test for pairwise independence.

\begin{test}[Test for pairwise independence] \label{testD} 
If either the $\psi_i$ are bounded or \eqref{eq:mom1-log} holds, then a test for pairwise independence is given by:
Reject pairwise independence if 
\begin{equation} \label{eq:test-pairwise}
N \cdot \hN \Mskript_2^2(\bm{x}^{(1)},\ldots,\bm{x}^{(N)}) > R,
\end{equation}
with $R$ as discussed in Section \ref{sec:test}. (Note that one has to replace $\Mskript$ by $\Mskript_2$ in \eqref{eq:Rresampling} and \eqref{eq:RMC} to get $R$ for the resampling test and the Monte Carlo test, respectively.)
\end{test}

Examples of the use of $m$-multivariance are given in the Section \ref{sec:examp}, e.g.\ Example \ref{ex:YaoZhanShao2017}, and in the supplement. To roundup this section we discuss some related estimators.

\begin{remark} \label{rem:mmulti}
\begin{enumerate}
\item Analogously to total multivariance one can define \textbf{total $\bm{m}$-multivariance} for $\bm{X} = (X_1,\ldots,X_n)$ by
\begin{equation}
\overline{M}_{m,\rho}^2(\bm{X}) :=  \sum_{\substack{1\leq i_1 < \ldots < i_r \leq n \\ 2 \leq l \leq m}} M_{\otimes_{k=1}^l \rho_{i_k}}^2(X_{i_1},\ldots,X_{i_l}) =  \sum_{l =2}^m M_{l,\rho}^2(\bm{X})
\end{equation}
and calculate its sample version. There might be computationally simpler representations using formulas for $(A_1 + \ldots + A_n + 1)^m$. Moreover, also the complements of these measures, e.g.\ $\overline{M}-M_3-M_2 = \overline{M}-\overline{M}_3$, might be of interest for multiple testing of higher order dependencies with disjoint hypotheses.
\item The simple form of the sample $2$-multivariance in \eqref{def:2mmulti} might suggest other generalizations. For example one could also consider
\begin{align}
 \hN \widetilde M_3(\bm{x}^{(1)},\ldots,\bm{x}^{(N)}) &:= \sqrt{ \frac{1}{2} \frac{1}{N^2}\sum_{k,l=1}^N \left( \left((A_1 + \ldots + A_n)_{kl}\right)^3 - \sum_{i=1}^n \left((A_i)_{kl}\right)^3 \right) }
\end{align}
as an estimator for 3-independence. In fact in the case of 2-independence this provides (assuming \eqref{d1m} and using \cite[Corollary 4.7]{BoetKellSchi2019}) a weakly consistent estimator for $M_3$. Hereto just note that the sums of all mixed terms of the form $((A_i)_{kl})^2 (A_j)_{kl}$ with $i \neq j$ are estimators for multivariances like $M(X_i,X_i,X_j)$, and the factorization for independent subsets \eqref{eq:factorization} yields  $M(X_i,X_i,X_j)= M(X_i,X_i) M(X_j) = 0$. But note that the estimators for these terms squared and scaled by $N$ do usually not vanish for $N\to \infty$. Thus a result like Theorem \ref{thm:convergencem} fails to hold.
\item \label{remsub:mmulticor} A further natural extension is to introduce the corresponding global scaled measures of $m$-dependence, i.e.\ $m$-multicorrelations. These require finite $\psi$-moments of order $m$ (cf.\ \eqref{eq:mom-n}). E.g.\ $\bm{2}$\textbf{-multicorrelation} is given by 
\begin{equation} \label{eq:mcor2}
\Mcor_{2,\rho}(X_1,\ldots,X_n) :=  \sqrt{ \genfrac(){0pt}{}{n}{2}^{-1} \sum_{1\leq i < j \leq n} \frac{M_{\rho_i \otimes \rho_{j}}^2(X_{i},X_{j})}{\sqrt{M_{\rho_i\otimes\rho_{i}}^2(X_{i},X_{i})M_{\rho_j\otimes\rho_{j}}^2(X_{j},X_{j})}}}
\end{equation}
and it coincides with the (analogously defined) $\Rskript_2$ since for $n=2$ the scaling factors in \eqref{eq:multicorr} and \eqref{eq:Mcor} coincide. Moreover these factors have for each summand in \eqref{eq:mcor2} the same exponent, thus (in contrast to $\overline{\Rskript}$ and $\overline{\Mcor}$) one gets efficient sample representations by replacing the $A_i$ in \eqref{def:2mmulti} by those in \eqref{eq:multicorr} (or equivalently \eqref{eq:Mcor}). This correlation satisfies all the dependence measure axioms of Section \ref{sec:axioms} when one replaces (A1) by the \textit{characterization of pairwise independence}, see Table \ref{tab:axioms}.
\end{enumerate}
\end{remark}

\section{Dependence structure visualization and detection} \label{sec:algo}

In this section a visualization of higher order dependencies of random variables $X_1,\ldots,X_n$ using an undirected graph is introduced. The population version and estimation procedures are discussed. In Section \ref{sec:examp} and in the supplement various examples are presented. The implementation of the visualization in R relies in particular on the package \texttt{igraph} \cite{CsarNepu2006}. 

The dependence structure graph consists of three elements (cf.\ Figure \ref{fig:intro}):
\begin{itemize}
	\item \textbf{Circled nodes} denote random variables.
    \item \textbf{Edges} denote dependencies. 
		\item \textbf{Non-circled nodes} ('dependency nodes') are primarily used to denote the dependence of the connected nodes and a label might denote the strength of the dependence or a related quantity (in our sample setting it is the value of the test statistic $N \cdot \hN \Mskript^2$ or the order of dependence) of the connected nodes. Secondarily they might be used to represent the 'random variable' which consists of all components of a connected cluster, e.g.\ in the right graph in Figure \ref{fig:intro} the node with label '97.1' represents the cluster of $X_1,X_2,X_{11}$.
	\end{itemize}

%

A visualization of the \textbf{full dependence structure} is constructed by adding the corresponding 'dependency nodes' and edges for any $m$-tuple of $X_i,\ldots,X_n$ which is $m$-dependent but $(m-1)$-independent.
In general this graph can be very overloaded, see Example \ref{ex:full}.

The direct approach to the full dependence structure based on samples is to test successively all $(m-1)$-independent $m$-tuples for $m$-independence for $m=2,\ldots,n$, adjust the p-values appropriately for multiple testing and add the significant dependency nodes and edges. For such a  test procedure a direct visualization of the tests p-values was introduced in \cite{GeneRemi2004}: the dependogram. Note that the full dependence structure visualizes the (lowest order) significant findings in a dependogram, see Example \ref{ex:dependogram} for more details. In practice a visualization of the full dependence structure is only feasible for small $n$, since for $n$ random variables there are $2^n-n-1 = \sum_{k=2}^n \genfrac(){0pt}{}{n}{k}$ tuples to consider.

To overcome (or at least to reduce) the drawbacks of the full dependence structure one can use a clustered dependence structure. Hereto each set of connected vertices in an undirected graph will be called cluster. Then the \textbf{clustered dependence structure} graph is constructed by the following algorithm: 0. Include the circled nodes for $X_1,\ldots,X_n.$
\begin{enumerate}[1.]
\item Let $k$ be the number of clusters currently in the graph and $Y_i$, $i=1,\ldots,k$ be random variables which have as components the connected $X_j$ of cluster $i$, e.g. $Y_1= (X_1,X_2,X_{11})$ if $X_1,X_2,X_{11}$ are connected via some edges. (In the very first run this amounts to: $k:=n$ and $Y_i:= X_i$.) Moreover, set $m=2.$
\item If $m>k$ the graph construction is finished, otherwise: For all $m$-dependent subsets of $Y_1,\ldots,Y_k$ add the corresponding dependency nodes and edges (connected to some non-circled node representing the cluster, if the cluster consists of more than one random variable) to the graph. If new nodes were introduced, go to step 2 otherwise repeat this step with $m$ increased by 1. 
\end{enumerate}
Since dependence and independence are not transitive, some information might be lost in the clustered dependence structure. Nevertheless, note that clustering preserves dependence, e.g.\ if at least one of the random variables $X_i,i\in  I$ is dependent with one of $X_k,k\in K$ then also $(X_i, i\in I\cup J)$ is dependent with $(X_k, k\in K\cup L).$

The visualization algorithm for a clustered dependence structure based on samples is analogous to the above, just in step 2 the $m$-independence has to be tested. Here one can (we do so) choose to skip sets of variables which have been tested before, i.e., sets which remained unchanged after the last cluster detection. But in any case the p-values have to be adjusted appropriately for multiple testing.

The \textbf{\textit{appropriate} adjustment of p-values} due to multiple testing is the basis for many debates. For the full dependence structure and for the clustered dependence structure the situation is complicated by the fact that the total number of tests is unknown at the beginning, and the result of the tests in one step influence (by indicating that some tuples are lower order dependent or by clustering) the data for the tests thereafter. Thus adjusting p-values after clustering would usually require new tests. An approach which avoids this uses Holm's method separately for each set of multiple tests in step 2 of the algorithm, but one has to keep in mind that by this the global type I error bound increases with each set of tests.

In general one might also distinguish between visualizations of the results of \textbf{tests using a given significance level} (in this case there is a bound for a type I error based on the significance level and it depends also on the correction for multiple testing used) or visualizations using \textbf{consistent estimators} (if these exist they might also be based on tests, but then the significance level or rejection level is adapted based on the sample size, which might make it harder to get explicit error estimates). In the case of tests with a fixed significance level, the range of significance levels which yield the same test results might give an additional indication of the \textit{reliability} of the detection.

\begin{remark}[Comment on detection errors]\label{rem:detect-error}
The probability of a type I error can be estimate and/or bounded by the choice of the rejection level or significance level. But a type II error bound or estimate might not be available. In this case one has to keep in mind that, due to the successive estimation/testing procedure a type II error (i.e., a not detected dependence) for some tuple can yield a detected higher order dependence for a superset of the tuple. Thus in this case the higher order dependence still indicates that the components of the tuple are not independent (but they might not be lower order independent).
\end{remark}

All of the above applies to the use of any multivariate dependence measure or test in the dependence structure detection algorithms. Now we turn explicitly to the case of multivariance.

\subsection{Dependence structure detection using distance multivariance}

For consistent estimates using multivariance the following observation is essential (cf.\ Theorem \ref{thm:convergence} and Corollary \ref{cor:betaM-consistency}):
\begin{corollary}[Consistent dependence estimation]\label{cor:consistent-dep-struct}
Let $X_i, i=1,\ldots,n$ be $(n-1)$-independent non-constant random variables and let $\bm{X}^{(k)}, k=1,\ldots,N$ be independent copies of $(X_1,\ldots,X_n)$. If either the $\psi_i$ are bounded or \eqref{eq:mom1-log} holds, then for any $\beta \in (0,1)$
\begin{equation}
N^\beta \cdot \Mskript^2(\bm{X}^{(1)},\ldots,\bm{X}^{(N)}) \xrightarrow[N\to\infty]{a.e.} \begin{cases}
\infty&\text{,if $X_i,\ldots,X_n$ are dependent},\\
0&\text{,if $X_i,\ldots,X_n$ are independent.}
\end{cases}
\end{equation}
\end{corollary}
Hence using $R=R(N):= N^{1-\beta} \cdot C$ for any fixed constants $\beta \in (0,1), C>0$ in the independence Tests \ref{testA}, \ref{testB}, \ref{testC}, \ref{testD} provides strongly consistent tests, in the sense that (under the assumptions of the tests) the test result converges almost surely to the correct statement as $N\to \infty.$ Clearly the convergence speed of the estimator depends on the choice of $\beta$ and $C$ (see below for a rough error estimate). 

Therefore there are several options for the dependence structure detection using Test \ref{testA}: the very fast but conservative rejection level given in Theorem \ref{thm:testqf}, the (extremely) slow but approximately sharp rejection level provided by the resampling approach \eqref{eq:Rresampling} (or by the Monte Carlo approach \eqref{eq:RMC}) and the value $R:=N^{1-\beta}\cdot C$ for consistent tests. 

Of the above options the conservative estimate and the resampling approach provide (almost) directly also the corresponding p-values, but only the conservative and the consistent approach are feasible. For the resampling approach the sample size would have to be adapted (increased!) if the p-values are adjusted -- yielding in general an extremely slow algorithm. For the consistent estimator the corresponding p-value can only be estimated (using one of the other methods) and the convergence rates have not been analyzed in detail yet, thus the actual type I error for a given finite sample is not directly available. Nevertheless note that fast and approximately sharp methods to estimate the p-values of multivariance are developed in the preprint \cite{BersBoet2018v2}, see also \cite{Guet2017}. Moreover, an approximation of an upper bound for the type I error of the consistent estimator is given by the following: If one performs $k$ independent sharp tests with significance levels $\gamma_i$ then the probability of a type I error is $1-\prod_{i = 1}^k (1-\gamma_i)$. In the setting of multivariance the tests are in the limit (under $H_0$) independent and $\gamma_i \leq F_{\chi_1^2}(N^{1-\beta} \cdot C),$ thus posterior to testing the number of tests performed is known, say $k$, and the bound becomes $1-(1-F_{\chi_1^2}(N^{1-\beta} \cdot C))^k$ for the consistent estimator. Concerning $\beta$ and $C$ note that for the estimator discussed in Corollary \ref{cor:consistent-dep-struct} with $\beta$ close to 1 the convergence to 0 (in the case of independence) becomes slower, for $\beta$ close to 0 the divergence to $\infty$ (in the case of dependence) becomes slower, here $\beta = 1/2$ seems a balanced choice. For the value of $C$ an optimal recommendation is still open -- we will use $C = 2$ in our examples where it seems a reasonable choice. Naturally, the constant $C$ could also be based on a rejection level for a fixed sample size, e.g.\ choose $C$ such that $5\cdot C$ is the rejection level for a sample of size 25 for a significance level of 0.05. Then at least for the this sample size the probability of a type I error is known, but this would still require some p-value estimation.   

Finally, note that the above are basic algorithms. There are certainly several variants and extensions possible, e.g.\ a further speedup might be obtained by using total multivariance and $m$-multivariance for initial tests of independence (but beware of the problem of multiple vs.\ single tests). Furthermore, if pairwise dependence is detected (and clustered) this can be further analyzed in the framework of graphical models. Hereto also note that \cite[Section 5.2]{PfisBuehSchoPete2017} (see also \cite[Section 6]{ChakZhan2019}) provides a method for the detection of causal relations of variables using multivariate dependence measures, this can be used to refine an undirected graph visualizing the dependence structure into a directed graph. Moreover, clearly the visual layout allows many variants, e.g.\ one might also use different line types and thicknesses to indicate the dependence strength or order, also the denoted values could for example be replaced by p-values (since for given marginals there is a one-to-one correspondence between the value of multivariance and its (exact) p-value).

\section{Examples} \label{sec:examp}

In the supplementary\footnote{pages \pageref{sec:supp-example} ff.\ of this manuscript} Section \ref{sec:supp-example} a comprehensive collection of illustrating examples is provided. These discuss in detail several (toy-)examples of higher order dependence and their visualization, including the full and clustered dependence structure detection. Moreover also the detection power, empirical size and various other properties of multivariance are studied. 

Here in the main text we will only discuss two types of examples: Comparisons with other multivariate dependence measures and two basic real data examples.

The presented tables will contain additionally a test called 'Comb' which combines the tests of $m$- and total multivariance by Holm's method. This provides a reference for readers with an interest in a joint test procedure, rather than comparing individual tests in their realm. For a full explanation of the setting, terms and parameter values of the studies we refer to the introduction of the supplementary Section \ref{sec:supp-example}. 

\subsection{Empirical comparison of multivariance with other dependence measures} \label{sec:ex-compare}

As discussed in Section \ref{sec:compare-3} there are several dependence measures closely related to distance multivariance and its variants. For these empirical power comparisons are provided in Examples \ref{ex:JinMatt2017} and \ref{ex:YaoZhanShao2017}. But we begin with an example of a different visualization of higher order dependence which was proposed alongside a copula based dependence measure.

\begin{example}[Dependogram vs. visualization]\label{ex:dependogram}
In \cite{GeneRemi2004} copula based higher order dependence tests were proposed together with a dependogram, which provides a graphical representation of the test results of multiple testing. Our proposed visualization is closely related, it provides a visualization of the (lowest order) significant dependencies.

In Figure \ref{fig:dependogram} the dependogram and the corresponding dependence structure (which is here actually detected using the same samples and distance multivariance) are depicted for the example provide in \cite[Section 4.2]{GeneRemi2004}: Let $Z_i, i=1,\ldots,5$ be independent standard normal variables and let $X_1:=|Z_i| \sgn(Z_1Z_2),$ $X_i := Z_i$ for $i=2,3,4$ and $X_5:= Z_4/2 + \sqrt{3}Z_5/2$. Now consider $N=50$ samples of $(X_1,\ldots,X_5).$

\begin{figure}[H]\centering
\includegraphics[width = 0.4\textwidth]{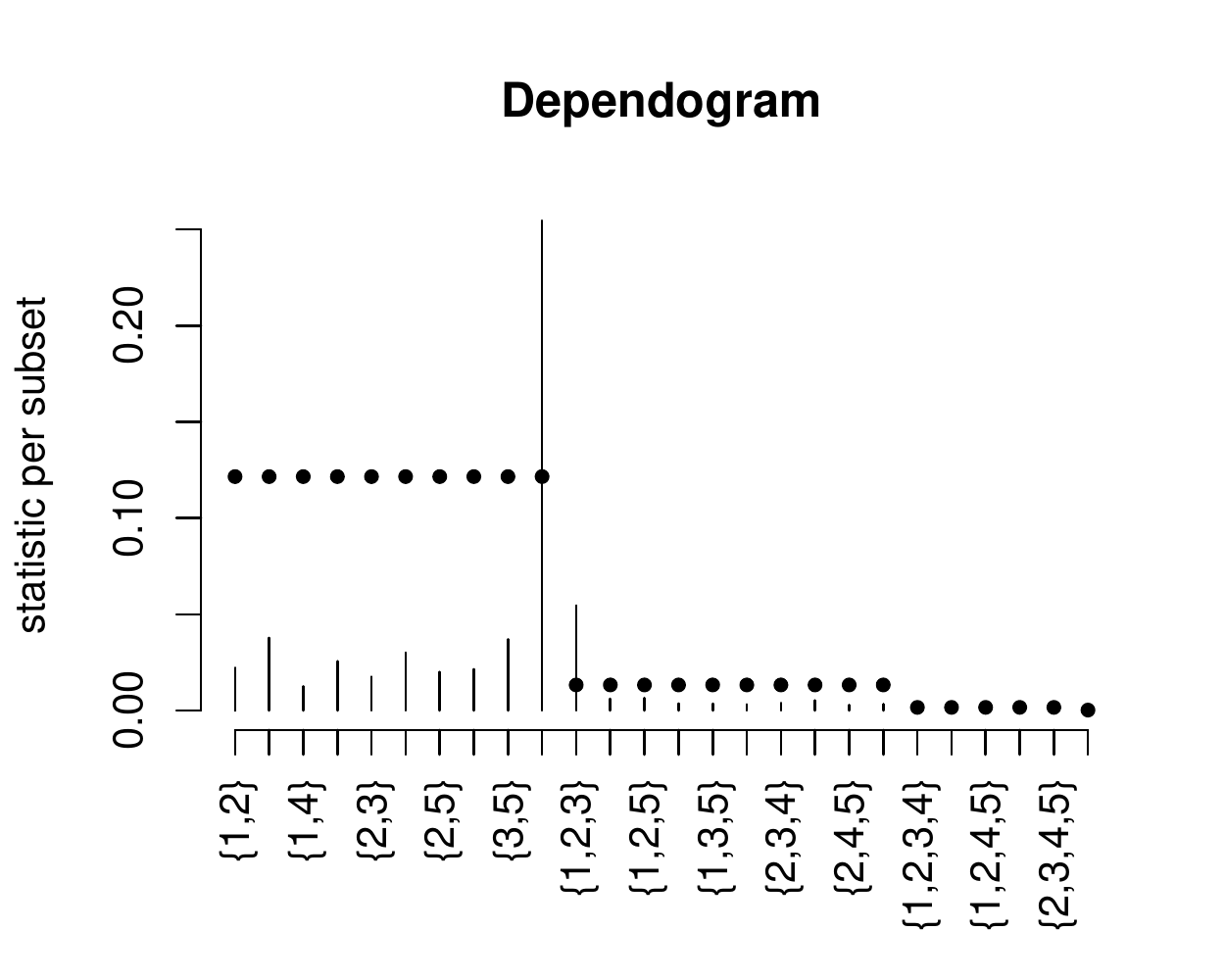}
\quad \includegraphics[width = 0.4\textwidth]{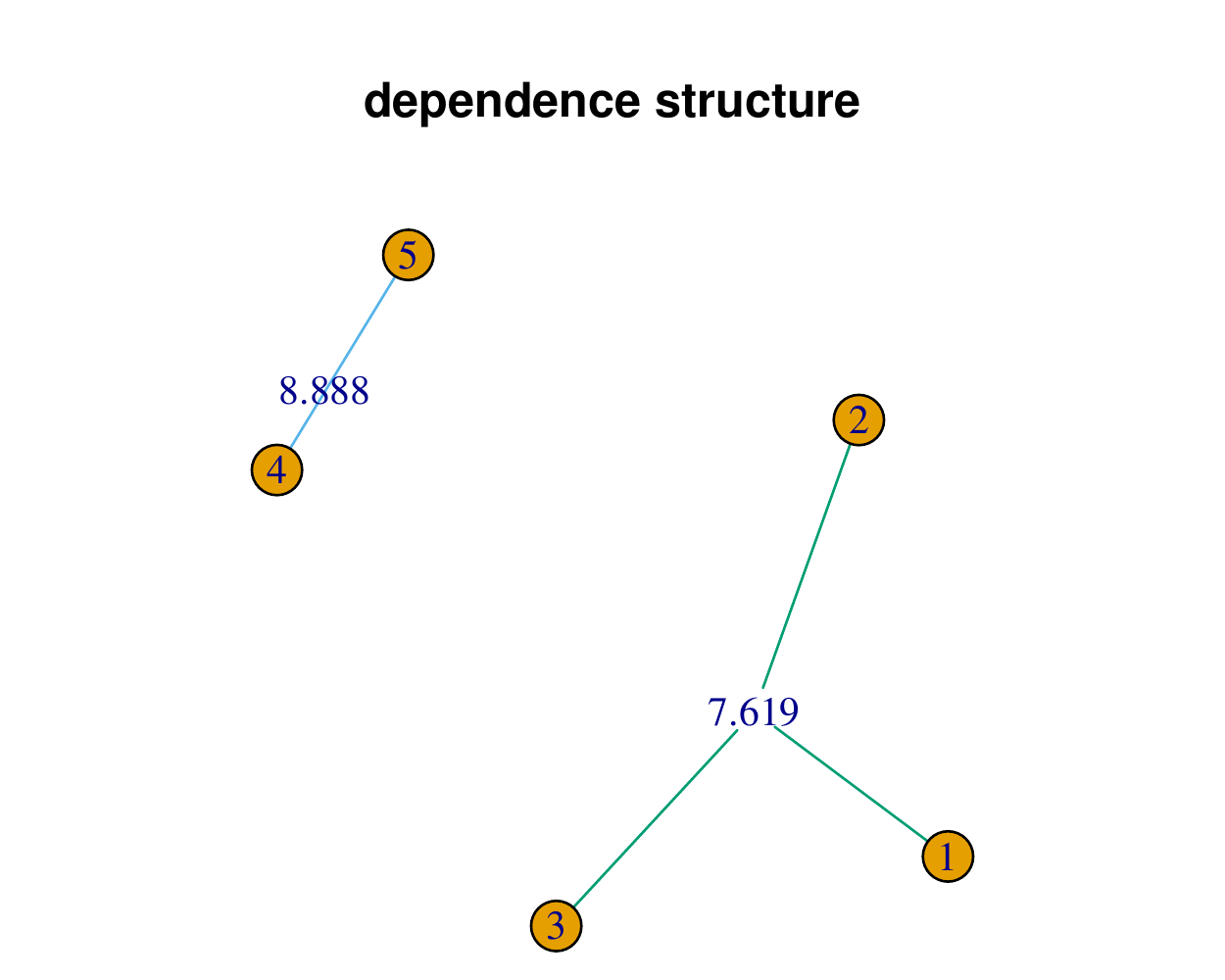}
\caption{Ex.\ \ref{ex:dependogram}: dependogram (see \cite{GeneRemi2004}; implemented in \cite{HofeKojaMaecYan2018}) and the corresponding dependence structure.}	
	\label{fig:dependogram}
\end{figure}
\end{example}

\begin{example}[Comparison with the methods of \cite{JinMatt2018}] \label{ex:JinMatt2017} Here we compare our estimators to those presented in  \cite{JinMatt2018}. To avoid confusion, note that in their main representation results is a sign typo: the second sum in Lemmata 1 and 2 should have a minus sign. In general, one should note that the estimators for multivariance have complexity $O(N^2)$ whereas the exact estimators of \cite{JinMatt2018} (e.g.\  $\Qskript_N,\Sskript_N$) have higher complexity. To reduce the complexity they introduce approximate estimators (e.g.\ $\Qskript_N^\star, \Jskript^\star_N$) which have the same complexity as ours. Note that these approximate estimators are not permutation invariant with respect to the order of the samples. In fact their positive finding (significant p-values) in the real data example \cite[6.2 Financial data]{JinMatt2018} is an artifact due to this shortcoming. Their estimators yield for the same date with permuted samples p-values about 0.3 and above. Therefore we strongly advise against the use of their approximate estimators in the given form. This problem can be reduced by permuting the samples prior to the use of their estimators.

Nevertheless, we decided to use their estimators for a comparison, since these are the most recent estimators related to the approach discussed in Section \ref{sec:compare-3} corresponding to \eqref{eq:kan-structure}. Moreover \cite{JinMatt2018} also provides several variants and comparative tables including other measures. The following tables are computed with their parameter settings, e.g.\ $\alpha = 0.1$. We only include their best exact and approximate estimators (for each particular example), for further comparisons see the full tables in \cite{JinMatt2018}. 

	  

The example \cite[Example 3]{JinMatt2018} considers random variables $X_i$ with values in $\R^5$ such that $(X_1,X_2,X_3)\sim N_{15}(0,\Sigma)$ with $\Sigma_{ij} = 1$ for $i=j$ and $0.1$ otherwise. For this example total multivariance and $2$-multivariance match the power of the exact estimator and outperform the approximate estimator (Figure \ref{fig:jm-multinorm}).
\begin{figure}[H] 
\centering
\includegraphics[width = 0.3\textwidth]{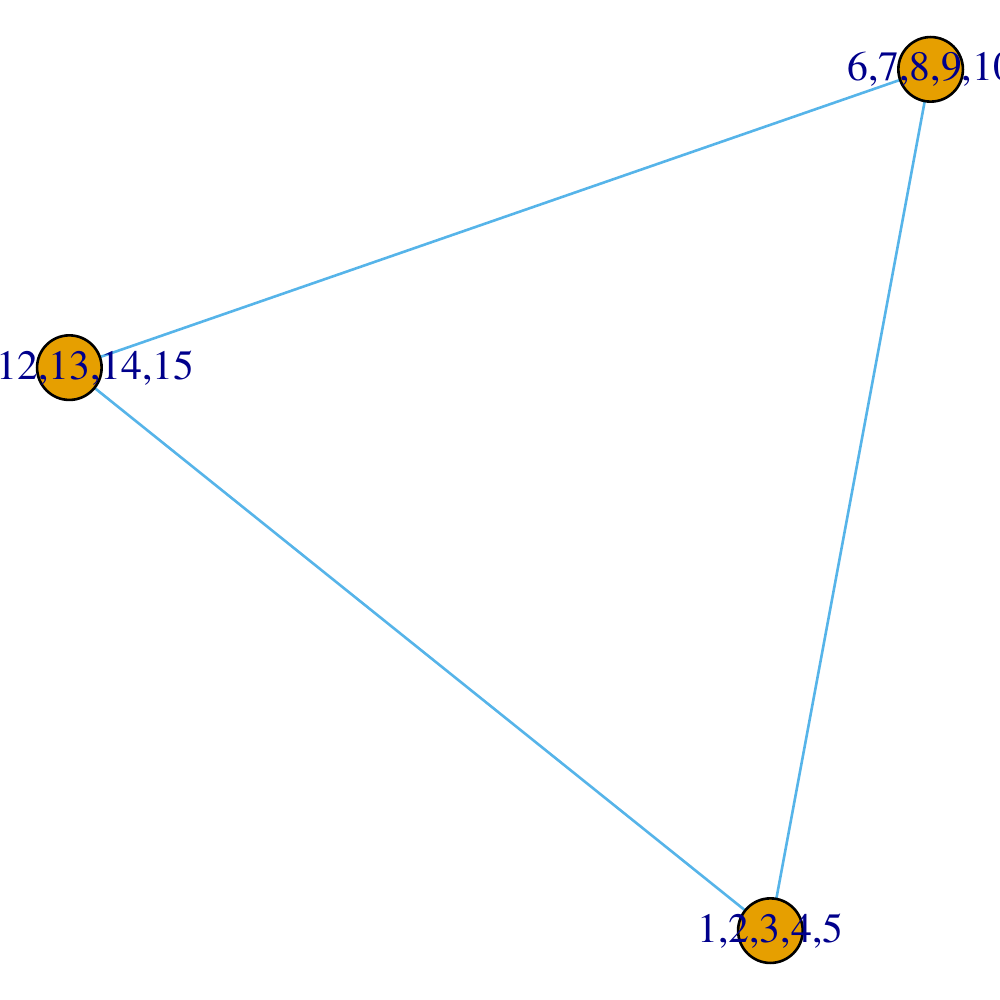}
\quad
\begin{tabular}[b]{r|cc|c|ccc}
	& \multicolumn{3}{c|}{resampling} & \multicolumn{3}{c}{*}\\	$N$ & $\hN \overline{\Mskript}$ & $\hN \Mskript_2$ & Comb & $\Qskript_N^\star$ & $\Sskript_N$ & $dHSIC$ \\ 	\hline
	25  & 0.408 & 0.417 & 0.359 & 0.220 & 0.418 & 0.982 \\
	50  & 0.712 & 0.722 & 0.631 & 0.378 & 0.719 & 1.000 \\
	100 & 0.960 & 0.970 & 0.941 & 0.707 & 0.961 & 1.000 \\
	150 & 0.995 & 0.995 & 0.993 & 0.873 & 0.996 & 1.000 \\
	200 & 1.000 & 1.000 & 1.000 & 0.946 & 1.000 & 1.000 \\
	300 & 1.000 & 1.000 & 1.000 & 0.997 & 1.000 & 1.000 \\
	500 & 1.000 & 1.000 & 1.000 & 1.000 & 1.000 & 1.000 \\ \hline
	\multicolumn{7}{r}{* values from \cite[Table 6]{JinMatt2018}}
\end{tabular}

\caption{Dependence structure sketch and empirical power comparison with \cite[Example 3]{JinMatt2018} (Ex.\ \ref{ex:JinMatt2017}).} \label{fig:jm-multinorm}
\end{figure}

As second example \cite[Example 4]{JinMatt2018} we consider $(Y_1,\ldots,Y_{15})\sim N_{15}(0,\Sigma)$ with $\Sigma_{ij} = 1$ for $i=j$ and $0.4$ otherwise and set $X_i:=(\ln(Y^2_{5i}),\ldots,\ln(Y^2_{5i + 4}))$ for $i=1,2,3$. Again, total multivariance and $2$-multivariance are close to the power of the exact estimator and outperform the approximate estimator (Figure \ref{fig:ex-JM-4}).

\begin{figure}[H] 
\centering
\includegraphics[width = 0.3\textwidth]{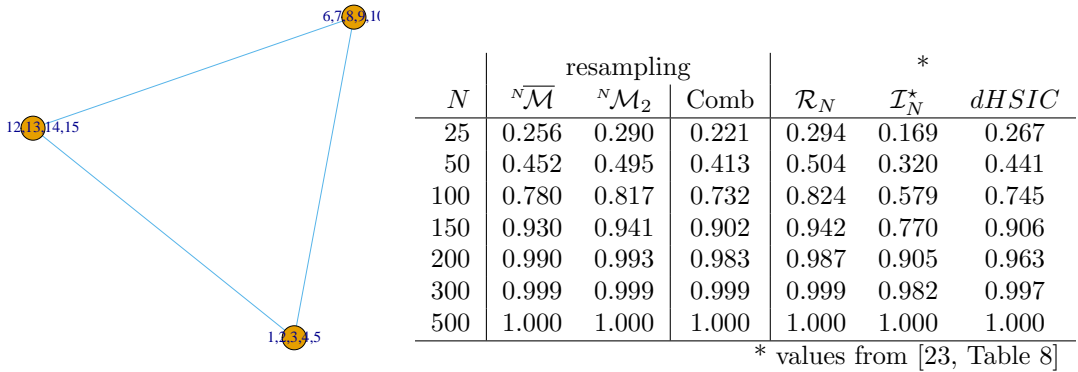}
\quad
\begin{tabular}[b]{r|cc|c|ccc}
	& \multicolumn{3}{c|}{resampling} & \multicolumn{3}{c}{*}\\
	$N$ & $\hN \overline{\Mskript}$ & $\hN \Mskript_2$ & Comb & $\Rskript_N$ & $\Iskript^\star_N$ & $dHSIC$\\ 
	\hline
	25 & 0.256 & 0.290 & 0.221 & 0.294 & 0.169 & 0.267\\ 
	50 & 0.452 & 0.495 & 0.413 & 0.504 & 0.320 & 0.441\\ 
	100 & 0.780 & 0.817 & 0.732 & 0.824 & 0.579 & 0.745\\ 
	150 & 0.930 & 0.941 & 0.902 & 0.942 & 0.770 & 0.906\\ 
	200 & 0.990 & 0.993 & 0.983 & 0.987 & 0.905 & 0.963\\ 
	300 & 0.999 & 0.999 & 0.999 & 0.999 & 0.982 & 0.997\\ 
	500 & 1.000 & 1.000 & 1.000 & 1.000 & 1.000 & 1.000\\ 
	\hline 
			   \multicolumn{7}{r}{* values from \cite[Table 8]{JinMatt2018}}
\end{tabular}
\caption{Dependence structure sketch and empirical power comparison with \cite[Example 4]{JinMatt2018} (Ex.\ \ref{ex:JinMatt2017}).}\label{fig:ex-JM-4}
\end{figure}

Finally, we discuss \cite[Example 5]{JinMatt2018}: for dimensions $n\in \{5,10,15,20,$ $25,30,50\}$ and sample size $N = 100$ we consider $(X_1,\ldots,X_{n})\sim N_{n}(0,\Sigma)$ with $\Sigma_{ij} = 1$ for $i=j$ and $0.1$ otherwise. Here $2$-multivariance is close to the power of the exact estimator and outperforms the approximate estimator (Figure \ref{fig:ex-JM-5}).

One might argue that the comparison with 2-multivariance is unjust, since it is a measure of pairwise dependence, whereas the other measures detect any kind of dependence. Hereto note that also the combination of the measures in 'Comb' has a higher detection rate than the approximate estimators.
\begin{figure}[H]
	\centering
	\includegraphics[width = 0.3\textwidth]{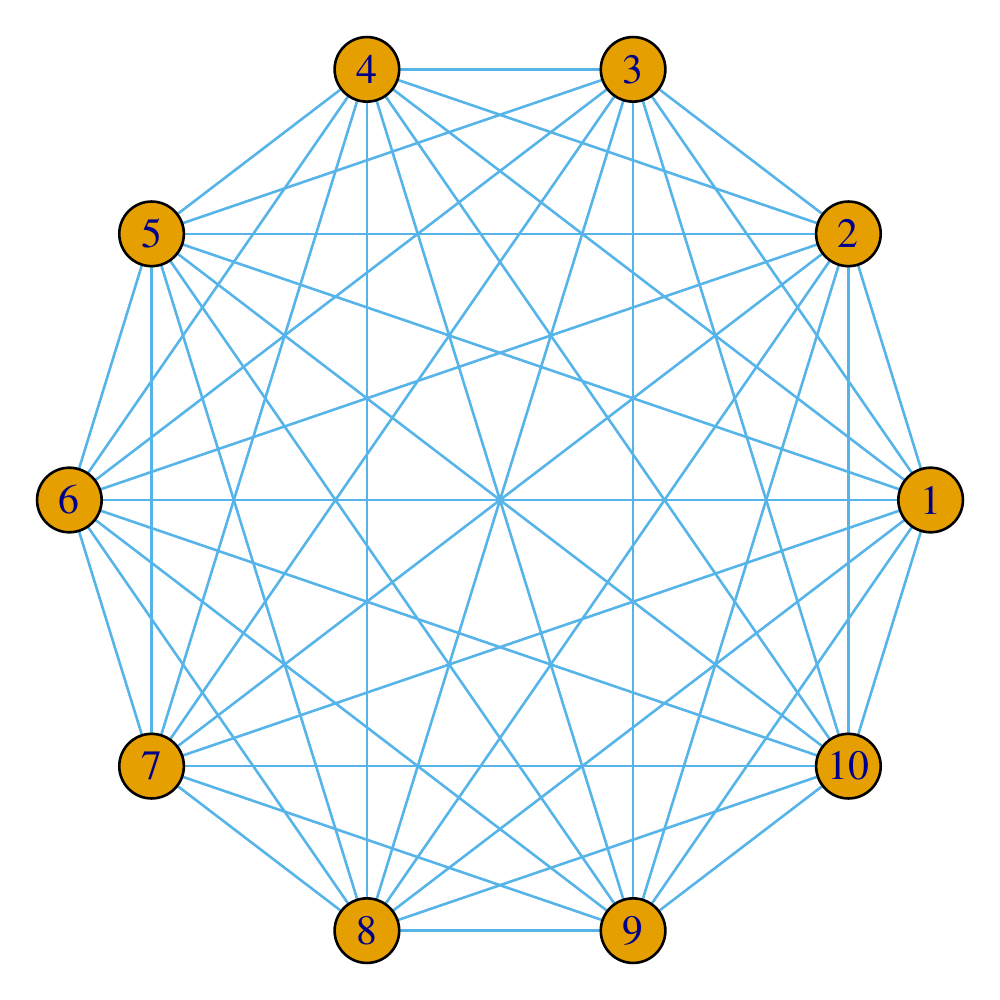}
	\quad
			\begin{tabular}[b]{r|cc|c|cc|cc}
	& \multicolumn{3}{c|}{resampling} & \multicolumn{2}{c|}{distribution-free} &  \multicolumn{2}{c}{* }\\
		$n$ & $^{100} \overline{\Mskript}$ & $^{100} \Mskript_2$ & Comb & $^{100} \overline{\Mskript}$ & $^{100} \Mskript_2$  & $\Qskript_{100}^\star$ & $\Sskript_{100}$ \\ 
		\hline		5 & 0.423 & 0.515 & 0.409 & 0.000 & 0.000 & 0.298 & 0.557 \\ 
		10 & 0.252 & 0.873 & 0.780 & 0.003 & 0.000 & 0.557 & 0.915 \\ 
		15 & 0.374 & 0.972 & 0.946 & 0.012 & 0.000 & 0.822 & 0.982\\ 
		20 & 0.443 & 0.995 & 0.988 & 0.054 & 0.000 & 0.924 & 0.999 \\ 
		25 & 0.532 & 1.000 & 0.999 & 0.164 & 0.000 & 0.977 & 0.999\\ 
		30 & 0.588 & 1.000 & 1.000 & 0.234 & 0.000 & 0.980 & 1.000\\ 
		50 & 0.821 & 1.000 & 1.000 & 0.657 & 0.000 & 0.998 & 1.000\\ 
		\hline 
		   \multicolumn{8}{r}{* values from \cite[Table 10]{JinMatt2018}}
	\end{tabular}
	\caption{Dependence structure sketch ($n=10$) and empirical power comparison with \cite[Example 5]{JinMatt2018} (Ex.\ \ref{ex:JinMatt2017}).} \label{fig:ex-JM-5}
\end{figure}
\end{example}

\begin{example}[Comparison with the methods of \cite{YaoZhanShao2017}] \label{ex:YaoZhanShao2017} 
In \cite{YaoZhanShao2017} several measures of dependence were introduced. The main contribution is a measure $dCov$ for pairwise dependence, which is closely related to $\hN \Mskript_2$. The examples in \cite{YaoZhanShao2017} use the parameters $N \in \{60,100\}$ and $n \in \{50,100,200,400,800\}$ and $\alpha = 0.05$, which we also use here to provide values which can be compared to other dependence measures given in their tables. Let $X_i,i = 1,\ldots,n$ be random variables with values in $\R^1$ such that $(X_1,\ldots,X_n) \sim N(0,\Sigma)$. We consider \cite[Example 2]{YaoZhanShao2017}, hereto let $\Sigma \in \R^{n\times n},$ $\Sigma_{ij} = 1$ for $i=j$ and otherwise (for $i\neq j$) set:
\begin{enumerate}[a)]
\item auto-regressive structure: $\Sigma_{ij}=(0.25)^{|i-j|},$ 
\item band structure: $\Sigma_{ij}=0.25$ for $0<|i-j|<3$ and 0 otherwise,
\item block structure: $\Sigma= I_{\lfloor n/5 \rfloor} \otimes A$ where $I_k\in\R^{k\times k}$ is the identity matrix and $A\in \R^{k\times k}$  with $A_{ij}=1$ for $i=j$ and  $0.25$ otherwise.
\end{enumerate}
\enlargethispage{0.7\baselineskip}
\begin{figure}[H]
	\centering
	\hspace{1.9cm}
	 \includegraphics[width = 0.25\textwidth]{./Figs/plain-graphs-1.pdf}
	\quad\quad
	\includegraphics[width = 0.25\textwidth]{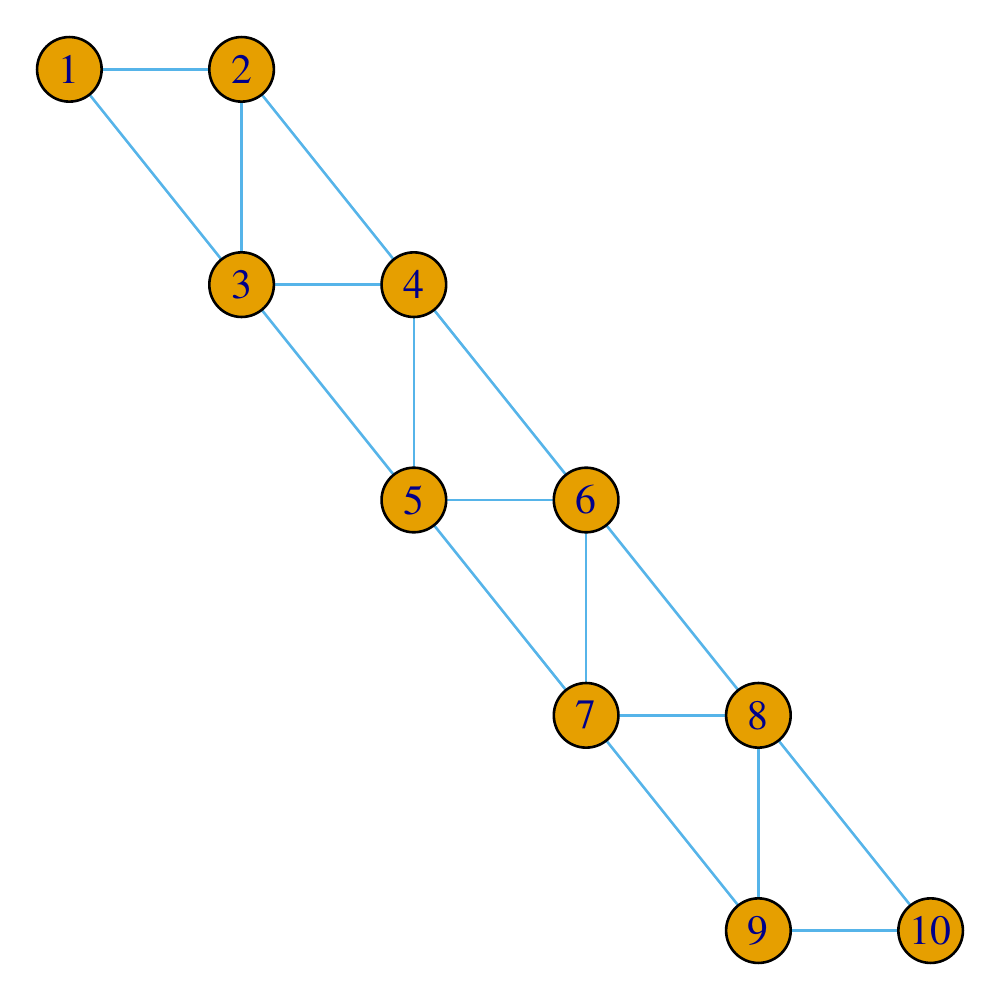}
	\quad\quad 
	\includegraphics[width = 0.25\textwidth]{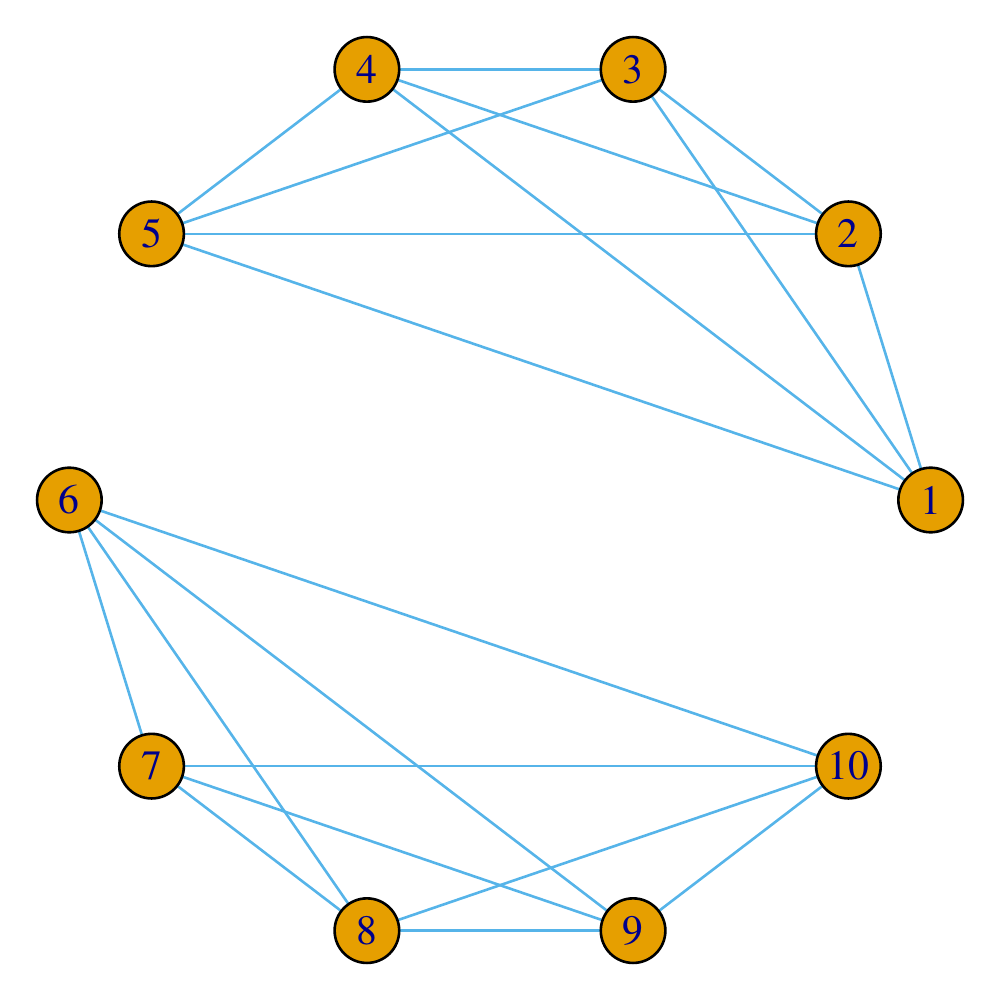}
	\begin{tabular}[b]{rr|cc|c|c|cc|c|c|cc|c|c}
	&& \multicolumn{4}{c|}{auto-regressive}& \multicolumn{4}{c|}{band structure}& \multicolumn{4}{c}{block structure}\\
	&& \multicolumn{3}{c|}{resampling} &  \multicolumn{1}{c|}{*}& \multicolumn{3}{c|}{resampling} &  \multicolumn{1}{c|}{*}& \multicolumn{3}{c|}{resampling} &  \multicolumn{1}{c}{*}\\
		 		$n$ & $N$ & $\hN \overline{\Mskript}$ & $\hN \Mskript_2$ & Comb &$ dCov$& $\hN \overline{\Mskript}$ & $\hN \Mskript_2$& Comb &$ dCov$& $\hN \overline{\Mskript}$ & $\hN \Mskript_2$& Comb &$ dCov$ \\		  \hline
50 & 60 & 0.052 & 0.898 & 0.768 & 0.886& 0.159 & 0.999 & 0.999 &1.000& 0.232 & 0.998 & 0.994 &0.999\\ 
  100 & 60 & 0.108 & 0.873 & 0.807 & 0.906& 0.192 & 1.000 & 1.000 &0.999& 0.234 & 1.000 & 0.998 &1.000\\ 
  200 & 60 & 0.104 & 0.896 & 0.765 & 0.909& 0.167 & 1.000 & 0.999 &1.000& 0.139 & 0.999 & 0.998 &1.000\\ 
  400 & 60 & 0.111 & 0.924 & 0.812 & 0.909& 0.174 & 0.998 & 0.998 &1.000& 0.177 & 1.000 & 0.999 &1.000\\ 
  800 & 60 & 0.101 & 0.937 & 0.843 & 0.908& 0.105 & 1.000 & 1.000 &1.000& 0.128 & 1.000 & 1.000 &1.000\\ 
  50 & 100 & 0.115 & 0.999 & 0.996 & 0.998& 0.137 & 1.000 & 1.000 &1.000& 0.195 & 1.000 & 1.000 &1.000\\ 
  100 & 100 & 0.071 & 0.999 & 0.986 & 0.999& 0.153 & 1.000 & 1.000 & 1.000& 0.170 & 1.000 & 1.000 &1.000\\ 
  200 & 100 & 0.128 & 1.000 & 0.999 & 1.000& 0.142 & 1.000 & 1.000 &1.000& 0.222 & 1.000 & 1.000 &1.000\\ 
  400 & 100 & 0.073 & 1.000 & 1.000 & 0.999& 0.168 & 1.000 & 1.000 &1.000& 0.169 & 1.000 & 1.000 &1.000\\ 
  800 & 100 & 0.084 & 1.000 & 0.998 & 0.999& 0.139 & 1.000 & 1.000 &1.000& 0.191 & 1.000 & 1.000 &1.000\\ 
   \hline
   \multicolumn{14}{r}{* the $dCov$ values are from \cite[Table 2]{YaoZhanShao2017}}
		\end{tabular}
		\vspace{-0.7cm}
			\caption{Dependence structure sketches ($n=10$) and empirical power comparison with \cite[Example 2.a]{YaoZhanShao2017} (Ex.\ \ref{ex:YaoZhanShao2017}).} \label{fig:yao_ar}
\end{figure}
In all cases the performance of $2$-multivariance is very similar to their estimator, see Figure \ref{fig:yao_ar}. Note that due to computation time restrictions we used for the table in Figure \ref{fig:yao_ar} the resampling distribution of one sample to compute all resampling p-values (instead of resampling each sample separately).

In \cite[Example 6]{YaoZhanShao2017} random variables $(X_1,\ldots,X_n)$ are considered where the 3-tuples $(X_1,X_2,X_3)$, $(X_4,X_5,X_6)$, ...\ are independent and each 3-tuple consists of pairwise independent but 3-dependent Bernoulli random variables (as explicitly constructed in Example \ref{ex:2coins}). Here only the sample sizes and dimensions $(N,n) \in \{(60,18),(100,36),(200,72)\}$ are used. Figure \ref{fig:multi-coins} shows that $3$-multivariance (and also the combined test 'Comb') clearly outperforms all measures included in their table (of which we only cite two in our table).

\begin{figure}[H]
\centering
\includegraphics[width = 0.3\textwidth]{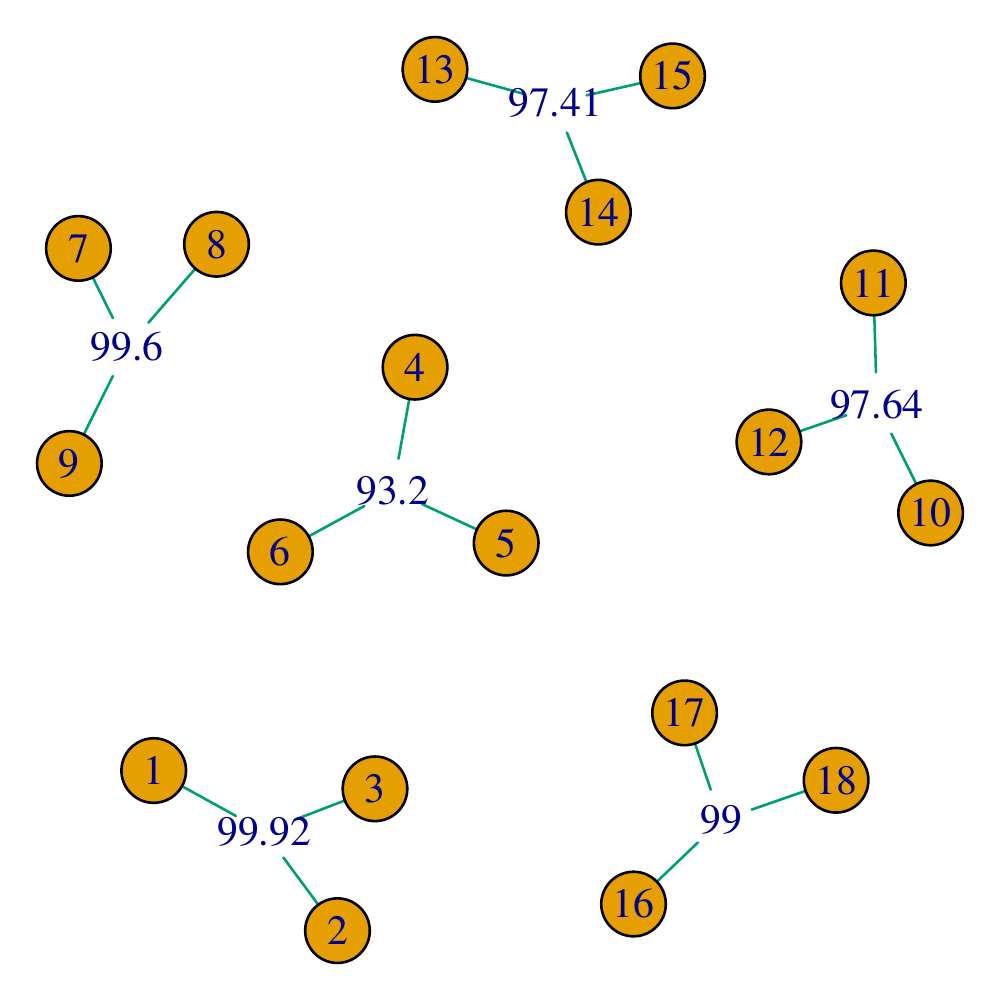}
\quad
\begin{tabular}[b]{rr|cc|c|cc}
	&& \multicolumn{3}{c|}{resampling} &  \multicolumn{2}{c}{\cite[Table 4]{YaoZhanShao2017}}\\
$n$ & $N$ & $\hN \overline{\Mskript}$ & $\hN \Mskript_3$ &Comb & $dCov$ & $dHSIC(3)$\\ 
  \hline
  18 & 60  & 0.112 & 1.000 & 1.000 & 0.051 & 0.708\\ 
  36 & 100 & 0.044 & 1.000 & 1.000 & 0.048 & 0.314\\ 
  72 & 200 & 0.047 & 1.000 & 1.000 & 0.057 & 0.073\\ 
   \hline
\end{tabular}
\caption{Dependence structure of \cite[Example 6]{YaoZhanShao2017} (with $n=18$) and the empirical power comparison. Note that here dHSIC(3) denotes dHSIC with a special choice of the bandwidth parameter, see \cite{YaoZhanShao2017} for details. (Ex.\ \ref{ex:YaoZhanShao2017}).} \label{fig:multi-coins}
\end{figure}
\end{example}

\subsection{Real data examples} \label{sec:ex-real-data}

As stated in the introduction, looking at other papers considering multivariate dependence measures (e.g.\ those discussed in Section \ref{sec:compare-3}) one notices that although these are capable of detecting dependencies of higher order the real data examples feature pairwise dependence. From our point of view this seems first of all to be due to the fact that the concept of higher order dependencies is not popular (or even unknown) in applied statistics. Therefore, on the one hand there is a very strong publication bias for datasets with pairwise dependencies, on the other hand even if datasets statistically feature higher order dependencies an explanation by field experts is yet missing. Nevertheless, we refer to \cite{Boet2019-tech} for a collection of more than 350 datasets which feature higher order dependence.

In the following we present two examples for which 2-multivariance and total multivariance detect some dependence. In terms of dependence structure detection they are more delicate: The first example illustrates the difference between the clustered and full dependence structure and it indicates an application of higher order dependencies to model selection. The second example discusses detected higher order dependencies which are actually based on pairwise dependence due to the small sample size, a conservative detection method and the chosen significance level (see also Remark \ref{rem:detect-error}).





\begin{example}[Quine's student survey data]\label{ex:quine} We consider a classical data set of a student survey \cite{Aitk1978} (see also \cite[R-package: \texttt{MASS}, dataset: \texttt{quine}]{VenaRipl2002}), which contains 146 samples of the variables: age (actually the class level), gender, cultural background, type of learner and the number of days school was missed. The dataset was extensively used in \cite{Aitk1978} to discuss model selection in a multi-factor analysis of variance to model the number missed school days. 

The conservative tests using 2-multivariance and total multivariance detect no dependencies (p-values: 0.0767, 0.1565), the corresponding resampling tests reject independence with actual p-values of 0.00.

The dependence structure detection yields the structures shown in Figure \ref{fig:quine}. Here the full structure provides a refinement of the clustered structure. For the detection we used resampling tests with 10000 resamples and significance level $\alpha = 0.01$. Based on the actually performed multiple tests the approximate probability of a type I error is 0.0297 for the clustered structure and 0.0199 for the full structure. By the large number of resamples used this example might just seem to be an (impractical) proof of concept, but note that the same results can also be obtained with the faster methods developed in \cite{BersBoet2018v2}.

For the variables: age, gender and missed-days 3-dependence (with lower order independence) was detected (Figure \ref{fig:quine}). To judge if this is really a sensible finding in terms of the field of study is beyond our expertise. Nevertheless the found dependencies naturally suggest candidates for a minimal model for the number of missed days: Based on the detected full dependence structure the missed days depend only on the cultural background and on the interaction term of age and gender. 
\begin{figure}[H]\centering
\includegraphics[width = 0.3\textwidth]{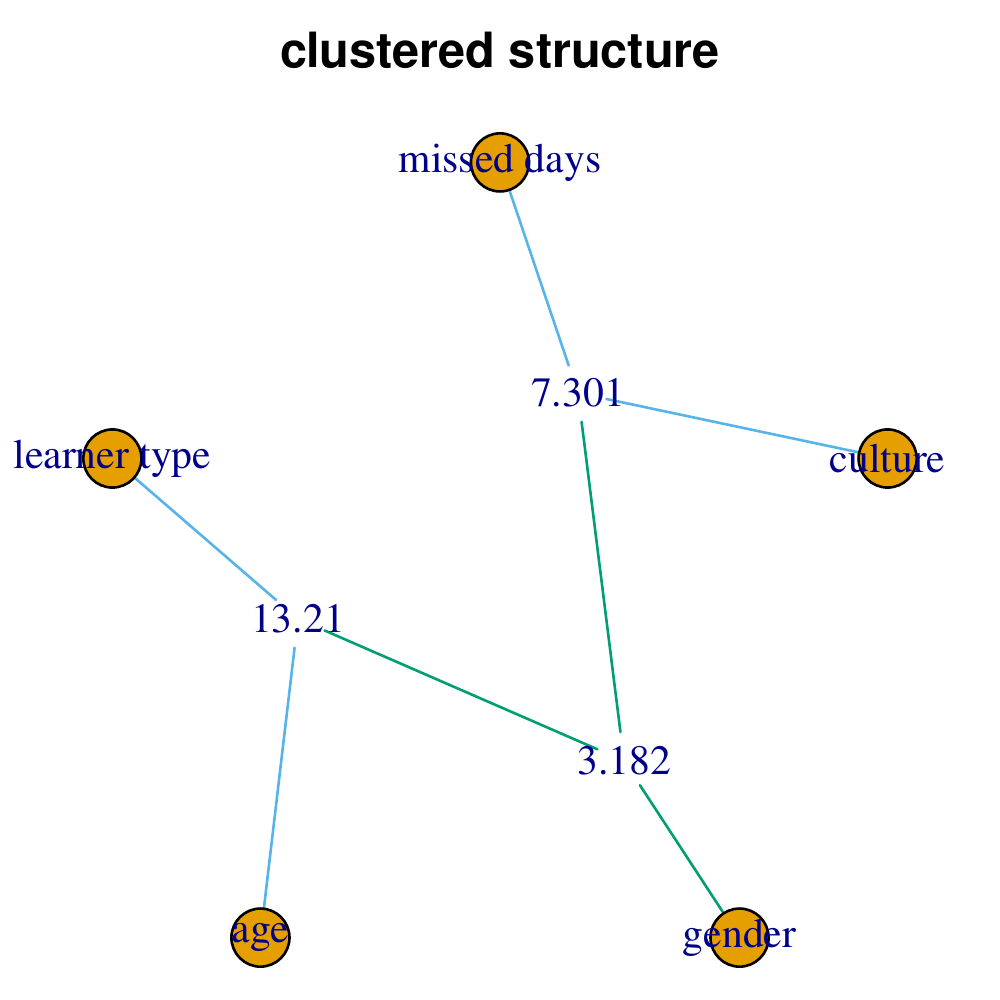}\quad
\quad \includegraphics[width = 0.3\textwidth]{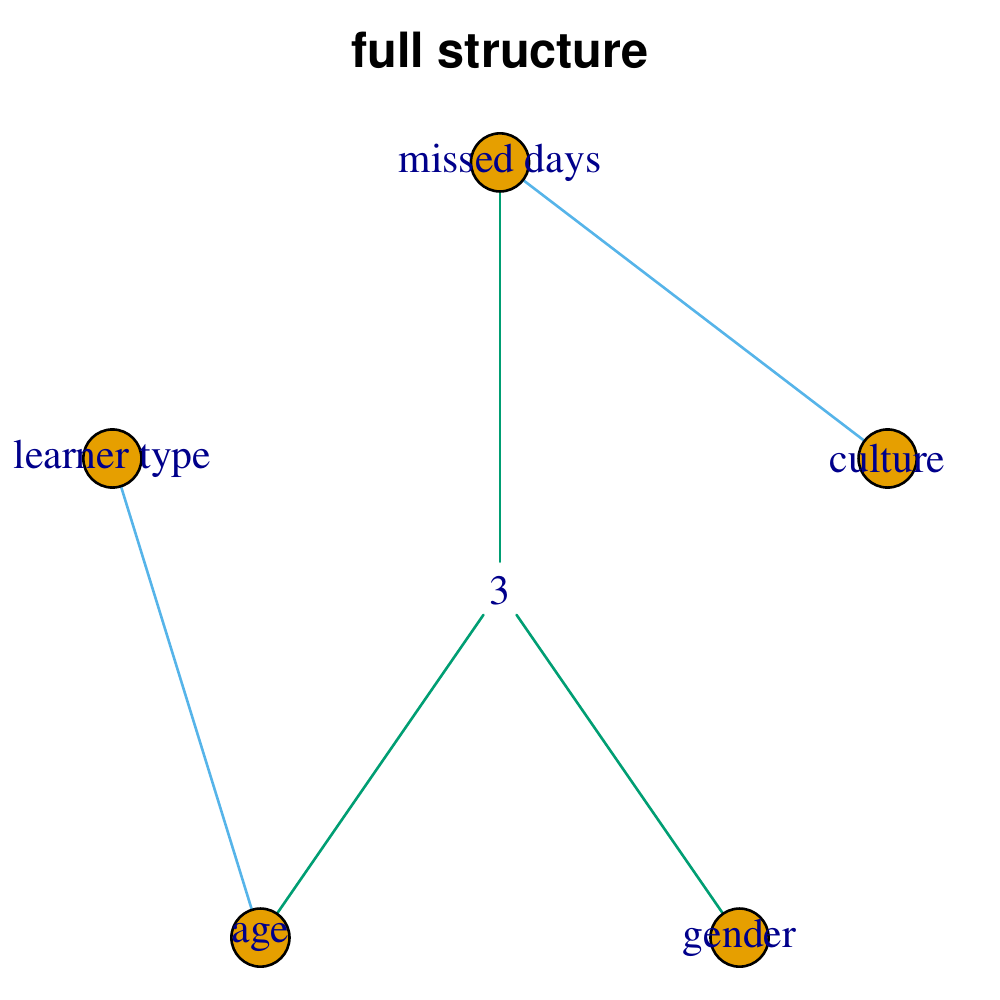}
\caption{Student survey data (Ex.\ \ref{ex:quine}): detected dependence structures.}	
	\label{fig:quine}
\end{figure}
\end{example}

\begin{example}[Decathlon]\label{ex:decathlon}
The results of decathlon athletes from 1985 to 2006 are provided by \cite{Unwi2015}. To consider these (and smaller subsets) as independent samples we only keep the personal best of each athlete, leaving 2709 samples, and order these by the achieved total points in increasing order (the field is denser for lower points, constituting more to the required i.i.d.\ setting for the samples tested). It is well known that the 10 disciplines are dependent, e.g.\ \cite{CoxDunn2002,WoolAnslBidg2007}. We are interested how many samples (using the real measurements of the results in each discipline) are required to detect a dependence using the presented methods: For $2$-multivariance $M_2$ the dependence is, for a significance level of $\alpha = 0.05$, first detected for $N = 5$ and finally for all $N>11$. For total-multivariance $\overline{M}$ it is detected also for all $N > 11.$

Here we used the dependence structure detection based on conservative tests, thus it is interesting to see which structures are detected for various sample sizes, see Figure \ref{fig:decathlon}. The detection of the higher order dependence indicates early on that these variables are dependent, but due to the conservative tests, the actual lower order dependence is missed. With increasing sample size only the dominant pairwise dependencies are detected. Also note that due to the repeated testing and the given significance level the probability of a type I error is large. Using the consistent estimator only some pairwise dependencies are detected for $N = 2706$ and no dependencies are detected for $N \in \{50,100,200\}.$ 

Notably there are some natural variants: 1.\ Instead of the results one could consider the achieved points in each discipline, which are obtained by non linear transformations of the results. This yields almost the same inference. 2.\ Starting with the elite athletes instead of our order causes a change in the detection: In this case $2$-multivariance detects a dependence for all $N>25$ but total-multivariance requires much more samples: $N>177$. Thus here a curse of dimension is at work (compare with Example \ref{ex:averaging}), which might indicate that for top athletes some disciplines are less dependent than for other athletes. We leave further analysis and interpretation to field experts. The setting also naturally yields to clustering methods (for dependent random variables) based on distance multivariance, a topic which is beyond the current paper.

\begin{figure}[H]\centering
\includegraphics[width = 0.22\textwidth]{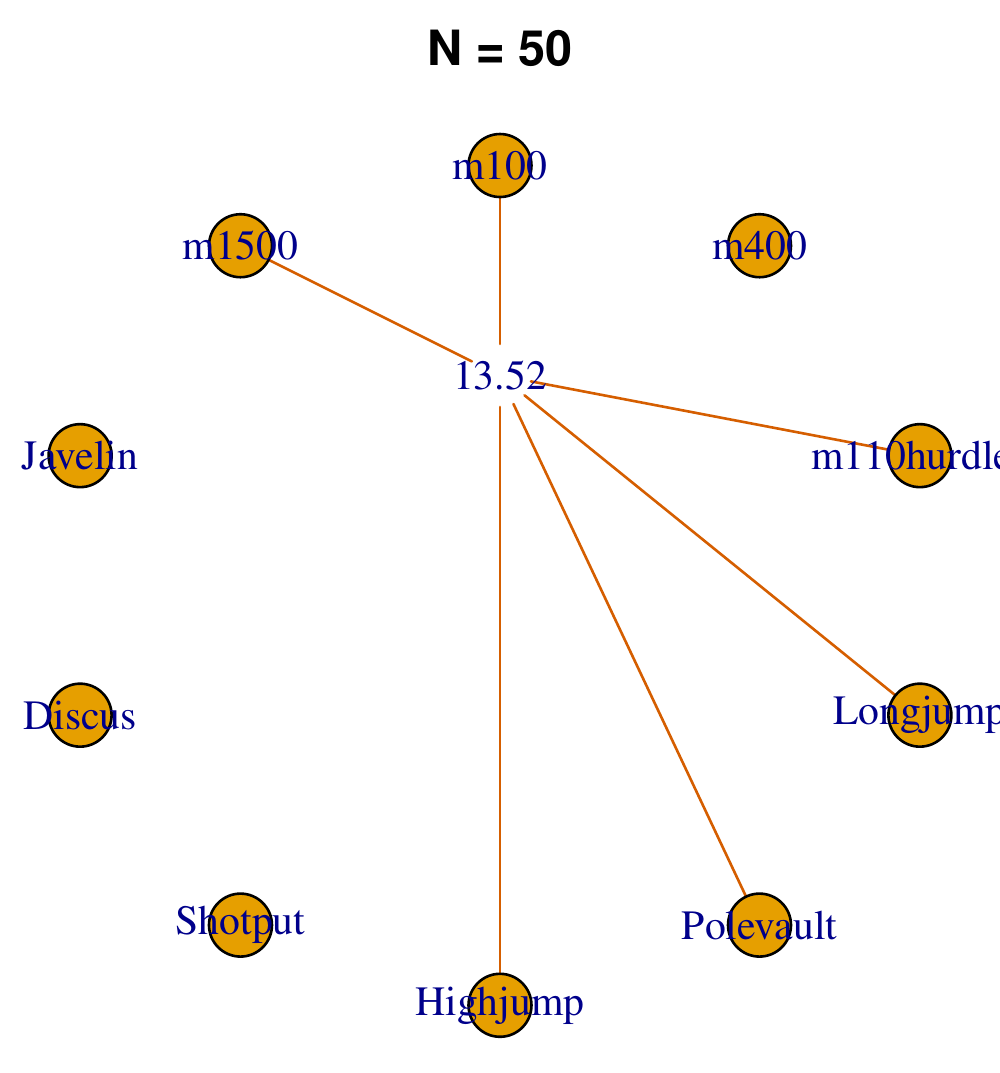}
\quad \includegraphics[width = 0.22\textwidth]{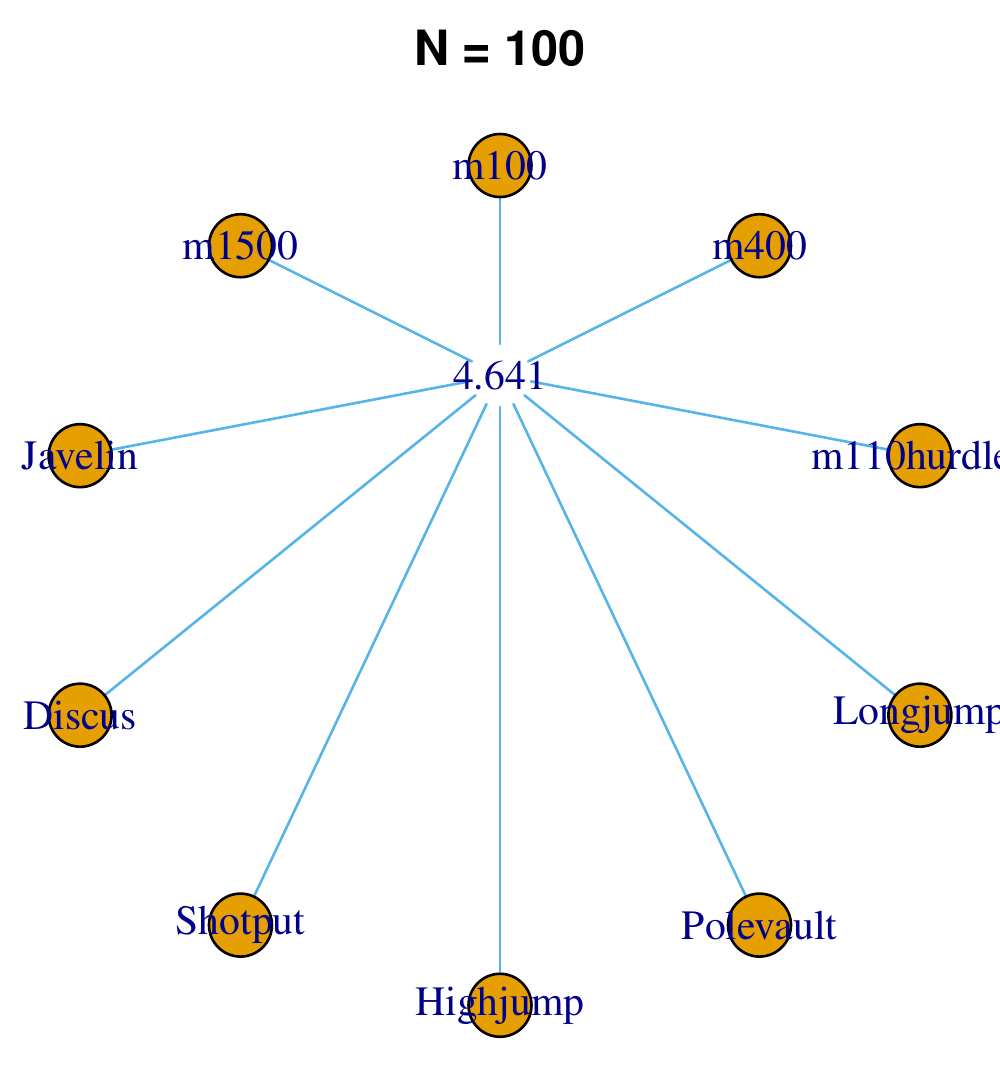}
\quad \includegraphics[width = 0.22\textwidth]{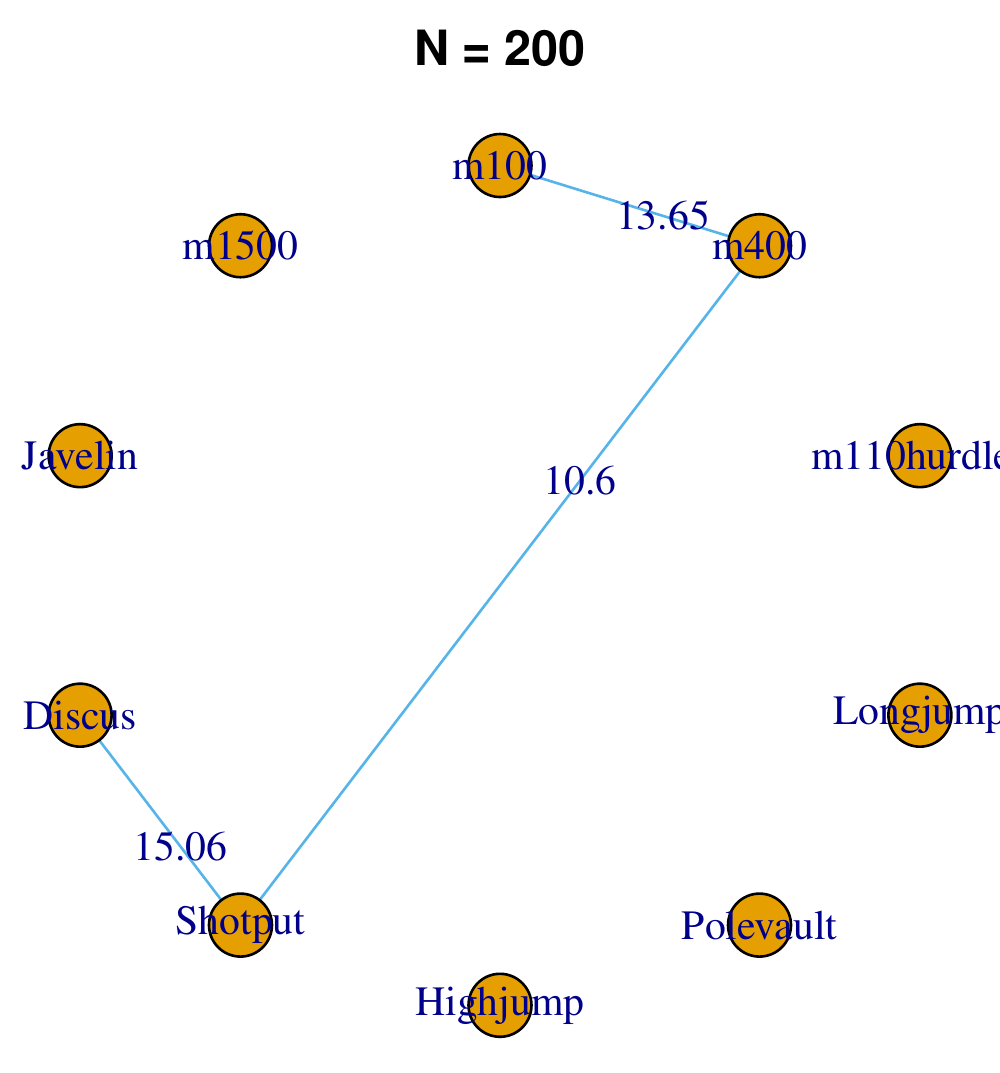}
\quad \includegraphics[width = 0.22\textwidth]{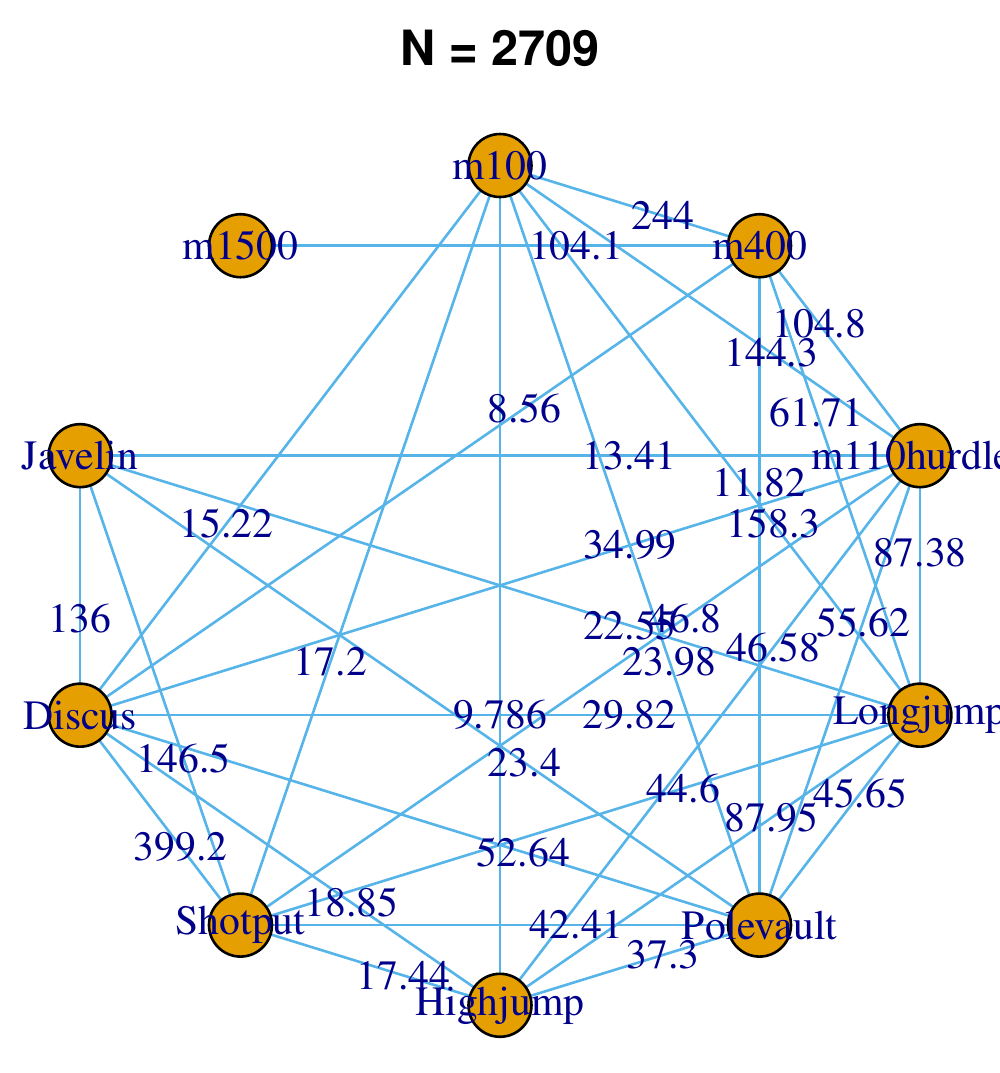}
\caption{Decathlon (Ex.\ \ref{ex:decathlon}): detected dependence structures for 50, 100, 200 and all 2709 samples.}	
	\label{fig:decathlon}
\end{figure}
\end{example}

\section{Further results and proofs} \label{sec:appendix}

Here we collect several results which are essential for (parts of) the previous sections, but which were postponed to this section due to their technicality.

\subsection{A theorem characterizing the support of L\'evy measures}

Note that in \cite[after Defintion 2.3]{BoetKellSchi2018} it was stated that it is unknown how to characterize the (full) support of L\'evy measures in terms of the corresponding continuous negative definite function. The following result provides a characterization (via Proposition \ref{prop:fullsupport}), it is related to \cite[Corollary 2]{ZingKakoKleb1992}.

\begin{theorem}\label{thm:rhochar} Let $\psi(x):= \int_{\R^d} 1-\cos(x t)\,\rho(dt)$ where $\rho$ is a symmetric measure integrating $1\land |.|^2$, and $X$, $Y$ be $\R^d$-valued random vectors with characteristic functions $f_X, f_Y$, and assume $\E(\psi(X))<\infty$ and $\E(\psi(Y))<\infty$. Then 
\begin{equation}
f_X = f_Y \text{ $\rho$-a.s. } \quad \Leftrightarrow \quad \text{ for all }z\in \R^d:\ \E(\psi(X-z))= \E(\psi(Y-z)).
\end{equation}
\end{theorem}
\begin{proof}
Additionally to the stated assumptions let $Z,X',Y'$ be independent random variables which are also independent of $X, Y$ and satisfy $\E(\psi(Z))<\infty$ and $X' \eqd X$, $Y'\eqd Y$. Note that
\begin{align}
 \int 1 - Re(f_X(t)f_Z(-t)) \,\rho(dt) = \iiint 1-\cos((x-z)t)\,\rho(dt)\Prob(X\in dx)\Prob(Z\in dz)
 = \E(\psi(X-Z)) < \infty 
\end{align}
by Tonelli and using the (generalized) triangle inequality for continuous negative definite functions \eqref{eq:psi-tri}. Thus the following implications hold:
\begin{align}
f_X = f_Y \text{ $\rho$-a.s. } \Rightarrow&  \quad  \text{for all }z \in \R^d:\ \int Re((f_X(t)-f_Y(t))\ee^{-\ii zt}) \, \rho(dt) = 0\\
\Leftrightarrow &\quad \text{for all }z \in \R^d:\ \E(\psi(X-z)) = \E(\psi(Y-z)) \\
 \Rightarrow & \quad\E(\psi(X-X')) = \E(\psi(Y-X')) \text{ and }\E(\psi(X-Y')) = \E(\psi(Y-Y')) \\
\Leftrightarrow & \quad \int |f_X(t)|^2 - Re(f_Y(t)f_{X'}(-t))\, \rho(dt) = 0 \quad \text{ and }  \quad \int |f_Y(t)|^2 - Re(f_X(t)f_{Y'}(-t))\, \rho(dt) = 0\\
\Rightarrow & \quad\int |f_X(t)-f_Y(t)|^2\, \rho(dt) = 0
\end{align}
and the last line is equivalent to the start. This completes the proof.
\end{proof}

\subsection{Moment condition} \label{sec:momcon}
In \cite{BoetKellSchi2019} the following condition was used: 
\begin{align}
\text{\textbf{mixed $\psi$-moment condition}: }
\label{d1p} \quad \E\left(\prod_{i=1}^n \psi_i(X_{k_i,i}-X_{l_i,i}')\right) < \infty\text{ for all $k_i,l_i \in \{0,1\}, i =1,\ldots,n$} 
\end{align}
where $\bm{X}_{0},\bm{X}_{0}',\bm{X}_{1},\bm{X}_{1}'$ are independent and have the same marginal distributions as $\bm{X}$ (for the dimensions $d_i$), $\bm{X}_{1},\bm{X}_{1}'$ have also the same joint distribution, but the marginal distributions of $\bm{X}_{0},\bm{X}_{0}'$ are independent (for further details see \cite[Def. 2.3.a]{BoetKellSchi2019}).

We show that for non constant random vectors $X_i$ the joint $\psi$-moment condition \eqref{d1} and \eqref{d1p} are equivalent. If a random vector is constant condition \eqref{d1p} becomes trivial since the corresponding factor is equal to 0.

Recall the (generalized) triangle inequality for any real valued negative definite function $\psi$ \cite[Equation (8)]{BoetKellSchi2018}:
\begin{equation}\label{eq:psi-tri}
\psi(x+y) \leq 2\psi(x) + 2 \psi(y). 
\end{equation}
By this inequality \eqref{d1} implies \eqref{d1p}. For the converse implication we begin with the following observation.

\begin{lemma} \label{lem:momn} For random variables $(X_1,\ldots,X_n)$ the following are equivalent:
\begin{enumerate}[i)]
\item for all $S\subset\{1,\ldots,n\}:$ $ \E \left( \prod_{i\in S} \psi_i(X_{i})\right) < \infty,$
\item for all $S\subset\{1,\ldots,n\}:$ $\E \left( \prod_{i\in S} \psi_i(X_{i}-x_i)\right) < \infty$ for some $(x_1,\ldots,x_n)$,
\item for all $S\subset\{1,\ldots,n\}:$ $\E \left( \prod_{i\in S} \psi_i(X_{i}-\tilde x_i)\right) < \infty$ for all $(\tilde x_1,\ldots,\tilde x_n).$
\end{enumerate}
\end{lemma}
\begin{proof}
Obviously iii) with $\tilde x_i = 0, i=1,\ldots,n$ is i) which implies ii). Finally, iii) follows from ii) by $\psi_i(X_i-\tilde x_i) \leq 2 \psi_i(X_i-x_i) + 2\psi(x_i - \tilde x_i)$ applied to each component. Note that hereto it is essential that the expectations are finite for all subsets $S$.
\end{proof}

Now note that $\E(\psi_i(X_i-X_i')) >0$ for non-constant random variables. Thus the expectations of independent components (i.e., for $k_i = l_i = 0$ in \eqref{d1p}) which factor out in \eqref{d1p} yield strictly positive factors. Therefore, due to the independence of $(X_1,\ldots,X_n)$ and $(X'_1,\ldots,X'_n)$, the condition \eqref{d1p} implies for all $S\subset\{1,\ldots,n\}:$ $\E \left( \prod_{i\in S} \psi_i(X_{i}-x_i)\right) < \infty$ for $\Prob_{(X_1,\ldots,X_n)}$-almost all $(x_1,\ldots,x_n).$ Hence the joint $\psi$-moment condition \eqref{d1} holds by Lemma \ref{lem:momn}.
 
\subsection{Proof of the asymptotics of sample distance multivariance (Theorem \ref{thm:convergence})} \label{sec:asymptoticsthm}
Here we are in the setting of Section \ref{sec:defs}. The asymptotics \eqref{eq:estimator-convergence} and \eqref{eq:totalestimator-convergence} of the test statistic were proved in \cite[Thm.\ 4.5, 4.10, Cor.\ 4.16, 4.18]{BoetKellSchi2019} and \cite[Cor.\ 4.8]{BoetKellSchi2018} under the condition \eqref{eq:mom1-log}. The following theorem provides a proof using an alternative condition. Combining the results yields  the convergence statements \eqref{eq:estimator-convergence} and \eqref{eq:totalestimator-convergence} of Theorem \ref{thm:convergence}.

\begin{theorem} \label{thm:asymp-Vstat}
 Let $X_i,$ $i=1,\ldots,n$ be non-constant random variables such that
 \begin{equation} \label{eq:secondmom}
 \E(\psi_i^2(X_i))< \infty \text{ for all $i =1,\ldots,n$}
 \end{equation}
and let $\bm{X}^{(k)}, k=1,\ldots,N$ be independent copies of $\bm{X}=(X_1,\ldots,X_n)$.
Then
\begin{align}
\label{eq:estimator-convergence-Vstat}N \cdot \hN \Mskript^2(\bm{X}^{(1)},\ldots, \bm{X}^{(N)}) &\xrightarrow[N \to \infty]{d} Q& \quad \text{ if $X_1,\ldots,X_n$ are independent,}\\
\label{eq:totalestimator-convergence-Vstat}\Mfrac{N \cdot \hN \overline{\Mskript}^2(\bm{X}^{(1)},\ldots, \bm{X}^{(N)})}{2^n-n-1}  &\xrightarrow[N \to \infty]{d} \overline{Q} & \quad \text{ if $X_1,\ldots,X_n$ are independent,}
\end{align}
where $Q$ and $\overline{Q}$ are Gaussian quadratic forms with $\E Q = 1 = \E \overline{Q}$. 
\end{theorem} 
\begin{proof}
Let $\bm{X}',\bm{X}^{(k)}, k=1,\ldots,N$ be independent copies of $\bm{X}=(X_1,\ldots,X_n)$ with independent components.
Note, $\hN M^2(\bm{X}^{(1)}, \ldots,\bm{X}^{(N)}) = N^{-2} \sum_{j,k=1}^N \hN\Phi(j,k)$ with $\hN\Phi(j,k) := \hN\Phi_{\{1,\ldots,n\}}(j,k)$ where
\begin{equation}
\begin{split}\label{eq:PhiS}
\hN\Phi_S(j,k) &:= \hN\Phi_S(j,k;\bm{X}^{(1)},\ldots,\bm{X}^{(N)}):= \\
& \prod_{i\in S} \Bigg(-\psi_i(X_i^{(j)}-X_i^{(k)}) + N^{-1}\sum_{m=1}^N \psi_i(X_i^{(j)}-X_i^{(m)}) + N^{-1}\sum_{l=1}^N \psi_i(X_i^{(l)}-X_i^{(k)}) - N^{-2}\sum_{l,m=1}^N \psi_i(X_i^{(l)}-X_i^{(m)})\Bigg).
\end{split}
\end{equation}
Similarly, define $\Phi(\bm{x}^{(j)},\bm{x}^{(k)}) := \Phi_{\{1,\ldots,n\}}(\bm{x}^{(j)},\bm{x}^{(k)})$ with
\begin{equation}
\begin{split}\label{eq:Phi}
\Phi_S(\bm{x}^{(j)},\bm{x}^{(k)}) :=  \prod_{i\in S} \bigg(-\psi_i(x_i^{(j)}-x_i^{(k)}) +\E(\psi_i(x_i^{(j)}-X_i)) + \E( \psi_i(X_i-x_i^{(k)})) - \E(\psi_i(X_i-X_i'))\bigg).
\end{split}
\end{equation}
Then $\E(\Phi(\bm{X},\bm{X})) = \prod_{i=1}^N \E(\psi_i(X_i-X_i'))$ and
\begin{equation} \label{eq:vstatmoms}
\E(\Phi(\bm{x},\bm{X})) = 0,\ \E(\Phi(\bm{X},\bm{X}')) = 0,\ \E(|\Phi(\bm{X},\bm{X})|) < \infty \text{ and }  \E(\Phi(\bm{X},\bm{X}')^2) < \infty.
\end{equation}
where \eqref{eq:secondmom} was used for the bounds.
Therefore $N \cdot N^{-2} \sum_{j,k=1}^N \Phi(\bm{X}^{(j)},\bm{X}^{(k)})$ converges in distribution to a Gaussian quadratic form by \cite[Thm.\ 4.3.2, p.\ 141]{KoroBoro1994}. Note that in the limit in \cite{KoroBoro1994} appear $\E(\Phi(\bm{X},\bm{X}))$ and a sum $\sum_{i=1}^{\infty} \lambda_i$, which cancel in our setting -- this equality also implies that the limit for normalized multivariance has expectation 1, cf.\ \cite[Lemma 2.3 and Remark 4.9.1]{BersBoet2018v2}. Finally, \eqref{eq:estimator-convergence-Vstat} follows by Slutsky's theorem since 
\begin{equation} \label{eq:pnull-slutsky}
 N \cdot N^{-2} \sum_{j,k=1}^N \left(\hN \Phi(j,k)  -  \Phi(\bm{X}^{(j)},\bm{X}^{(k)}) \right) \xrightarrow[N \to \infty]{\Prob} 0.
\end{equation}
To avoid a false impression, note that \eqref{eq:pnull-slutsky} seems natural since the strong law of large numbers implies that the expectations in \eqref{eq:Phi} are approximated by the corresponding sums in \eqref{eq:PhiS}. But the additional factor $N$ in \eqref{eq:pnull-slutsky} makes the proof technical, which we only sketch here: For \eqref{eq:pnull-slutsky} it is, using the Markov inequality, sufficient to show that the second moment of the left hand side converges to 0. This moment and its limit can be calculated explicitly based on and similar to \cite[Theorem 4.15]{BersBoet2018v2}, where the second moment of $\hN \Phi_S(j,k)$ is analyzed in-depth.

Considering analogously $\hN\overline{\Phi}(j,k) := \sum_{\substack{S\subset\{1,\dots,n\}\\ |S|>1}} \hN\Phi_{S}(j,k)$ instead of $\hN \Phi$ yields the result for total multivariance.
\end{proof}

\begin{remark}
Based on the methods developed in the preprint \cite[e.g.\ Section 7.7]{BersBoet2018v2} the second order moment in \eqref{eq:vstatmoms} seems to be already bounded under the weaker assumption $\E(\psi_i(X_i)) < \infty$ for all $i=1,\ldots,n$, see also the special case $\psi_i(x_i)=|x_i|$ discussed in \cite{ChakZhan2019}. To make this rigorous one would have to rewrite (or at least discuss) the steps in \cite{KoroBoro1994} in much more detail, which is beyond the bounds of this paper. Moreover, this clearly also requires a discussion if (and why) the counterexample, which shows that the log moment condition in \eqref{eq:mom1-log} (see also Remark \ref{rem:logmom}) is necessary for the convergence of the empirical characteristic functions, is somehow compensated by the $L^2(\rho)$ norm.
\end{remark}

To prove the divergence in \eqref{eq:estimator-divergence} and \eqref{eq:totalestimator-divergence} we require further notation. Let $\varepsilon>0$ and $\rho_\varepsilon:=\otimes_{i=1}^n\rho_{i,\varepsilon}$ with $\rho_{i,\varepsilon}(.): = \rho_i(.\cap B_{i,\varepsilon}^c)$ where $B_{i,\varepsilon}^c:=\{x\in \R^{d_i} \,:\, |x|>\varepsilon\}$. Note that the corresponding continuous negative definite functions $\psi_{i,\varepsilon}(x_i):= \int 1- \cos(x_i\cdot t_i) \, \rho_{i,\varepsilon}(dt_i)$ are bounded (with non-full support; alternatively one could also use the truncation of \cite[Eq.\ (40)]{BoetKellSchi2018} which preserves the full support). Moreover recall that by \cite[(S.7)]{BoetKellSchi2019-supp}  
\begin{equation}
\hN M_\rho(x^{(1)},\ldots,x^{(N)}) = \sqrt{\int \left| \frac{1}{N} \sum_{j=1}^N \prod_{i=1}^n\left( \ee^{\ii x_i^{(j)}t_i}  - \frac{1}{N} \sum_{k=1}^N  \ee^{\ii x_i^{(k)}t_i} \right) \right|^2 \,\rho(dt)}.
\end{equation}
This and \eqref{def:multi} yield by the monotone convergence theorem:
$\sup_{\varepsilon>0}  \hN M_{\rho_\varepsilon}(\bm{X}^{(1)},\ldots,\bm{X}^{(N)}) = \hN M_{\rho}(\bm{X}^{(1)},\ldots,\bm{X}^{(N)})$  and $\sup_{\varepsilon>0} M_{\rho_\varepsilon}(X_1,\ldots,X_n) = M_{\rho}(X_1,\ldots,X_n).$ Which are the key ingredients for the proof of the following Lemma, which in turn is the key to prove \eqref{eq:estimator-divergence} and \eqref{eq:totalestimator-divergence} without any further moment restrictions. 

\begin{lemma} \label{lem:liminfNM} Let $\bm{X}^{(k)}, k=1,\ldots,N$ be independent copies of $\bm{X}=(X_1,\ldots,X_n)$. Then, without any moment assumptions, we have
\begin{equation}
\liminf_{N\to \infty} \hN M_\rho(\bm{X}^{(1)},\ldots,\bm{X}^{(N)})  \geq  M_\rho(X_1,\ldots,X_n).
\end{equation}
\end{lemma}
\begin{proof} In this proof we omit $(\bm{X}^{(1)},\ldots,\bm{X}^{(N)})$ and $(X_1,\ldots,X_n)$ in the notation. Note that for $\rho_{\varepsilon}$ instead of $\rho$ the joint $\psi$-moment condition \eqref{d1} is always satisfied, and therefore $\lim_{N\to \infty} \hN M_{\rho_\varepsilon} =  M_{\rho_\varepsilon}$ by \cite[Theorem 4.3]{BoetKellSchi2019}. Thus 
\begin{equation*}
M_\rho = \sup_{\varepsilon>0} M_{\rho_{\varepsilon}} = \sup_{\varepsilon>0} \lim_{N\to \infty} \hN M_{\rho_{\varepsilon}} 
= \sup_{\varepsilon>0} \liminf_{N\to \infty} \hN M_{\rho_{\varepsilon}} \leq \liminf_{N\to \infty} \sup_{\varepsilon>0}  \hN M_{\rho_{\varepsilon}} =  \liminf_{N\to \infty} \hN M_{\rho}.  \qedhere
\end{equation*}
\end{proof}

Now the proof of the divergence \eqref{eq:estimator-divergence} is identical to \cite[Thm.\ 4.5.b]{BoetKellSchi2019} just replacing \cite[Theorem 4.3]{BoetKellSchi2019} by Lemma \ref{lem:liminfNM}: If the random variables are $(n-1)-$independent but dependent Theorem \ref{thm:indep} implies $M > 0$. Thus $N \cdot \hN M $ diverges for $N \to \infty$ by Lemma \ref{lem:liminfNM}. This also applies in the case of dependence to at least one summand of total multivariance (and the others are non negative) therefore also the divergence \eqref{eq:totalestimator-divergence} follows.

\subsection{The population representation of \cite[Lemma 1a]{FanMichPeneSalo2017}} \label{sec:fanstruct}
Note that the $\gamma,\beta$ terms appearing in \cite[Lemma 1a]{FanMichPeneSalo2017} correspond in our notation to $\gamma_{j,l} = \psi_l(x_l^{(j)}),\ \gamma_{j,j',l} = \psi_l(x_l^{(j)} - x_l^{(j')})$ and
\begin{equation} \label{eq:fanphi}
\beta_{j,j',l} = \psi_l(x_l^{(j)}) + \psi_l(x_l^{(j)}) - \psi_l(x_l^{(j)}-x_l^{(j')})=: \phi_l(x_l^{(j)},x_l^{(j')}).
\end{equation}
Turning their sample sums into expectations and observing the independence implied by the indices yields the population version of their $T_n$: 
\begin{align}
\sum_{\substack{S\subset\{1,\dots,n\}\\ |S|>0}} \sum_{\substack{S'\subset\{1,\dots,n\}\\ |S'|>0}} (-1)^{|S|+|S'|} \Bigg[& \E\left( \prod_{l \in S \cap S'} \phi_l(X_l,X_l') \, \prod_{l \in S \backslash S'}  \psi_l(X_l) \prod_{l \in S' \backslash S}  \psi_l(X_l')\right)\\
&- 2 \E\left( \prod_{l \in S \cap S'} \E(\phi_l(X_l,X_l') \mid X_l) \, \prod_{l \in S \backslash S'}  \psi_l(X_l) \prod_{l \in S' \backslash S} \E\left(\psi_l(X_l')\right)\right)\\
&+ \prod_{l \in S \cap S'} \E(\phi_l(X_l,X_l')) \, \prod_{l \in S \backslash S'}  E( \psi_l(X_l)) \prod_{l \in S' \backslash S} \E (\psi_l(X_l')) \Bigg]
\end{align}
now note that $(-1)^{|S|+|S'|} = (-1)^{|S\backslash S'|+|S' \backslash S|}$ can be distributed as factor $-1$ to each factor in the products corresponding to $S\backslash S'$ and $S' \backslash S$. Then the formula 
\begin{equation}
\begin{split}
\sum_{\substack{S\subset\{1,\dots,n\}\\ |S|>0}} \sum_{\substack{S'\subset\{1,\dots,n\}\\ |S'|>0}} \prod_{l \in S \cap S'} a_l \prod_{l \in S \backslash S'}  (-b_l) \prod_{l \in S' \backslash S} (-c_l) = \prod_{i=1}^n(1 + a_i - b_i - c_i ) - \prod_{i=1}^n (1-b_i) - \prod_{i=1}^n (1-c_i) + 1 
\end{split}
\end{equation}
yields
\begin{align}
&\E\left(\prod_{i=1}^n \left(1+\phi(X_i,X_i') -  \psi_i(X_i) -  \psi_i(X_i')\right)  - \prod_{i=1}^n ( 1-  \psi_i(X_i)) -  \prod_{i=1}^n (1 -  \psi_i(X_i')) + 1\right)\\ 
\label{eq:fan2}&-2 \E\Bigg(\prod_{i=1}^n \left( 1 + \E(\phi(X_i,X_i') \mid X_i) -  \psi_i(X_i) - \E(\psi_i(X_i'))\right) - \prod_{i=1}^n ( 1-  \psi_i(X_i)) -  \prod_{i=1}^n ( 1-   \E(\psi_i(X_i'))) + 1\Bigg)\\
\label{eq:fan3}& + \prod_{i=1}^n (1 + \E(\phi(X_i,X_i')) -  \E(\psi_i(X_i)) -  \E(\psi_i(X_i'))) - \prod_{i=1}^n ( 1-  \E(\psi_i(X_i))) -  \prod_{i=1}^n (1 -  \E(\psi_i(X_i'))) + 1.
\end{align}
Finally, after using the linearity of the expectation, the last two products in the first row cancel with the second product in \eqref{eq:fan2}, and the third product in \eqref{eq:fan2} cancels with the last two products in \eqref{eq:fan3}; also the trailing ''$+1$'' cancel. For the remaining products the linearity of the expectation and the definition of $\phi$ in \eqref{eq:fanphi} yield
\begin{equation}
\E\left(\prod_{i=1}^n (1-\psi_i(X_i-X_i')\right) - 2 \E\left(\prod_{i=1}^n \E(1-\psi_i(X_i-X_i') \mid X_i)\right) + \prod_{i=1}^n \E(1-\psi_i(X_i-X_i')).
\end{equation}

\subsection{The difference of dHSIC and total multivariance for $n=3$} \label{sec:diffHSICtm}
Expanding the product in \eqref{eq:M} yields by careful accounting the representation:
\begin{align}
\label{eq:M3rep1a}M(X_1,X_2,X_3)= & - \E\left(\prod_{i=1}^3 \psi_i(X_i-X_i')\right)  - 4 \E\left( \prod_{i=1}^3 \E(\psi_i(X_i-X_i') \mid X_i)\right)- 4 \prod_{i=1}^3 \E(\psi_i(X_i-X_i'))\\
 &+ \sum_{(i,j,k) \in \pi(1,2,3)} \Big[  \E\left(\psi_i(X_i-X_i' \mid X_i) \psi_j(X_j-X_j') \psi_k(X_k-X_k')\right)\\
& \phantom{+ \sum_{(i,j,k) \in \pi(1,2,3)} \Big[}-\frac{1}{2} \E(\psi_i(X_i-X_i')) \,\E\left(\psi_j(X_j-X_j') \psi_k(X_k-X_k')\right)\\
& \phantom{+ \sum_{(i,j,k) \in \pi(1,2,3)} \Big[}-  \E\left(  \E(\psi_i(X_i-X_i')\mid X_i)\psi_j(X_j-X_j') \E(\psi_i(X_k-X_k') \mid X_k')\right) \\
\label{eq:M3rep5}& \phantom{+ \sum_{(i,j,k) \in \pi(1,2,3)} \Big[}+ 2 \E(\psi_i(X_i-X_i'))\, \E\left( \E(\psi_j(X_j-X_j')\mid X_j) \E(\psi_k(X_k-X_k') \mid X_k)\right)\Big],
\end{align}
where $\pi(1,2,3)$ is the set of all permutations of the vector $(1,2,3).$ Define 
\begin{align}
&H_k(X_1,\ldots,X_n) := \sum_{\substack{S\subset\{1,\dots,n\}\\ |S|=k}} \Bigg[ \E\left(\prod_{i\in S}(-\psi_i(X_i-X_i'))\right) - 2 \E \left(\prod_{i\in S}\E\left(-\psi_i(X_i-X_i')\mid X_i\right)\right) + \prod_{i\in S} \E\left(-\psi_i(X_i-X_i')\right)\Bigg]
\end{align}
and note $H_0 = H_1 = 0$ and $H_2(X_1,\ldots,X_n) = M_2(X_1,\ldots,X_n)$ where $M_2$ is 2-multivariance defined in \eqref{def:2mmulti}.  Using 
$
\prod_{i=1}^n (1- \alpha_i) = 1 + \sum_{k=1}^n \sum_{\substack{S\subset\{1,\dots,n\}\\ |S|=k}} \prod_{i\in S} (-\alpha_i)
$
one finds for arbitrary $n$ that dHSIC is equal to $\sum_{k=2}^n H_k$. Thus, recalling that $\overline M(X_1,\ldots,X_n) = \sum_{k=2}^n M_k(X_1,\ldots,X_n)$ and $M_n(X_1,\ldots,X_n) = M(X_1,\ldots,X_n)$, we find for $n=3$ 
\begin{align}
\label{eq:diffdhsictm} dHSIC(X_1,X_2,X_3) - \overline{M}(X_1,X_2,X_3) = H_2(\ldots) + H_3(\ldots) - M_2(\ldots) - M(\ldots)= H_3(X_1,X_2,X_3) - M(X_1,X_2,X_3).  
\end{align}
Thus the difference in \eqref{eq:diffdhsictm} has almost the same representation as given in \eqref{eq:M3rep1a}-\eqref{eq:M3rep5}, only in \eqref{eq:M3rep1a} the factors change. We did not succeed to find any simplified representation of the remaining terms which would allow a useful distinction. Obviously the values of the measures differ, but it remains an open problem if based on this difference one of the measures should be preferred.

\section{Supplement}
The supplementary Section \ref{sec:supp-example} can be found at pages \pageref{sec:supp-example} ff.\ of the current manuscript.

\section*{Acknowledgments}

We are grateful for detailed comments by Georg Berschneider and some remarks by R\'ene Schilling and Martin Bilodeau. Moreover, we want to thank several referees for their the detailed reports, which helped to bring the paper to its present form.


\newpage

\setcounter{section}{8}
\phantomsection\label{sec:supp-example} 
\section{Supplement - collection of examples}
The examples are arranged in several subsections: \ref{sec:ex-higher-order} discusses dependencies of higher order, \ref{sec:ex-properties} illustrates various properties of multivariance. A comparisons of multivariance with other dependence measures and real data examples can be found in the main body of the paper, Sections \ref{sec:ex-compare} and \ref{sec:ex-real-data}.

If not mentioned otherwise: We use the Euclidean distance $\psi_i(x_i)=|x_i|$, $L = 300$ repetitions for the resampling tests (Tests \ref{testA}, \ref{testB}, \ref{testC} with \eqref{eq:Rresampling}), sample sizes $N\in\{10,20,30,40,50,60,70,80,90,100\}$, significance level $\alpha = 0.05$ and $1000$ runs to compute the empirical power of the tests. 

The clustered dependence structure is detected using the conservative tests with significance level $\alpha$ (with Holm's correction for multiple tests) or using the consistent estimators of Corollary \ref{cor:consistent-dep-struct} with $\beta = \frac{1}{2}$ and $C= 2.$ 

In the power tables the normalized multivariances $\Mskript$, $\overline{\Mskript},$ $\Mskript_2$ and $\Mskript_3$ are studied using the distribution-free and the resampling tests. To avoid overloading we included only a selection of the test results in the tables. Hereto note that using $\Mskript_3$ as a test, should (to provide a test which is consistent against all alternatives) be preceded by a test for pairwise independence, e.g.\ using $\Mskript_2.$ Here we performed these tests independently. But we also include in the tables a test 'Comb' which combines the tests based on $\Mskript_2$, $\Mskript_3$ and for $n>3$ also $\overline{\Mskript}$ to a global test for the same significance level (rejecting independence if at least one p-value, adjusted by Holm's method, is significant).

For most examples the clustered dependence structure is illustrated using the test based scheme of Section \ref{sec:algo} (with conservative p-value). Explicit values in the graphs are based on a successful detection with $N=100$ samples, if not stated otherwise.

\subsection{Detection and visualization of higher order dependencies} \label{sec:ex-higher-order}

The generation of samples with higher order dependencies is explained by a detailed description of the two classical examples (Examples \ref{ex:tetra} and \ref{ex:2coins}) and a basic example for dependence of arbitrary order (Example \ref{ex:ncoins}). These provide reference examples to detect and build more involved dependence structures, which also illustrate different aspects of higher order dependence: higher order dependencies with continuous marginal distributions (Example \ref{ex:2coinsnorm}), disjoint clusters (Example \ref{ex:several}), a mixture of pairwise and higher order dependence (Example \ref{ex:star}), iterated dependencies (Example \ref{ex:iterated}) and joint dependence of all variables (such that all are connected by dependencies of higher order) without any pairwise dependence (Example \ref{ex:ring}). The full dependence structures for the examples are collected in Example \ref{ex:full}.

\begin{example}[Colored tetrahedron]\label{ex:tetra}
Consider a dice shaped as a tetrahedron with sides colored red, green, blue and stripes of all three colors on the fourth side. The events that a particular color is on the bottom side -- when throwing this dice -- are pairwise independent events. But they are not independent. Both properties follow by direct calculation:
\begin{align*}
&\Prob(\text{red}) = \Prob(\text{green}) = \Prob(\text{blue}) = \frac{2}{4}\\
&\Prob(\text{red and green}) = \Prob(\text{red and blue}) = \Prob(\text{green and blue}) = \frac{1}{4}\\
&\Prob(\text{red})\Prob(\text{green}) = \Prob(\text{red})\Prob(\text{blue}) = \Prob(\text{green})\Prob(\text{blue}) = \frac{1}{4}\\
&\Prob(\text{red and green and blue}) = \frac{1}{4} \neq \frac{1}{8} = \Prob(\text{red})\Prob(\text{green})\Prob(\text{blue}).
\end{align*}
Thus this provides an example of three variables which are $2$-independent, but dependent. In Figure \ref{fig:tetra} the empirical powers of the tests are denoted. Maybe it seems surprising that the empirical power of the test based dependence structure detection is not 1 albeit the others have power 1. Hereto recall that the distribution-free test is sharp for Bernoulli random variables, thus (due to the correction for multiple tests) it is expected that in $5\%$ of the cases already a (false) detection of pairwise dependence occurs. Furthermore, note that the distribution-free test for the normalized total multivariance has for $N=10$ an empirical power of 0 due to averaging (see also Example \ref{ex:averaging}).

\begin{figure}[H]
\includegraphics[width = 0.3\textwidth]{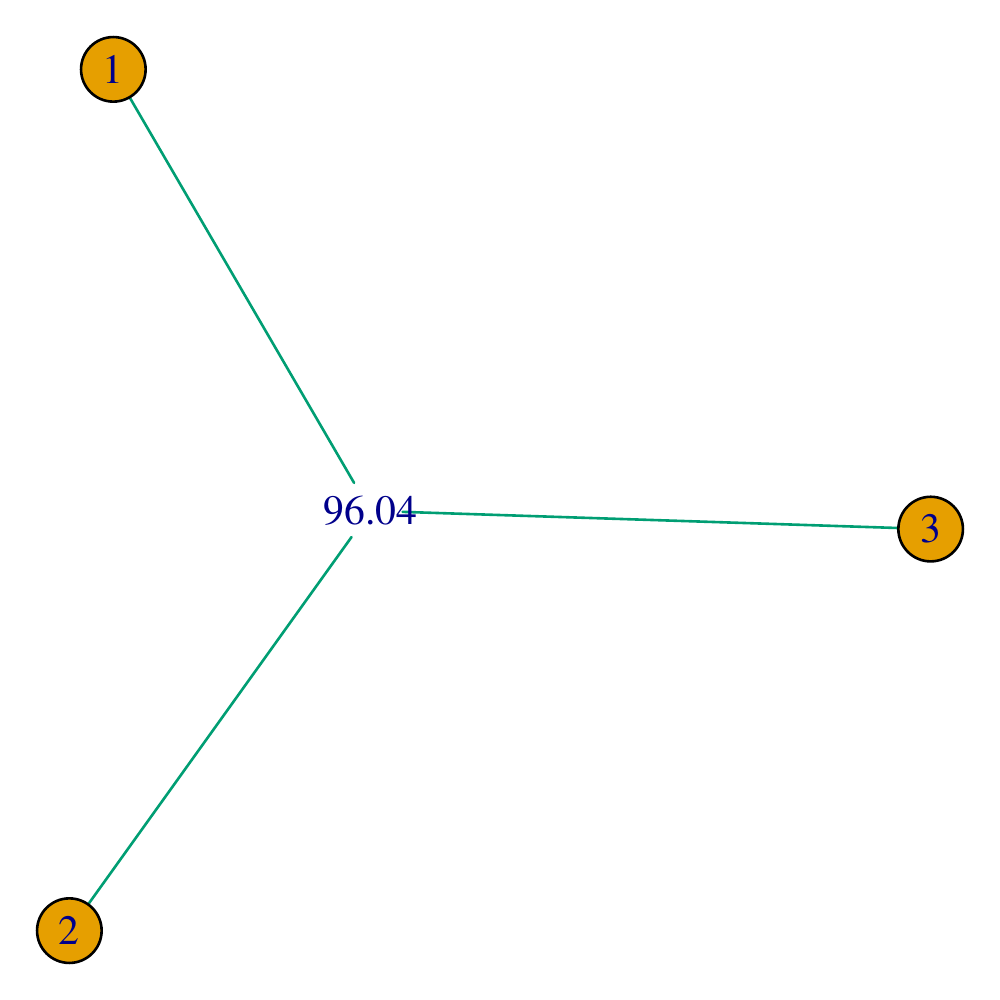}
\quad
\begin{tabular}[b]{r|cc|c|cc|c|c}
  & \multicolumn{3}{c|}{resampling}& \multicolumn{3}{c|}{distribution-free}& consistent\\
$N$ & $\hN \Mskript$ & $\hN \overline{\Mskript}$ & Comb & $\hN \Mskript$ & $\hN \overline{\Mskript}$ & detection & detection \\ 
  \hline
10 & 0.891 & 0.854 & 0.900 & 0.926 & 0.000 & 0.928 & 0.726 \\ 
  20 & 1.000 & 1.000 & 1.000 & 1.000 & 1.000 & 0.955 & 0.991 \\ 
  30 & 1.000 & 1.000 & 1.000 & 1.000 & 1.000 & 0.958 & 0.999 \\ 
  40 & 1.000 & 1.000 & 1.000 & 1.000 & 1.000 & 0.943 & 0.999 \\ 
  50 & 1.000 & 1.000 & 1.000 & 1.000 & 1.000 & 0.960 & 0.999 \\ 
  60 & 1.000 & 1.000 & 1.000 & 1.000 & 1.000 & 0.958 & 1.000 \\ 
  70 & 1.000 & 1.000 & 1.000 & 1.000 & 1.000 & 0.958 & 1.000 \\ 
  80 & 1.000 & 1.000 & 1.000 & 1.000 & 1.000 & 0.943 & 1.000 \\ 
  90 & 1.000 & 1.000 & 1.000 & 1.000 & 1.000 & 0.952 & 1.000 \\ 
  100 & 1.000 & 1.000 & 1.000 & 1.000 & 1.000 & 0.949 & 1.000 \\ 
   \hline
\end{tabular}

\caption{Colored tetrahedron (Ex.\ \ref{ex:tetra}): dependence structure and empirical power.}	
	\label{fig:tetra}
\end{figure}
\end{example}

\begin{example}[Two coins --- three events] \label{ex:2coins}
Throw two fair coins and consider the three events: the first shows head, the second shows tail, both show the same. Then again a direct calculation shows pairwise independence, but dependence. The probability that all three events occur simultaneously is 0.

Alternatively the same (but with a joint probability of $1/4$ as in Example \ref{ex:tetra}) holds for the events: the first shows head, the second shows head, both show the same.

Figure \ref{ex:2coins} shows the dependence structure and empirical power for the case with joint probability 0. The results are indistinguishable from Example \ref{ex:tetra}.

\begin{figure}[H]
\centering
\includegraphics[width = 0.3\textwidth]{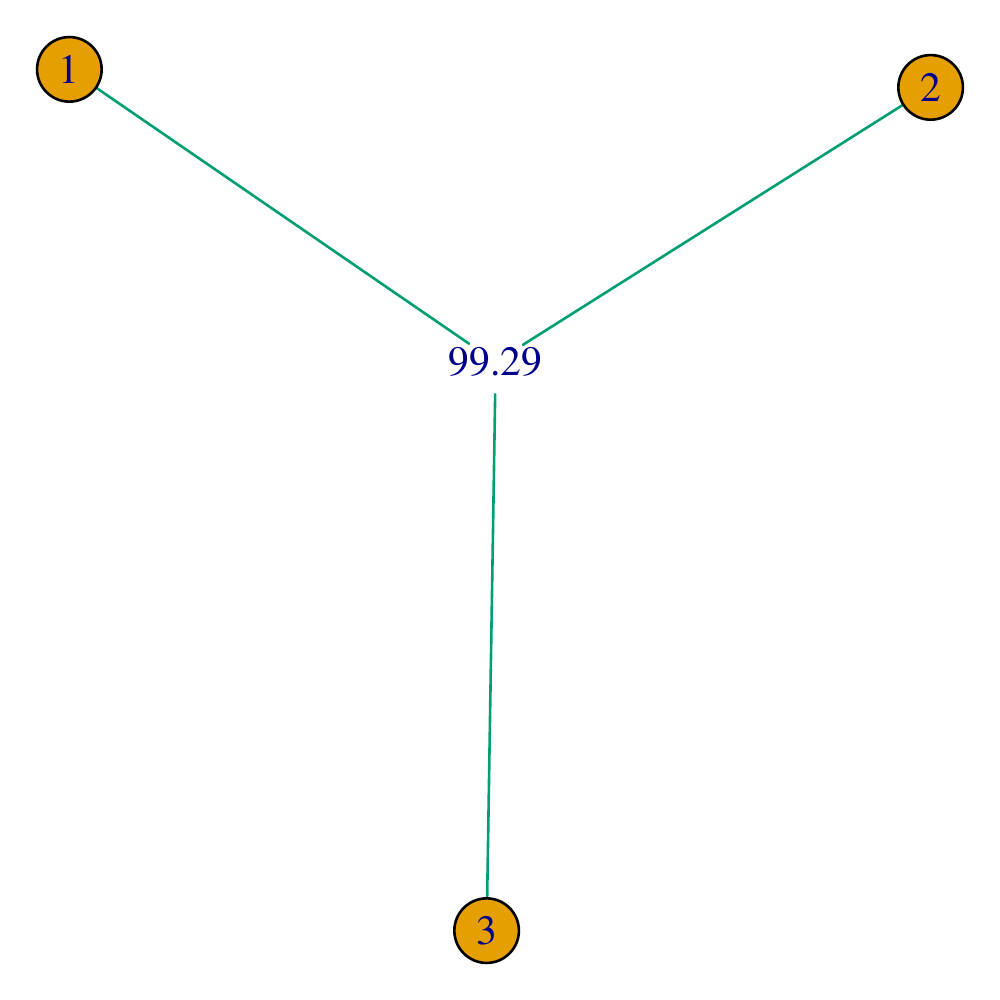}
\quad
\begin{tabular}[b]{r|cc|c|cc|c|c}
  & \multicolumn{3}{c|}{resampling}& \multicolumn{3}{c|}{distribution-free}& consistent\\
$N$ & $\hN \Mskript$ & $\hN \overline{\Mskript}$ & Comb & $\hN \Mskript$ & $\hN \overline{\Mskript}$ & detection & detection \\ 
  \hline
10 & 0.913 & 0.862 & 0.918 & 0.929 & 0.000 & 0.933 & 0.714 \\ 
  20 & 1.000 & 1.000 & 1.000 & 1.000 & 1.000 & 0.957 & 0.988 \\ 
  30 & 1.000 & 1.000 & 1.000 & 1.000 & 1.000 & 0.960 & 0.997 \\ 
  40 & 1.000 & 1.000 & 1.000 & 1.000 & 1.000 & 0.961 & 0.999 \\ 
  50 & 1.000 & 1.000 & 1.000 & 1.000 & 1.000 & 0.948 & 1.000 \\ 
  60 & 1.000 & 1.000 & 1.000 & 1.000 & 1.000 & 0.961 & 1.000 \\ 
  70 & 1.000 & 1.000 & 1.000 & 1.000 & 1.000 & 0.951 & 1.000 \\ 
  80 & 1.000 & 1.000 & 1.000 & 1.000 & 1.000 & 0.948 & 1.000 \\ 
  90 & 1.000 & 1.000 & 1.000 & 1.000 & 1.000 & 0.954 & 1.000 \\ 
  100 & 1.000 & 1.000 & 1.000 & 1.000 & 1.000 & 0.962 & 0.999 \\ 
   \hline
\end{tabular}

\caption{Three events of two coins (Ex.\ \ref{ex:2coins}): dependence structure and empirical power.}		\label{fig:2coins}
\end{figure}
\end{example}

A simple generalization yields the next important example, featuring higher order dependence in its 'purest' form.

\begin{example}[$n$ coins ---  $(n+1)$ events]\label{ex:ncoins}
Throw $n$ fair coins and consider the $n+1$ events: The first shows head, the second shows head, ..., the $n$-th shows head, there is an odd number of heads. Then by direct calculation these are $n$-independent, but dependent (the joint probability of the events is $0$ for even $n$ and it is $(1/2)^n$ for odd $n$). To get an intuition, note that given $n$ of these events one can directly calculate the $(n+1)$th event. But given less, provides not enough information to determine any further event - any option is equally likely.

Figure \ref{fig:ncoins} shows the dependence structure and the empirical power of the dependence measures. The total multivariance suffers a loss of power compared to the previous examples due to the averaging (only one of the $2^n-n-1$ summands diverges, see also Example \ref{ex:averaging}). Moreover one starts to see that the distribution-free method is conservative for total multivariance (recall that also with univariate Bernoulli marginals it is only sharp for multivariance, not for total multivariance). The low detection rate of the test based dependence structure detection is again due to the sharp rejection level for Bernoulli random variables and the p-value adjustment due to multiple testing of all $k$-tuples for each $k\in\{2,\ldots,n+1\}.$  

\begin{figure}[H]
\includegraphics[width = 0.3\textwidth]{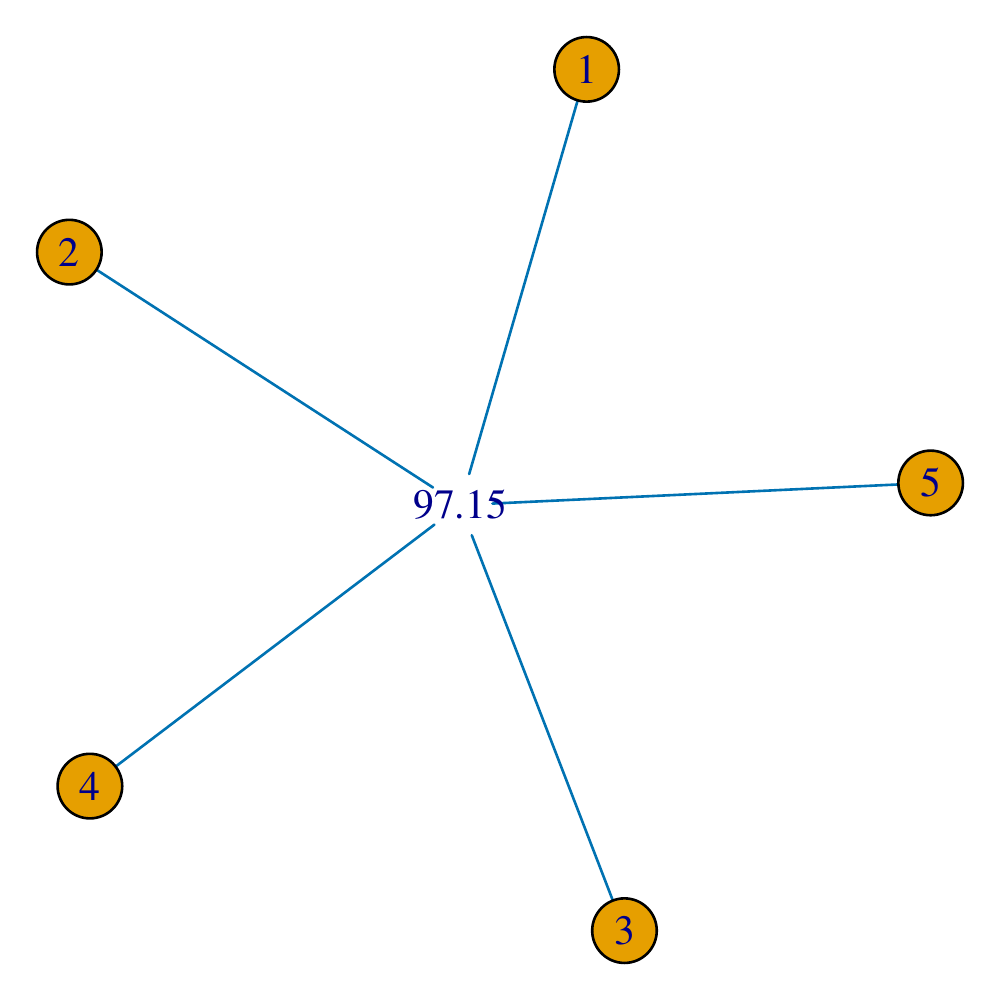}
\begin{tabular}[b]{r|cc|c|cc|c|c}
  & \multicolumn{3}{c|}{resampling}& \multicolumn{3}{c|}{distribution-free}& consistent\\
$N$ & $\hN \Mskript$ & $\hN \overline{\Mskript}$ & Comb & $\hN \Mskript$ & $\hN \overline{\Mskript}$ & detection & detection \\ 
  \hline
10 & 0.814 & 0.165 & 0.076 & 0.792 & 0.001 & 0.755 & 0.348 \\ 
  20 & 0.999 & 0.610 & 0.359 & 1.000 & 0.000 & 0.915 & 0.938 \\ 
  30 & 1.000 & 0.961 & 0.812 & 1.000 & 0.000 & 0.902 & 0.990 \\ 
  40 & 1.000 & 1.000 & 0.997 & 1.000 & 0.001 & 0.905 & 0.993 \\ 
  50 & 1.000 & 1.000 & 1.000 & 1.000 & 0.007 & 0.896 & 0.999 \\ 
  60 & 1.000 & 1.000 & 1.000 & 1.000 & 0.040 & 0.905 & 1.000 \\ 
  70 & 1.000 & 1.000 & 1.000 & 1.000 & 0.137 & 0.906 & 1.000 \\ 
  80 & 1.000 & 1.000 & 1.000 & 1.000 & 0.453 & 0.900 & 0.999 \\ 
  90 & 1.000 & 1.000 & 1.000 & 1.000 & 0.901 & 0.886 & 0.999 \\ 
  100 & 1.000 & 1.000 & 1.000 & 1.000 & 1.000 & 0.909 & 1.000 \\ 
   \hline
\end{tabular}

\medskip

\includegraphics[width = 0.3\textwidth]{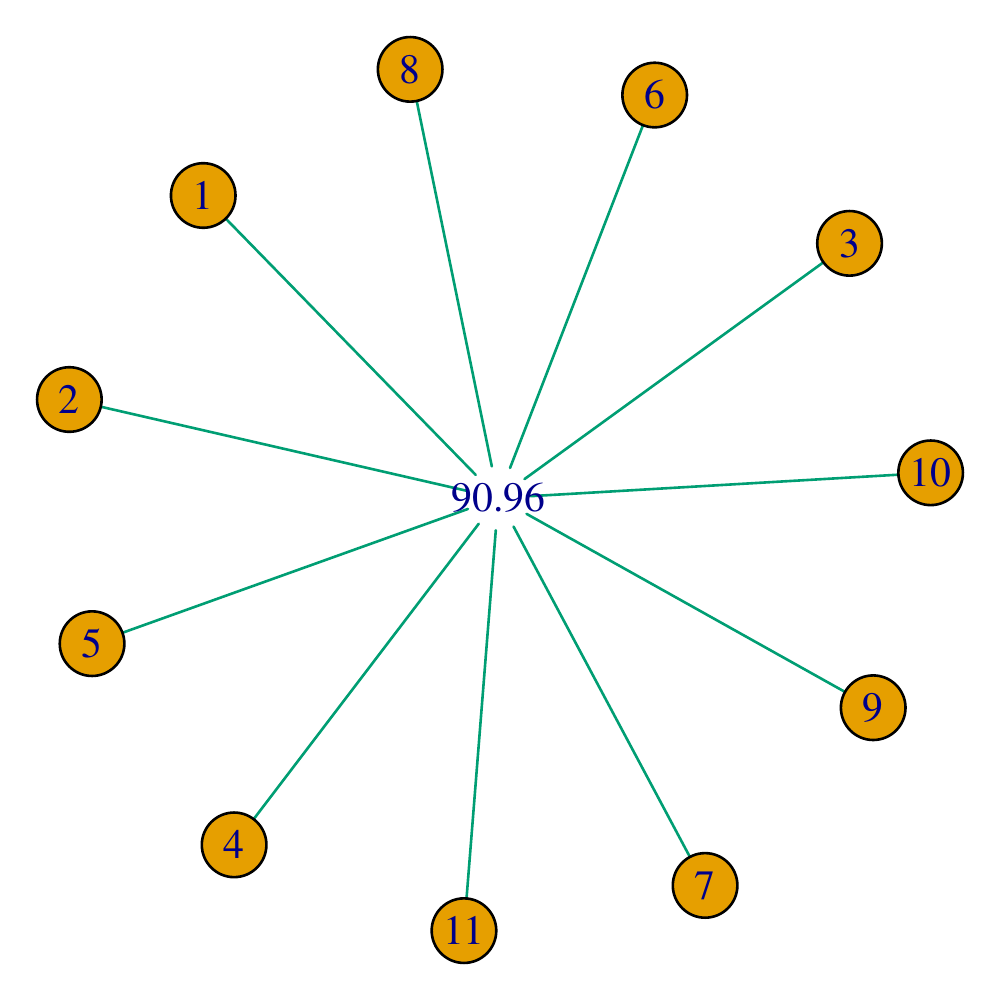}
\begin{tabular}[b]{r|cc|c|cc|c|c}
  & \multicolumn{3}{c|}{resampling}& \multicolumn{3}{c|}{distribution-free}& consistent\\
$N$ & $\hN \Mskript$ & $\hN \overline{\Mskript}$ & Comb & $\hN \Mskript$ & $\hN \overline{\Mskript}$ & detection & detection \\ 
  \hline
10 & 0.353 & 0.045 & 0.046 & 0.340 & 0.002 & 0.294 & 0.001 \\ 
  20 & 0.987 & 0.056 & 0.044 & 0.991 & 0.000 & 0.895 & 0.065 \\ 
  30 & 1.000 & 0.075 & 0.059 & 1.000 & 0.000 & 0.805 & 0.385 \\ 
  40 & 1.000 & 0.114 & 0.077 & 1.000 & 0.000 & 0.805 & 0.671 \\ 
  50 & 1.000 & 0.155 & 0.091 & 1.000 & 0.000 & 0.749 & 0.834 \\ 
  60 & 1.000 & 0.179 & 0.102 & 1.000 & 0.000 & 0.751 & 0.905 \\ 
  70 & 1.000 & 0.220 & 0.122 & 1.000 & 0.000 & 0.731 & 0.949 \\ 
  80 & 1.000 & 0.261 & 0.151 & 1.000 & 0.000 & 0.760 & 0.975 \\ 
  90 & 1.000 & 0.291 & 0.160 & 1.000 & 0.000 & 0.763 & 0.983 \\ 
  100 & 1.000 & 0.337 & 0.226 & 1.000 & 0.000 & 0.738 & 0.990 \\ 
   \hline
\end{tabular}

\caption{Events of 4 and 10 coins (Ex.\ \ref{ex:ncoins}): dependence structure and empirical power.}		\label{fig:ncoins}
\end{figure}

\end{example}

The previous examples only used dichotomous data. Obviously the same dependence structures can also appear (and be detected) for other marginal distributions. A basic example is the following.

\begin{example}[Perturbed coins] \label{ex:2coinsnorm}
Let $(Y_1,Y_2,Y_3)$ be the random variables corresponding to the events of $n = 2$ coins in Example \ref{ex:ncoins} and $Z_1,Z_2,Z_3$ be i.i.d.\  standard normal random variables. Now set $X_i := Y_i + r Z_i$ for $i=1,2,3$ and some fixed $r\in\R.$ For these the same dependence structure as in Example \ref{ex:2coins} (Figure \ref{fig:2coins}) is detected. Figure \ref{fig:2coinsnorm} shows the dependence structure and the empirical power for $r\in\{0.25,0.5,0.75,1\}$. Note that the rate of the detection of  the test based dependence structure algorithm improves in comparison to the previous examples (for $N$ large) whereas the consistent estimator requires larger samples. The former is due to the fact that only in the case of univariate Bernoulli distributed random variables the distribution-free method is sharp for multivariance. In all other cases it becomes conservative and therefore the rate of falsely detected pairwise dependencies is reduced. Increasing the value of $r$ reduces the empirical power. This is expected, since the dependence structure becomes blurred by the variability of the $Z_i$'s. 

\begin{figure}[H]
\includegraphics[width = 0.3\textwidth]{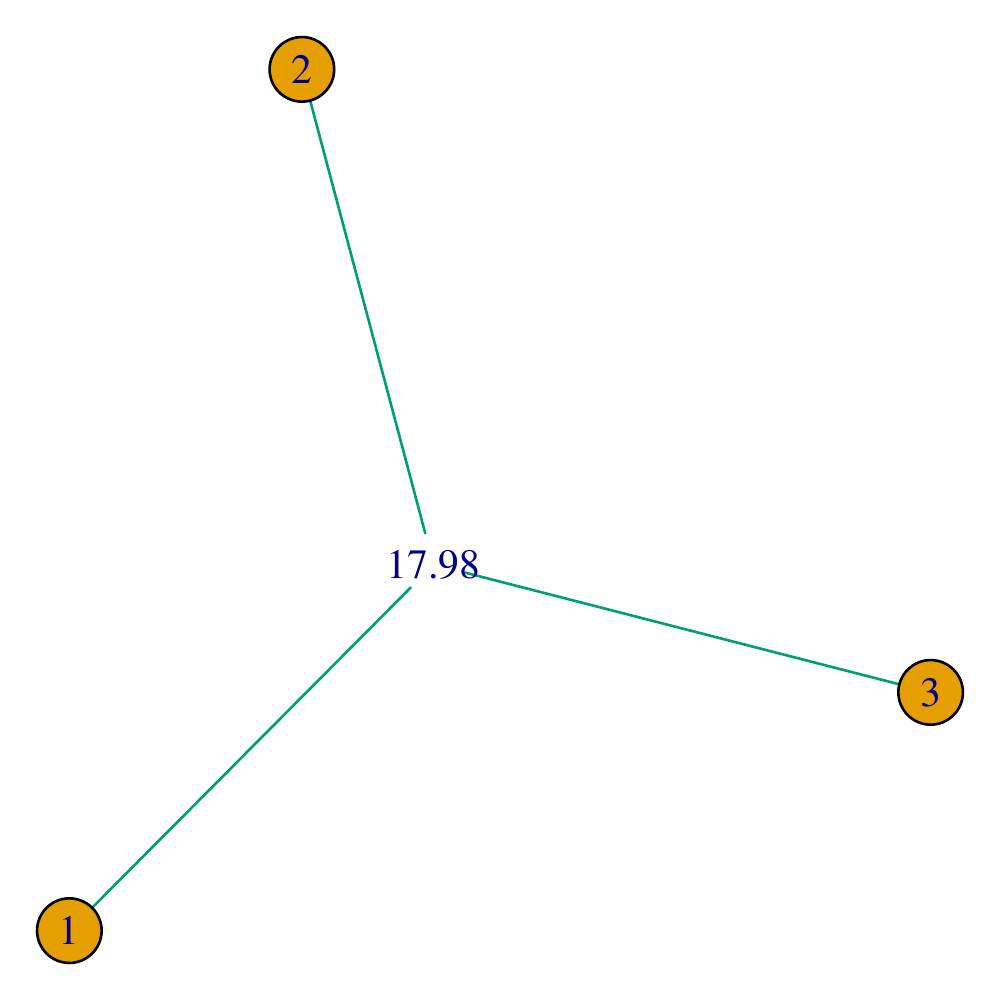}
\quad
\begin{tabular}[b]{r|cc|c|cc|c|c}
  & \multicolumn{3}{c|}{resampling}& \multicolumn{3}{c|}{distribution-free}& consistent\\
$N$ & $\hN \Mskript$ & $\hN \overline{\Mskript}$ & Comb & $\hN \Mskript$ & $\hN \overline{\Mskript}$ & detection & detection \\ 
  \hline
10 & 0.629 & 0.136 & 0.541 & 0.006 & 0.000 & 0.011 & 0.000 \\ 
  20 & 0.987 & 0.788 & 0.991 & 0.481 & 0.000 & 0.490 & 0.000 \\ 
  30 & 1.000 & 0.985 & 1.000 & 0.929 & 0.000 & 0.930 & 0.000 \\ 
  40 & 1.000 & 1.000 & 1.000 & 0.998 & 0.001 & 0.999 & 0.000 \\ 
  50 & 1.000 & 1.000 & 1.000 & 1.000 & 0.014 & 1.000 & 0.000 \\ 
  60 & 1.000 & 1.000 & 1.000 & 1.000 & 0.165 & 0.999 & 0.001 \\ 
  70 & 1.000 & 1.000 & 1.000 & 1.000 & 0.557 & 0.999 & 0.008 \\ 
  80 & 1.000 & 1.000 & 1.000 & 1.000 & 0.873 & 0.999 & 0.016 \\ 
  90 & 1.000 & 1.000 & 1.000 & 1.000 & 0.985 & 0.997 & 0.059 \\ 
  100 & 1.000 & 1.000 & 1.000 & 1.000 & 0.998 & 0.999 & 0.131 \\ 
   \hline
\end{tabular}

\medskip

\hspace{0.3\textwidth}\begin{tabular}[b]{r|cc|cc|cc|cc}
	& \multicolumn{8}{c}{resampling} \\
 & \multicolumn{2}{c|}{$r = 0.25$} &\multicolumn{2}{c|}{$r = 0.5$} &\multicolumn{2}{c|}{$r = 0.75$}&\multicolumn{2}{c}{$r = 1$} \\
	$N$ & $\hN \Mskript$ & $\hN \overline{\Mskript}$ & $\hN \Mskript$ & $\hN \overline{\Mskript}$ & $\hN \Mskript$ & $\hN \overline{\Mskript}$ & $\hN \Mskript$ & $\hN \overline{\Mskript}$ \\ 
	\hline
10 & 0.629 & 0.136 & 0.114 & 0.050 & 0.060 & 0.045 & 0.050 & 0.047 \\ 
  20 & 0.987 & 0.788 & 0.315 & 0.095 & 0.102 & 0.065 & 0.068 & 0.060 \\ 
  30 & 1.000 & 0.985 & 0.500 & 0.159 & 0.122 & 0.078 & 0.072 & 0.054 \\ 
  40 & 1.000 & 1.000 & 0.646 & 0.213 & 0.168 & 0.060 & 0.076 & 0.055 \\ 
  50 & 1.000 & 1.000 & 0.778 & 0.320 & 0.185 & 0.087 & 0.085 & 0.060 \\ 
  60 & 1.000 & 1.000 & 0.853 & 0.437 & 0.260 & 0.099 & 0.092 & 0.066 \\ 
  70 & 1.000 & 1.000 & 0.908 & 0.506 & 0.284 & 0.086 & 0.087 & 0.048 \\ 
  80 & 1.000 & 1.000 & 0.951 & 0.616 & 0.292 & 0.095 & 0.110 & 0.070 \\ 
  90 & 1.000 & 1.000 & 0.969 & 0.670 & 0.361 & 0.114 & 0.127 & 0.069 \\ 
  100 & 1.000 & 1.000 & 0.984 & 0.749 & 0.401 & 0.141 & 0.119 & 0.060 \\ 
   \hline
\end{tabular}

\caption{Normal perturbed events of 2 coins (Ex.\ \ref{ex:2coinsnorm}): dependence structure and empirical power.}		\label{fig:2coinsnorm}
\end{figure}
\end{example}

Now the above examples will be used as building blocks to illustrate the dependence structure detection algorithm. For the following examples the visualized dependence structure is (at least to us) much more comprehensible than the literal description.

\begin{example}[Several disjoint dependence clusters]\label{ex:several}
We look at samples of $(X_1,\ldots,$ $X_{26})$ where $(X_1,X_2,X_3)$ are as in Example \ref{ex:ncoins} with 2 coins, $ (X_7,\ldots,X_{11})$ are as in Example \ref{ex:ncoins} with 4 coins, $(X_4,X_5,X_6)$ and  $(X_{12},X_{13},X_{14})$ and  $(X_{15},X_{16},X_{17})$ are as in Example \ref{ex:tetra}, $(X_{18},\ldots,X_{21})$ and $(X_{22},\ldots,X_{25})$ are as in Example \ref{ex:ncoins} with 3 coins and $X_{26} \sim N(0,1)$. Furthermore, each of these tuples is independent of the others. Note that we added $X_{26}$ to make the detection much harder, since now the factorization for independent subsets \eqref{eq:factorization} implies $\Mskript(X_1,\ldots,X_{26}) = 0$. 

Figure \ref{fig:several} shows that the detection algorithm and the $3$-multivariance (with resampling) perform well, whereas the total multivariance suffers from averaging (see also Example \ref{ex:averaging}) and the distribution-free dependence tests are too conservative.

\begin{figure}[H]
\includegraphics[width = 0.3\textwidth]{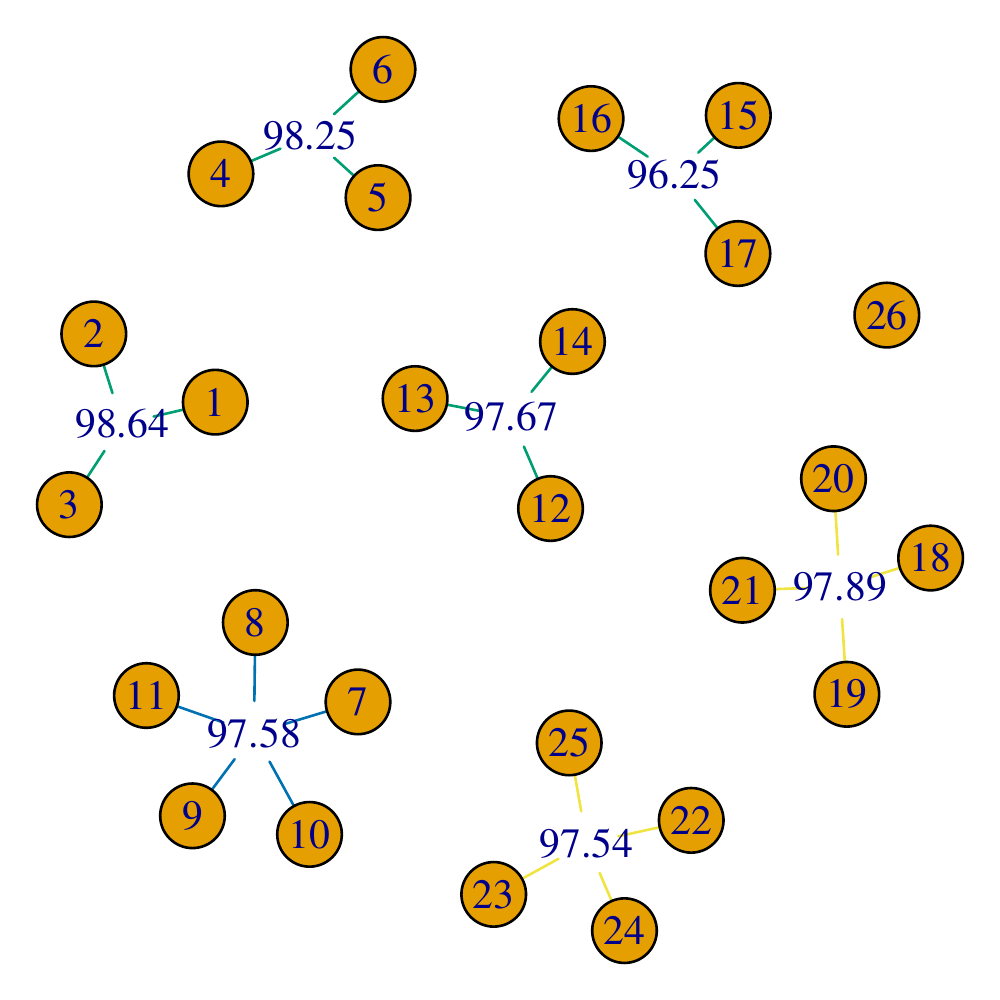}
\quad
\begin{tabular}[b]{r|cc|c|cc|c|c}
  & \multicolumn{3}{c|}{resampling}& \multicolumn{3}{c|}{distribution-free}& consistent\\
$N$ & $\hN \overline{\Mskript}$ & $\hN \Mskript_3$ & Comb & $\hN \overline{\Mskript}$ & $\hN \Mskript_3$ & detection & detection \\ 
  \hline
10 & 0.044 & 0.073 & 0.049 & 0.015 & 0.000 & 0.000 & 0.000 \\ 
  20 & 0.047 & 0.214 & 0.116 & 0.000 & 0.000 & 0.000 & 0.007 \\ 
  30 & 0.041 & 0.424 & 0.242 & 0.000 & 0.000 & 0.908 & 0.177 \\ 
  40 & 0.045 & 0.654 & 0.465 & 0.000 & 0.000 & 0.945 & 0.508 \\ 
  50 & 0.039 & 0.831 & 0.667 & 0.000 & 0.000 & 0.950 & 0.716 \\ 
  60 & 0.053 & 0.947 & 0.855 & 0.000 & 0.000 & 0.920 & 0.835 \\ 
  70 & 0.053 & 0.986 & 0.955 & 0.000 & 0.000 & 0.920 & 0.911 \\ 
  80 & 0.047 & 0.998 & 0.989 & 0.000 & 0.000 & 0.916 & 0.966 \\ 
  90 & 0.034 & 1.000 & 0.999 & 0.000 & 0.000 & 0.915 & 0.973 \\ 
  100 & 0.051 & 1.000 & 1.000 & 0.000 & 0.000 & 0.923 & 0.981 \\ 
   \hline
\end{tabular}

\caption{The dependence structure with several clusters (Ex.\ \ref{ex:several}).}	\label{fig:several}
\end{figure}
\end{example}

\begin{example}[Star dependence structure]\label{ex:star}
Consider samples of $(X_1,X_2,X_3,\,X_1,$ $X_2,X_3,\,X_1,X_2,X_3)$ where $X_1,X_2,X_3$ are as in Example \ref{ex:ncoins} with 2 coins. Then the structure in Figure \ref{fig:star} is detected. Here the graph was slightly cleaned up: vertices representing only pairwise multivariance were reduced to edges with labels.

The variables are Bernoulli distributed and thus (as e.g.\ in Example \ref{ex:ncoins}) the detection rate of $95\%$ reflects the $5\%$ falsely detected pairwise dependencies.
\begin{figure}[H]

\includegraphics[width = 0.3\textwidth]{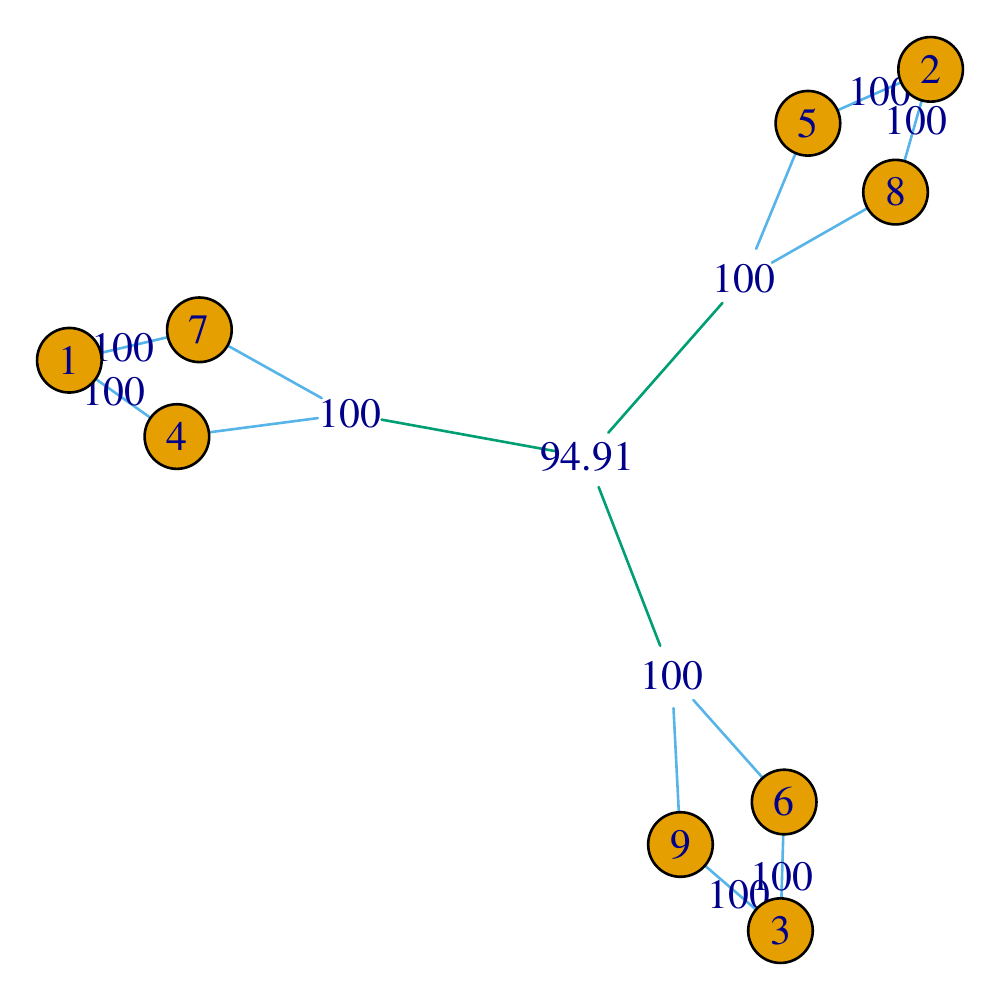}
\quad\scriptsize\centering
\begin{tabular}[b]{r|ccc|c|ccc|c|c}
  & \multicolumn{4}{c|}{resampling}& \multicolumn{4}{c|}{distribution-free}& consistent\\
$N$ & $\hN \overline{\Mskript}$ & $\hN \Mskript_2$ & $\hN \Mskript_3$ & Comb & $\hN \overline{\Mskript}$ & $\hN \Mskript_2$ & $\hN \Mskript_3$ & detection & detection \\ 
  \hline
10 & 1.000 & 1.000 & 0.999 & 1.000 & 0.269 & 0.237 & 0.049 & 0.000 & 0.679 \\ 
  20 & 1.000 & 1.000 & 1.000 & 1.000 & 1.000 & 1.000 & 1.000 & 0.973 & 0.988 \\ 
  30 & 1.000 & 1.000 & 1.000 & 1.000 & 1.000 & 1.000 & 1.000 & 0.950 & 0.998 \\ 
  40 & 1.000 & 1.000 & 1.000 & 1.000 & 1.000 & 1.000 & 1.000 & 0.959 & 0.999 \\ 
  50 & 1.000 & 1.000 & 1.000 & 1.000 & 1.000 & 1.000 & 1.000 & 0.960 & 1.000 \\ 
  60 & 1.000 & 1.000 & 1.000 & 1.000 & 1.000 & 1.000 & 1.000 & 0.963 & 1.000 \\ 
  70 & 1.000 & 1.000 & 1.000 & 1.000 & 1.000 & 1.000 & 1.000 & 0.954 & 1.000 \\ 
  80 & 1.000 & 1.000 & 1.000 & 1.000 & 1.000 & 1.000 & 1.000 & 0.955 & 1.000 \\ 
  90 & 1.000 & 1.000 & 1.000 & 1.000 & 1.000 & 1.000 & 1.000 & 0.946 & 1.000 \\ 
  100 & 1.000 & 1.000 & 1.000 & 1.000 & 1.000 & 1.000 & 1.000 & 0.958 & 1.000 \\ 
   \hline
\end{tabular}

\caption{The star dependence structure of Ex.\ \ref{ex:star}.}	\label{fig:star}
\end{figure}

\end{example}

\begin{example}[Iterated dependence structure]\label{ex:iterated}
Consider samples of random variables $(X_1,\ldots,X_{13})$ where $X_1,\ldots,X_{10}$ are independent but $X_1,X_2,X_{11}$ are dependent (but all subtuples are independent), the same holds for $X_1,\ldots,X_5,X_{12}$ and $X_1,\ldots,X_9,X_{13}$. Such examples can be constructed by letting $X_{11} = f(X_1,X_2)$ for some (special) $f$, and analogously for the others. If such a structure is detected the graph looks like Figure \ref{fig:iterated}.

For the dependence we used $f(x_1,\ldots,x_k) = \sum_{i=1}^k x_i \mod 2$, and $X_i,i=1,\ldots,10$ were i.i.d.\ Bernoulli random variables. The dependence structure is reasonably detected given 100 samples by the test based algorithm, the consistent estimator requires a much large sample size. Both total multivariance and 3-multivariance also detect the dependence, see Figure \ref{fig:iterated}.

\begin{figure}[H]
	\includegraphics[width = 0.3\textwidth]{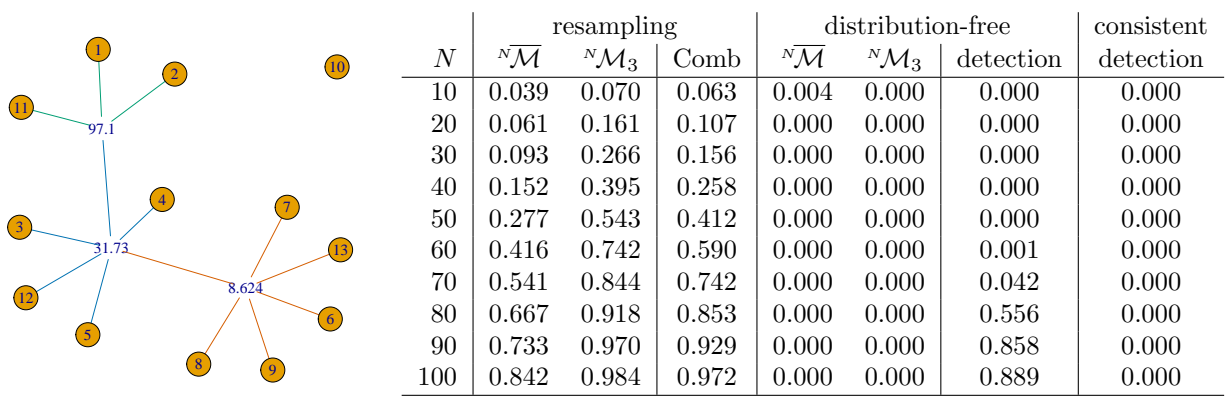}
	\quad 
\begin{tabular}[b]{r|cc|c|cc|c|c}
  & \multicolumn{3}{c|}{resampling}& \multicolumn{3}{c|}{distribution-free}& consistent\\
$N$ & $\hN \overline{\Mskript}$ & $\hN \Mskript_3$ & Comb & $\hN \overline{\Mskript}$ & $\hN \Mskript_3$ & detection & detection \\ 
  \hline
10 & 0.039 & 0.070 & 0.063 & 0.004 & 0.000 & 0.000 & 0.000 \\ 
  20 & 0.061 & 0.161 & 0.107 & 0.000 & 0.000 & 0.000 & 0.000 \\ 
  30 & 0.093 & 0.266 & 0.156 & 0.000 & 0.000 & 0.000 & 0.000 \\ 
  40 & 0.152 & 0.395 & 0.258 & 0.000 & 0.000 & 0.000 & 0.000 \\ 
  50 & 0.277 & 0.543 & 0.412 & 0.000 & 0.000 & 0.000 & 0.000 \\ 
  60 & 0.416 & 0.742 & 0.590 & 0.000 & 0.000 & 0.001 & 0.000 \\ 
  70 & 0.541 & 0.844 & 0.742 & 0.000 & 0.000 & 0.042 & 0.000 \\ 
  80 & 0.667 & 0.918 & 0.853 & 0.000 & 0.000 & 0.556 & 0.000 \\ 
  90 & 0.733 & 0.970 & 0.929 & 0.000 & 0.000 & 0.858 & 0.000 \\ 
  100 & 0.842 & 0.984 & 0.972 & 0.000 & 0.000 & 0.889 & 0.000 \\ 
   \hline
\end{tabular}

\caption{The iterated dependence structure of Ex.\ \ref{ex:iterated}.}	\label{fig:iterated}
\end{figure}
\end{example}

\begin{example}[Ring dependence structure]\label{ex:ring}
The random variables $(X_1,\ldots,X_{15})$ are defined as follows. $X_i$ are i.i.d.\  Bernoulli random variables for $i \in \{1,2,3,5,6,8,$ $9,11,12,14\}$, $X_k:=(\sum_{i=k-3}^{k-1} X_i) \mod 2$ for $k \in \{4,7,10,13\}$ and $X_{15}:=(X_{13}+X_{14}+X_1) \mod 2$.

Since here only quadruple dependence is present, only total multivariance is used to detect it. The dependence structure detection works surprisingly well, also with small sample sizes, see Figure \ref{fig:ring}. 
\begin{figure}[H]
\includegraphics[width = 0.3\textwidth]{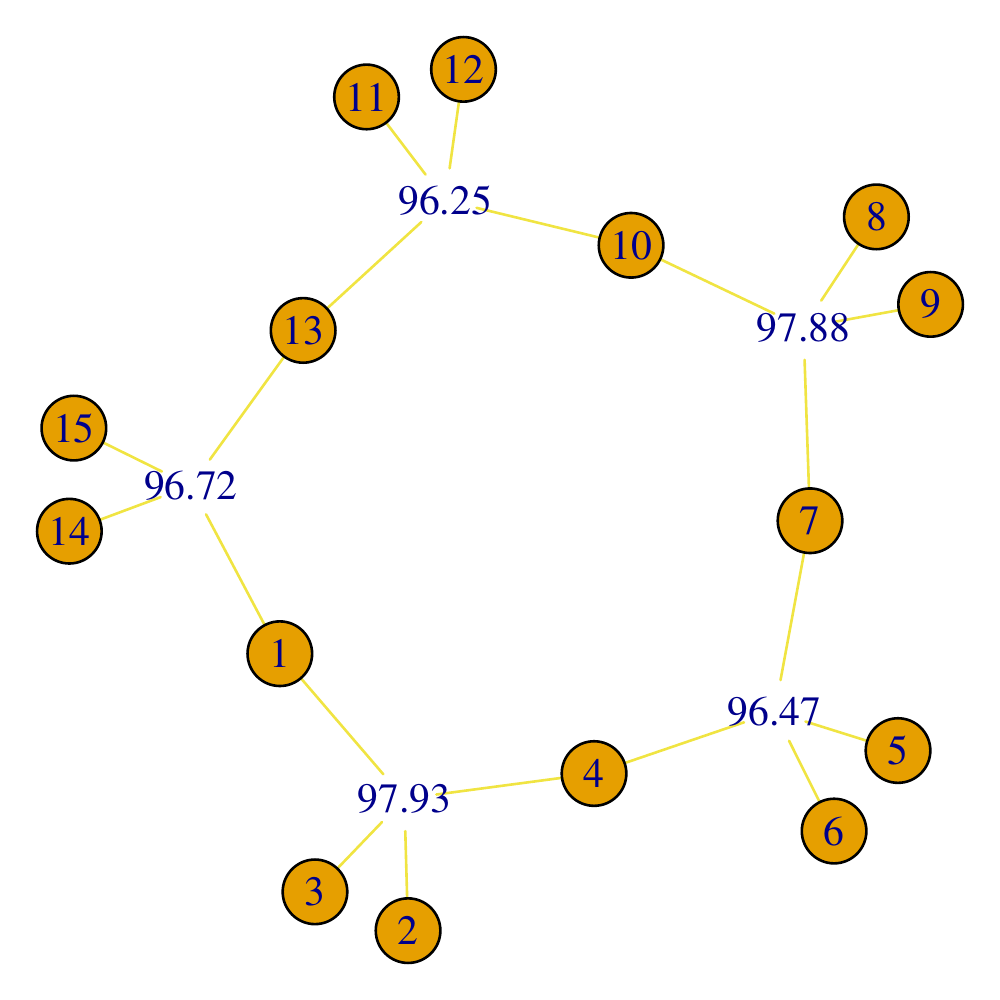}
\quad\centering
\begin{tabular}[b]{r|c|c|c|c|c}
  & \multicolumn{2}{c|}{resampling}& \multicolumn{2}{c|}{distribution-free}& consistent\\
$N$ & $\hN \overline{\Mskript}$ & Comb & $\hN \overline{\Mskript}$ & detection & detection \\ 
  \hline
10 & 0.040 & 0.063 & 0.002 & 0.000 & 0.010 \\ 
  20 & 0.050 & 0.082 & 0.000 & 0.054 & 0.404 \\ 
  30 & 0.089 & 0.120 & 0.000 & 0.939 & 0.750 \\ 
  40 & 0.180 & 0.157 & 0.000 & 0.938 & 0.903 \\ 
  50 & 0.305 & 0.256 & 0.000 & 0.948 & 0.961 \\ 
  60 & 0.481 & 0.383 & 0.000 & 0.928 & 0.971 \\ 
  70 & 0.620 & 0.510 & 0.000 & 0.925 & 0.983 \\ 
  80 & 0.760 & 0.672 & 0.000 & 0.933 & 0.994 \\ 
  90 & 0.851 & 0.769 & 0.000 & 0.931 & 0.998 \\ 
  100 & 0.924 & 0.887 & 0.000 & 0.915 & 0.997 \\ 
   \hline
\end{tabular}

\caption{The ring dependence structure of Ex.\ \ref{ex:ring}.}	\label{fig:ring}
\end{figure}
\end{example}

\begin{example}[The full dependence structures]\label{ex:full}
For the Examples \ref{ex:tetra} to \ref{ex:several} the clustered dependence structure and the full dependence structure coincide. For Examples \ref{ex:star}, \ref{ex:iterated} and \ref{ex:ring} the full dependence structures are given in Figure \ref{fig:full-star}, \ref{fig:full-iterated} and \ref{fig:full-ring}, respectively. The full graph is not as easy to comprehend as the clustered graphs, to improve it we used the order of dependence as labels of the dependency nodes. Moreover, besides the full graph also individual graphs depicting only the dependence of a certain order (for tuples for which no lower order dependence was detected) are presented.  

\begin{figure}[H]\centering
\includegraphics[width = 0.32\textwidth]{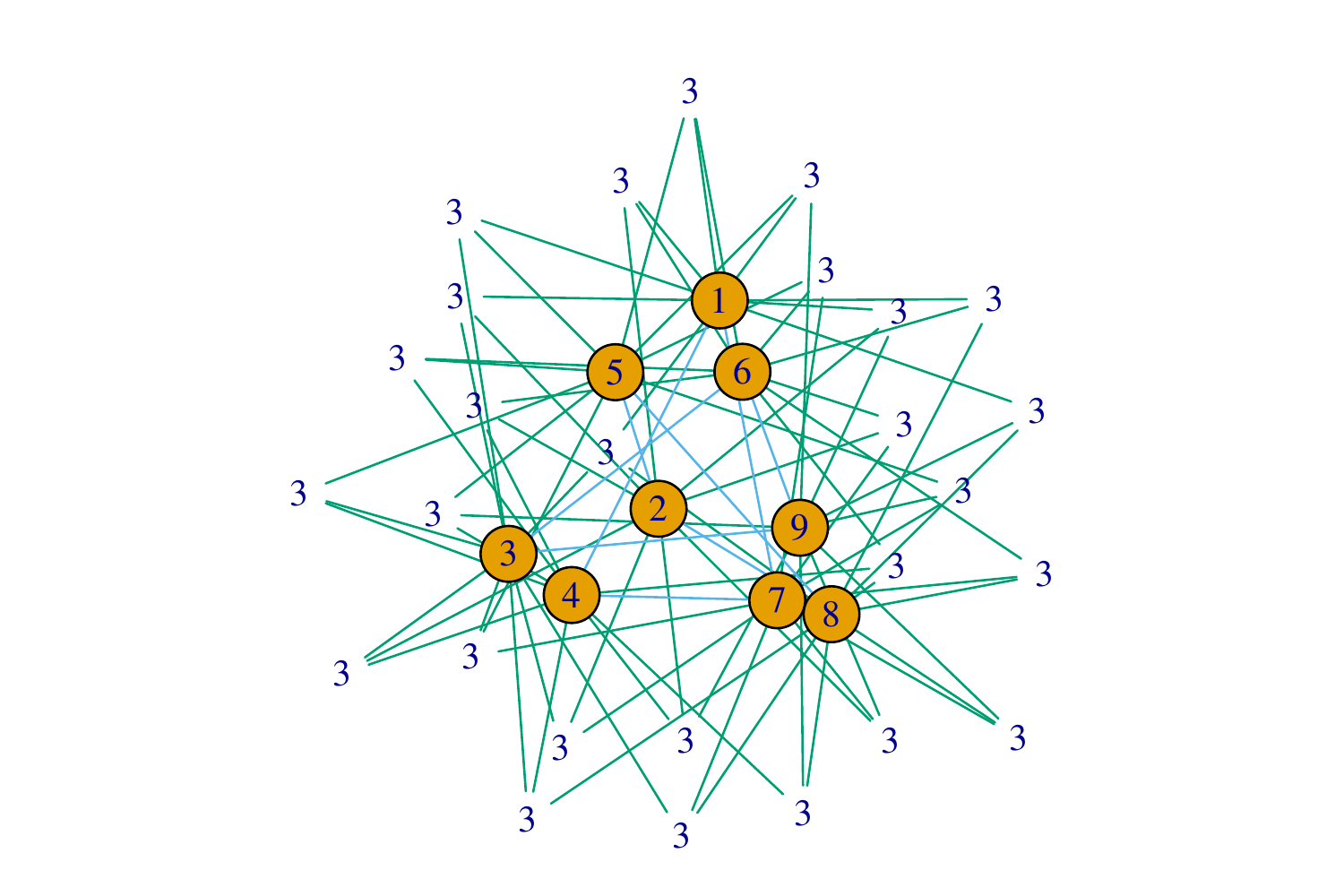}
\includegraphics[width = 0.32\textwidth]{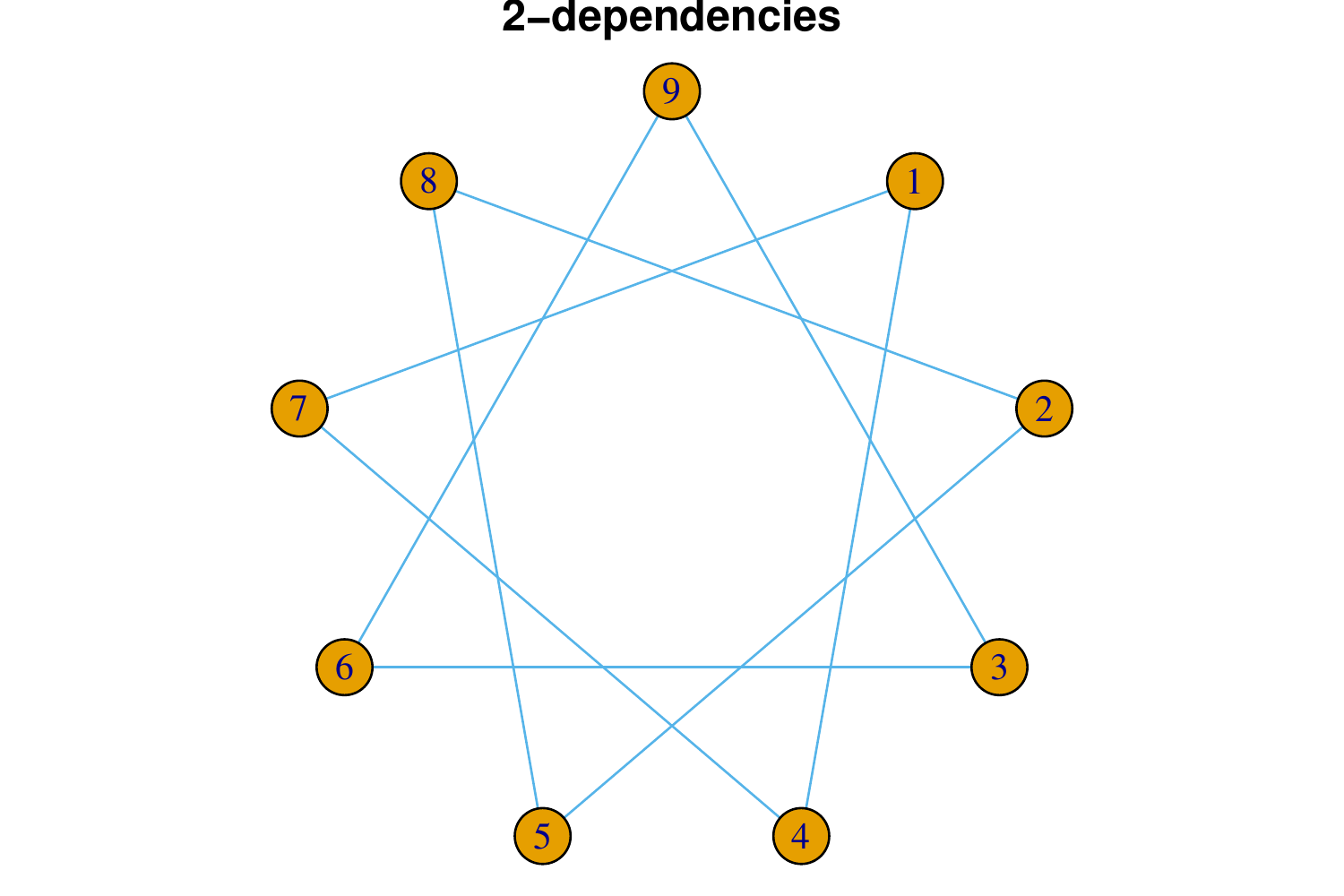}
\includegraphics[width = 0.32\textwidth]{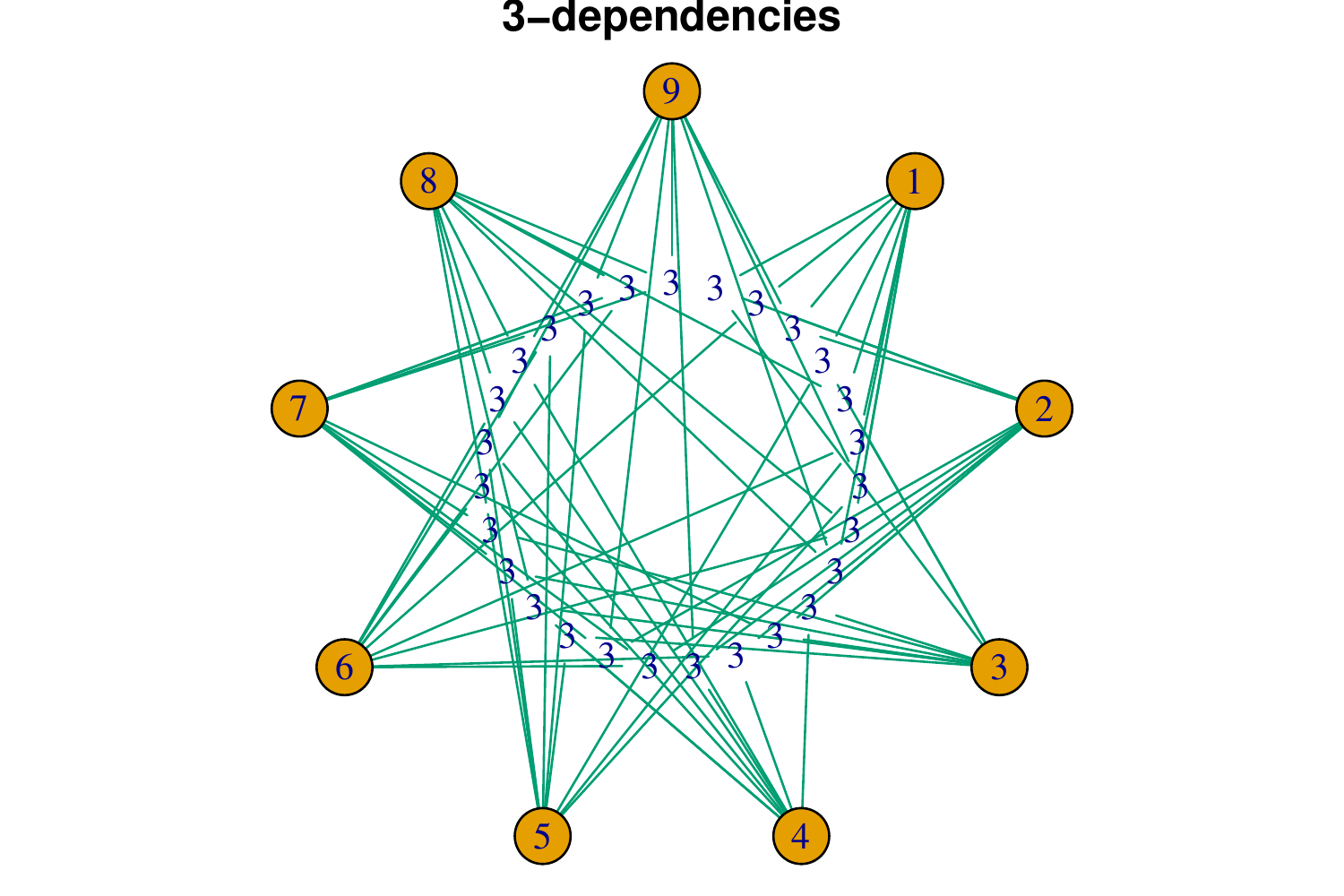}\\
\caption{The full dependence structure of Ex.\ \ref{ex:star} (star dependence structure).}	\label{fig:full-star}
\end{figure}
\begin{figure}[H]\centering
\includegraphics[width = 0.32\textwidth]{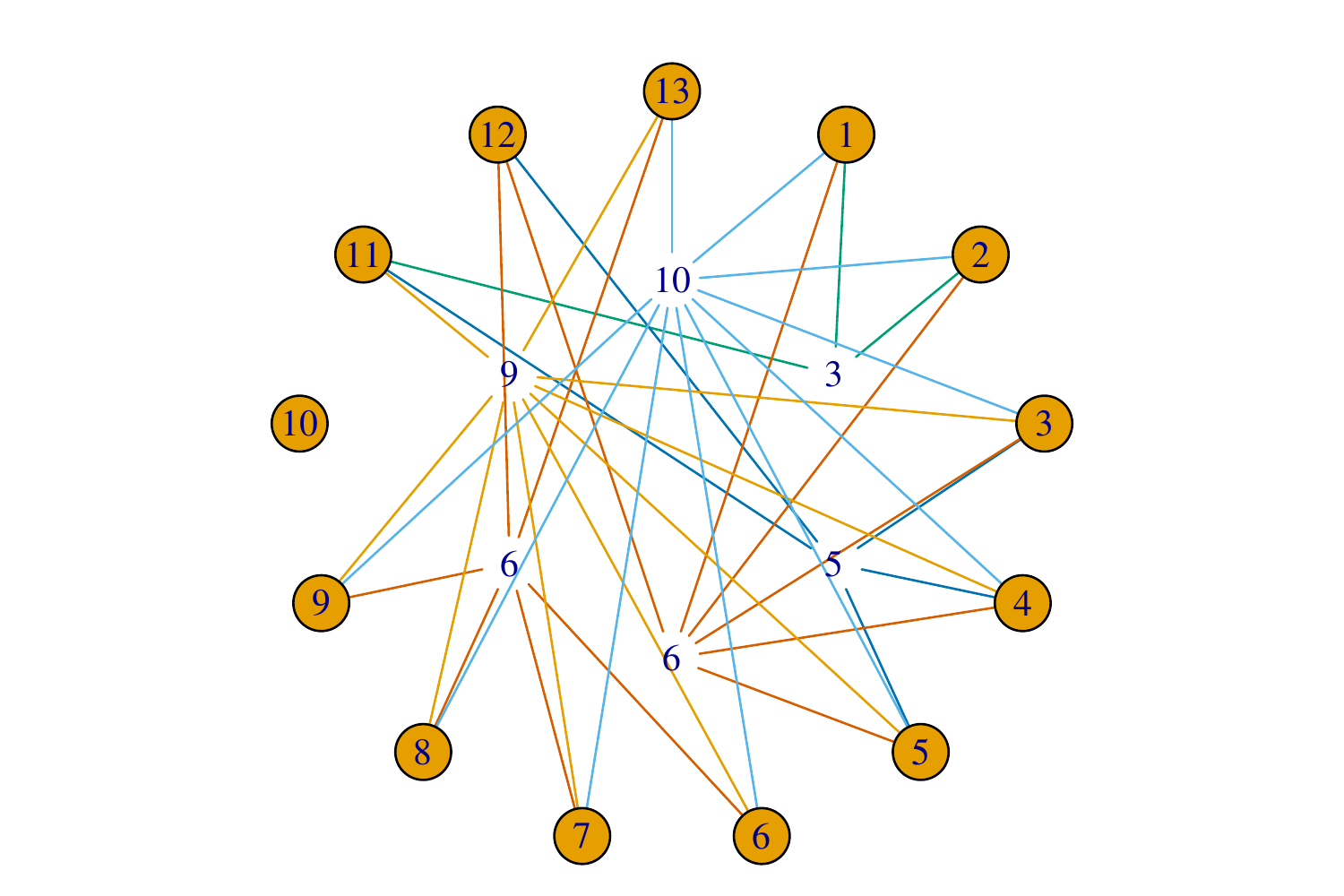}
\includegraphics[width = 0.32\textwidth]{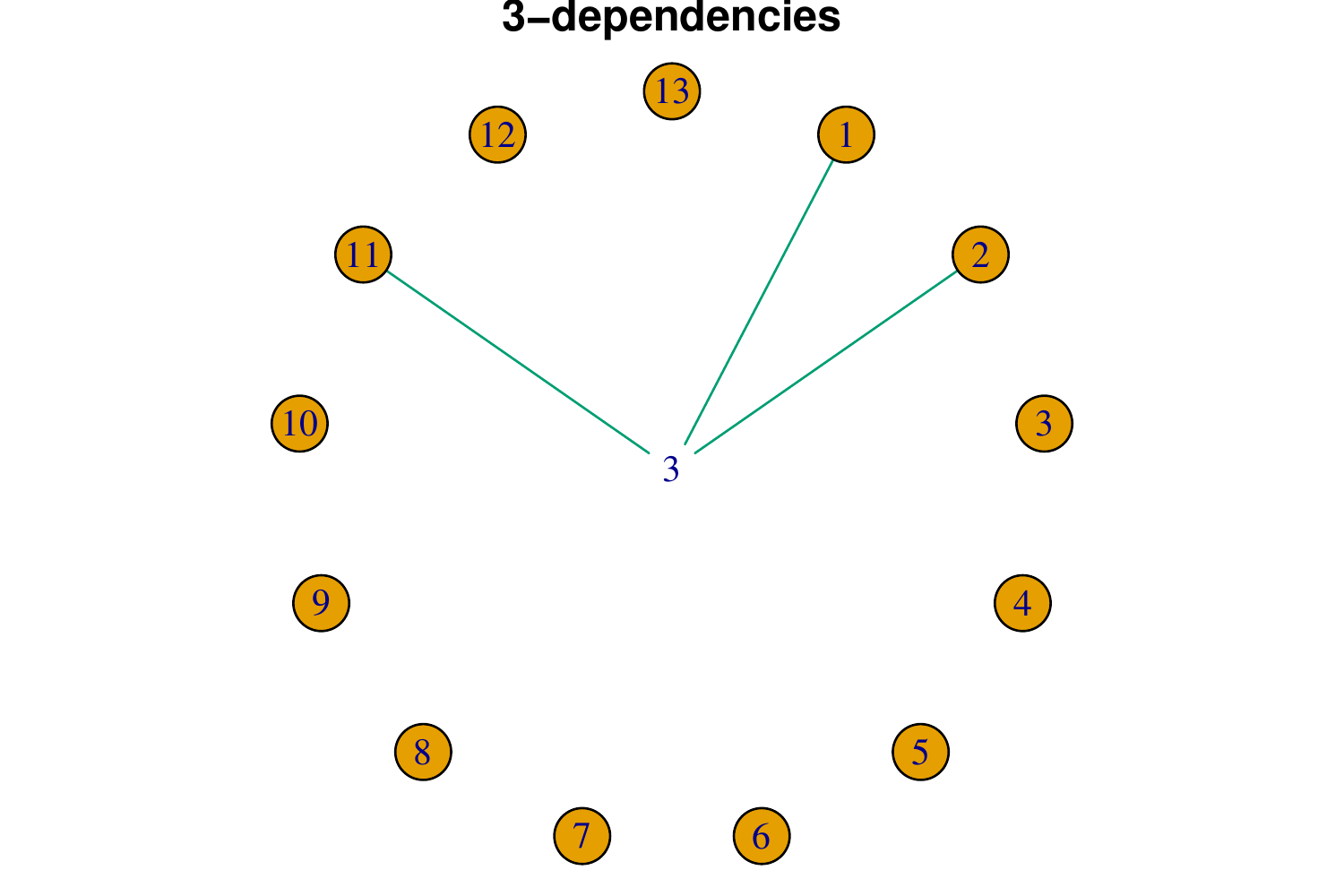}
\includegraphics[width = 0.32\textwidth]{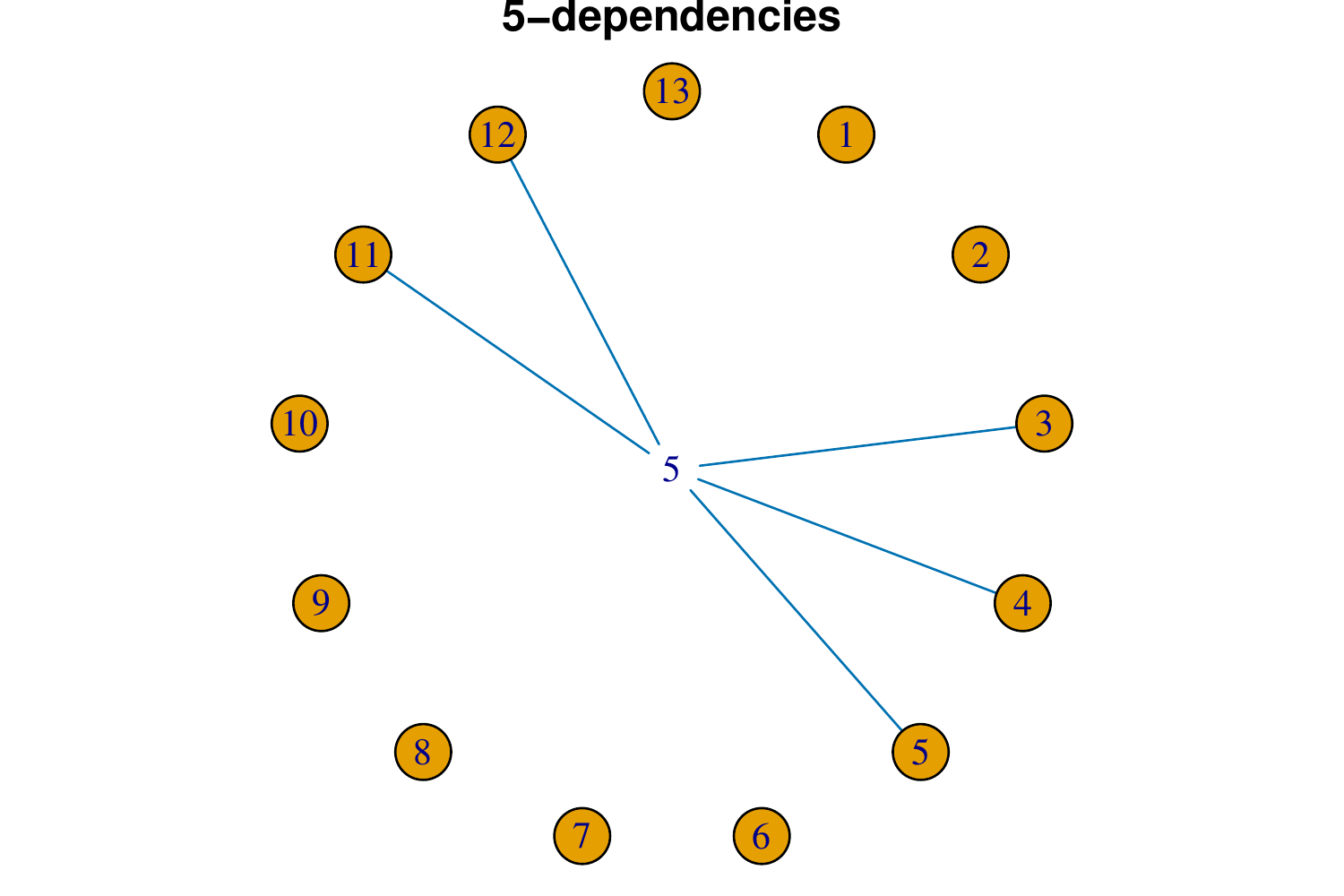}
\includegraphics[width = 0.32\textwidth]{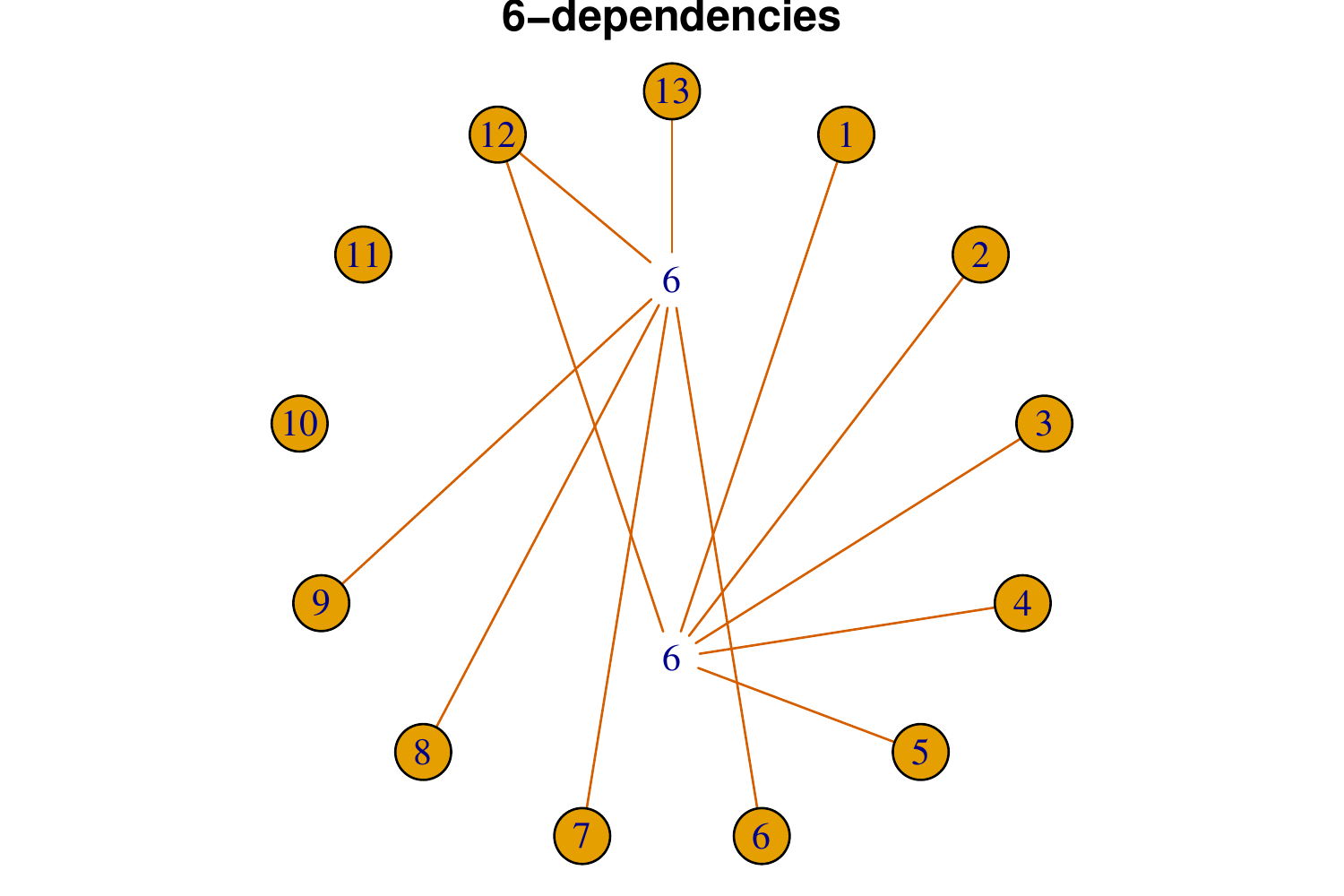}
\includegraphics[width = 0.32\textwidth]{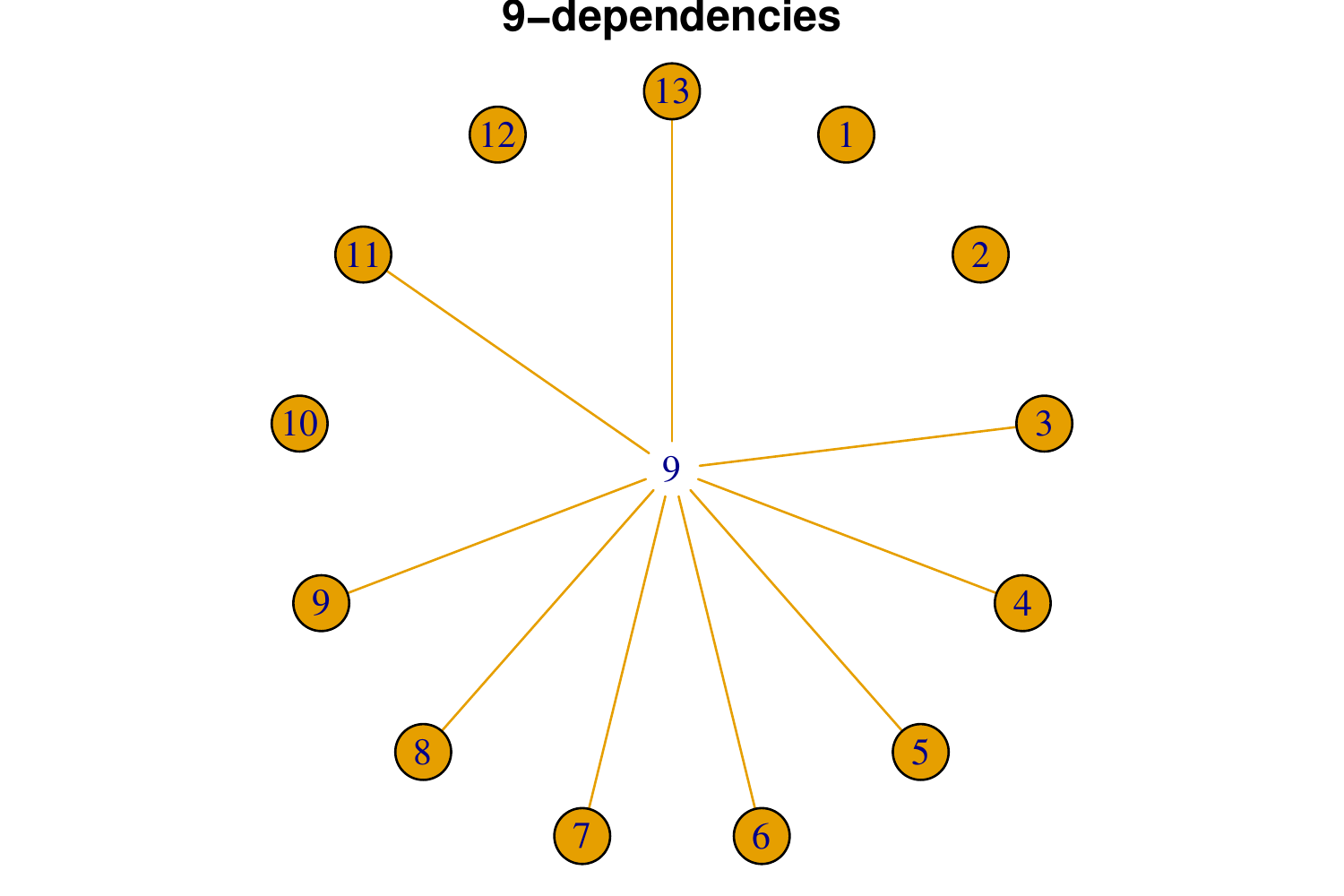}
\includegraphics[width = 0.32\textwidth]{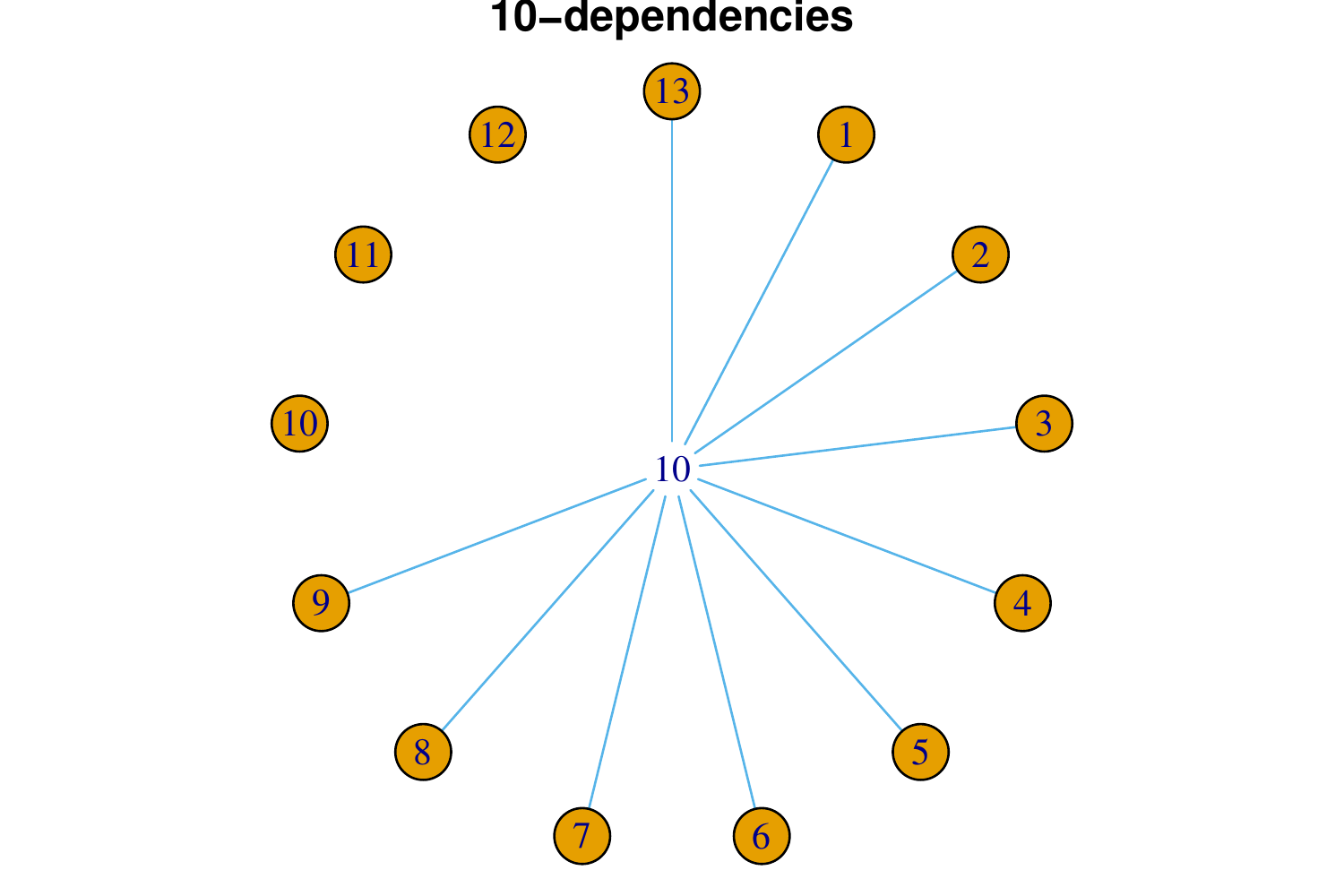}
\caption{The full dependence structure of Ex.\ \ref{ex:iterated} (iterated dependence structure).}	\label{fig:full-iterated}
\end{figure}
\begin{figure}[H]\centering
\includegraphics[width = 0.32\textwidth]{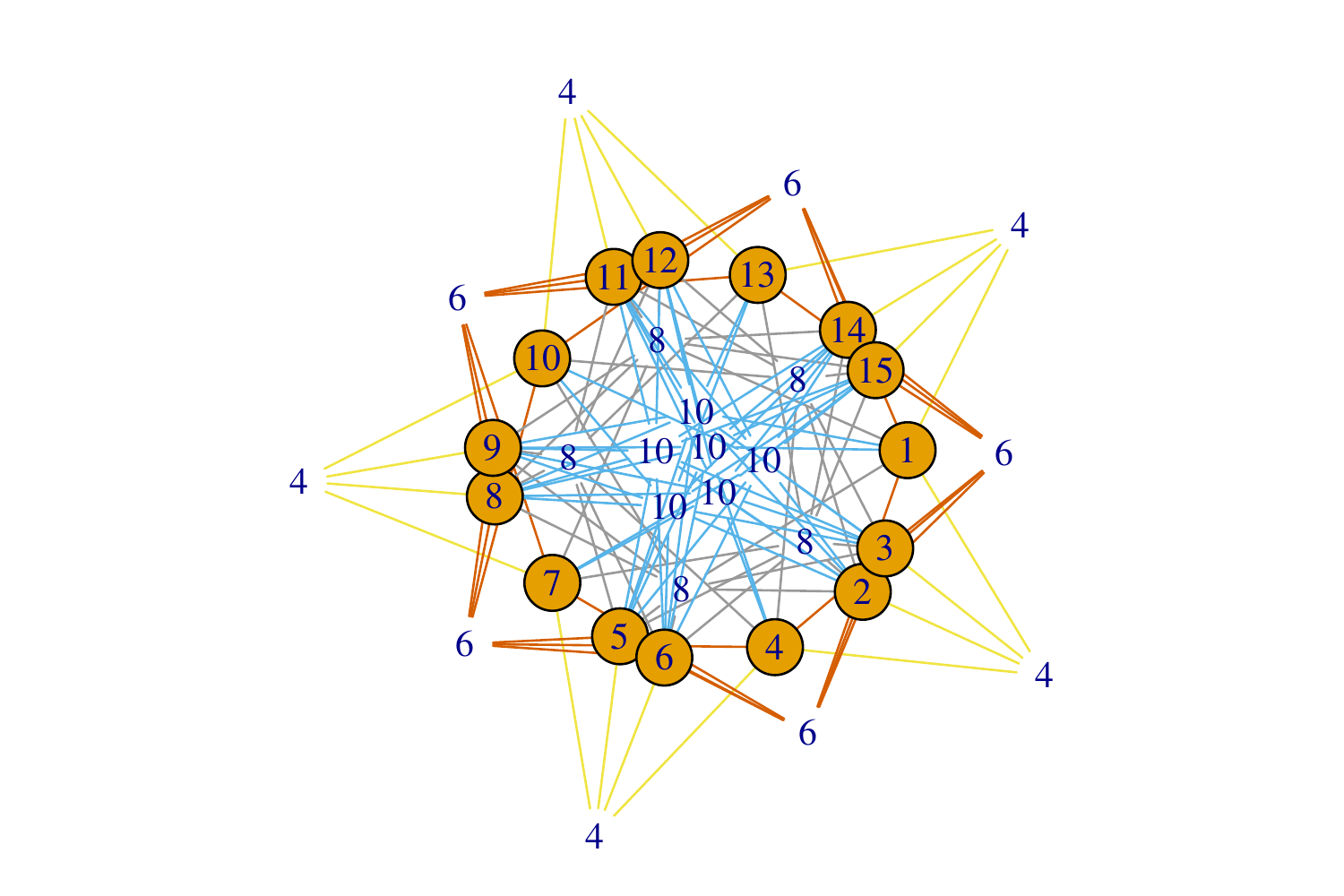}
\includegraphics[width = 0.32\textwidth]{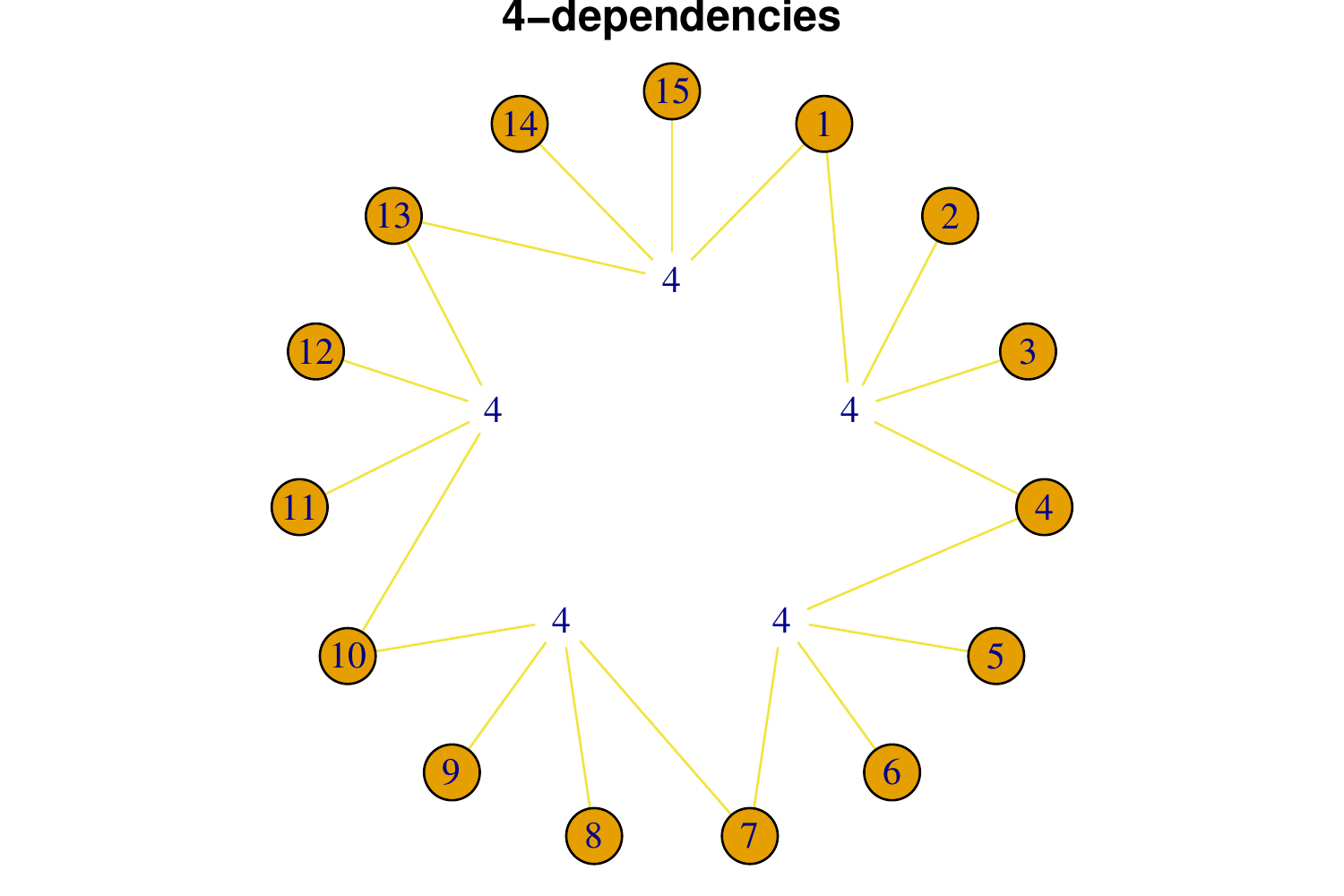}
\includegraphics[width = 0.32\textwidth]{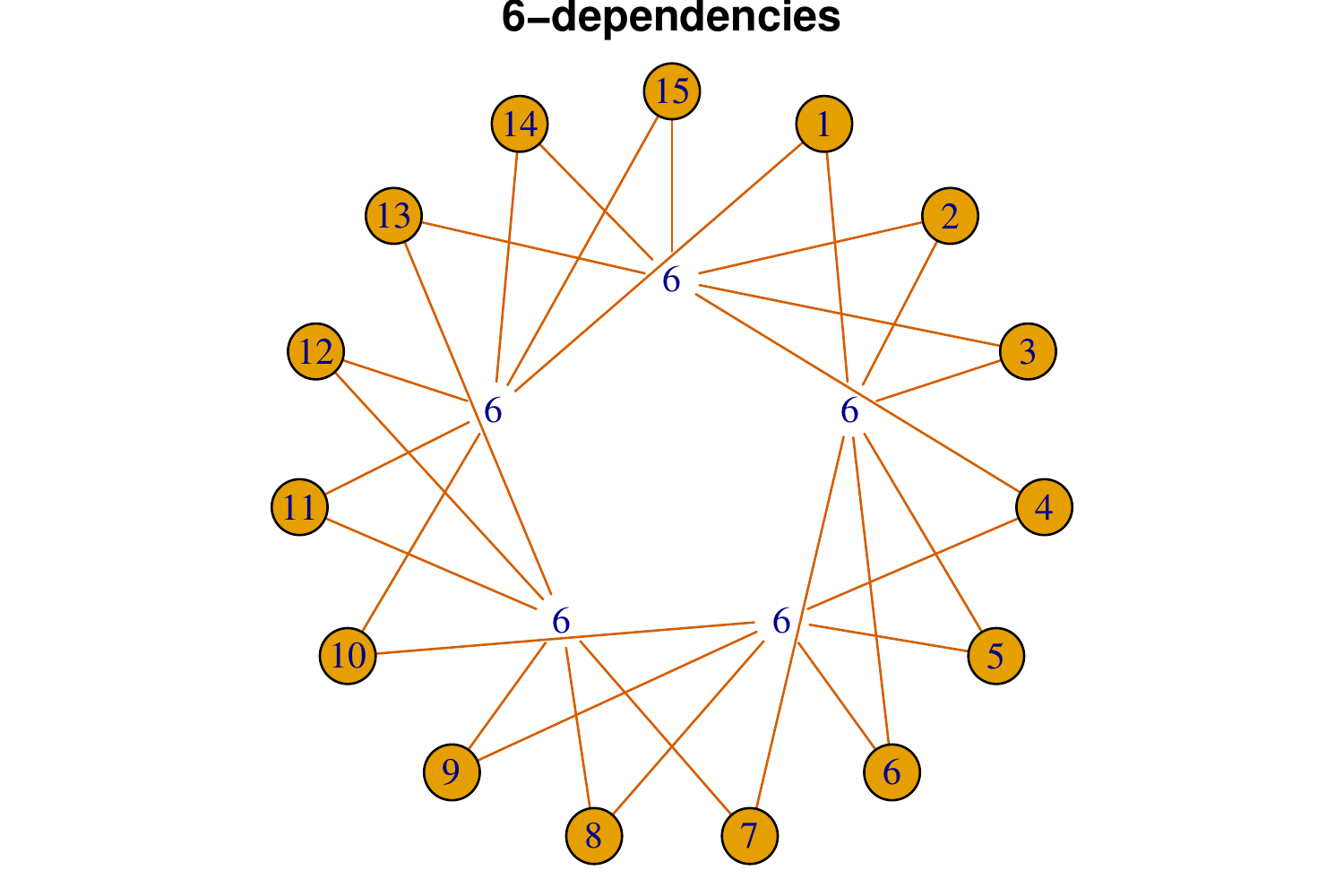}
\includegraphics[width = 0.32\textwidth]{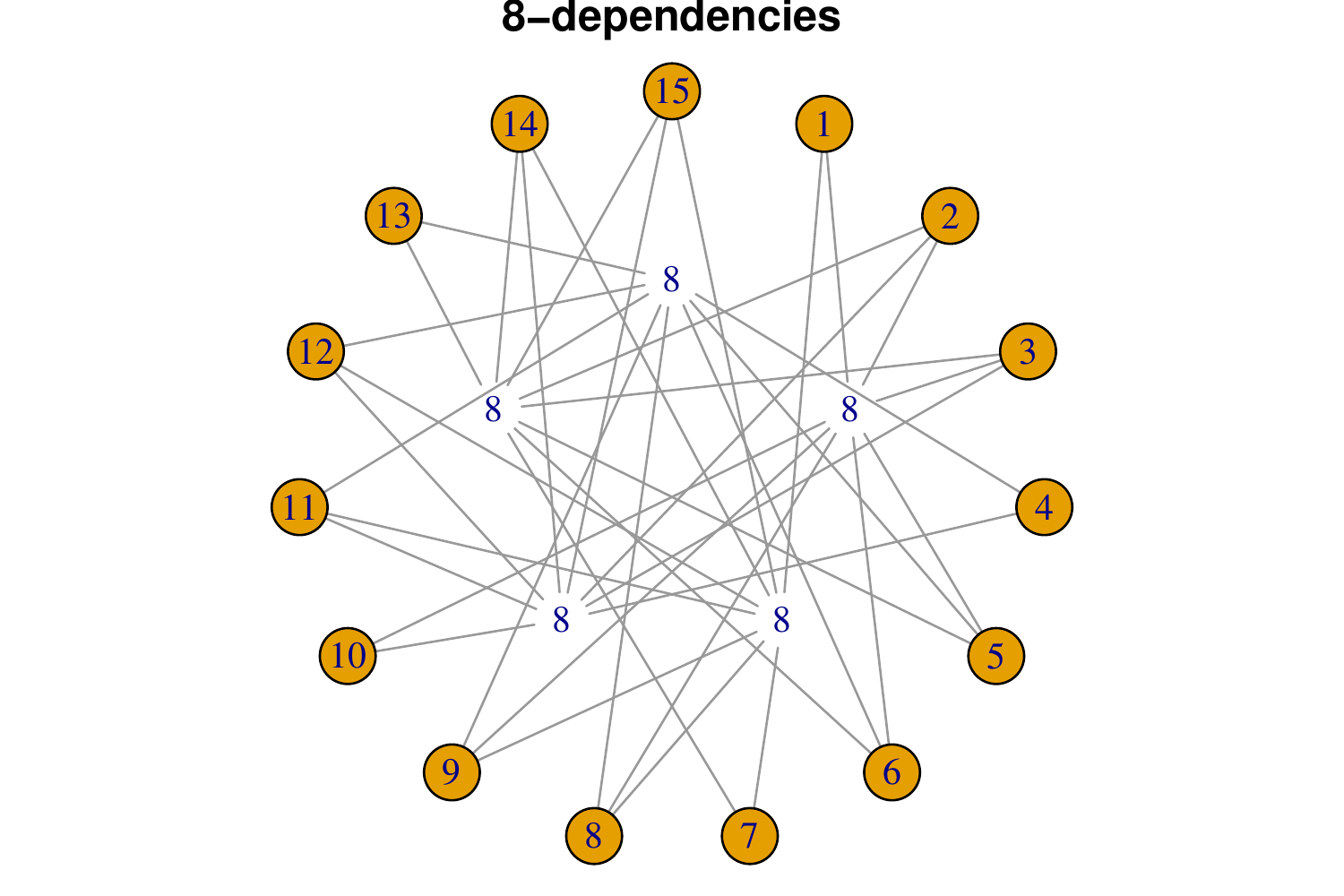}
\includegraphics[width = 0.32\textwidth]{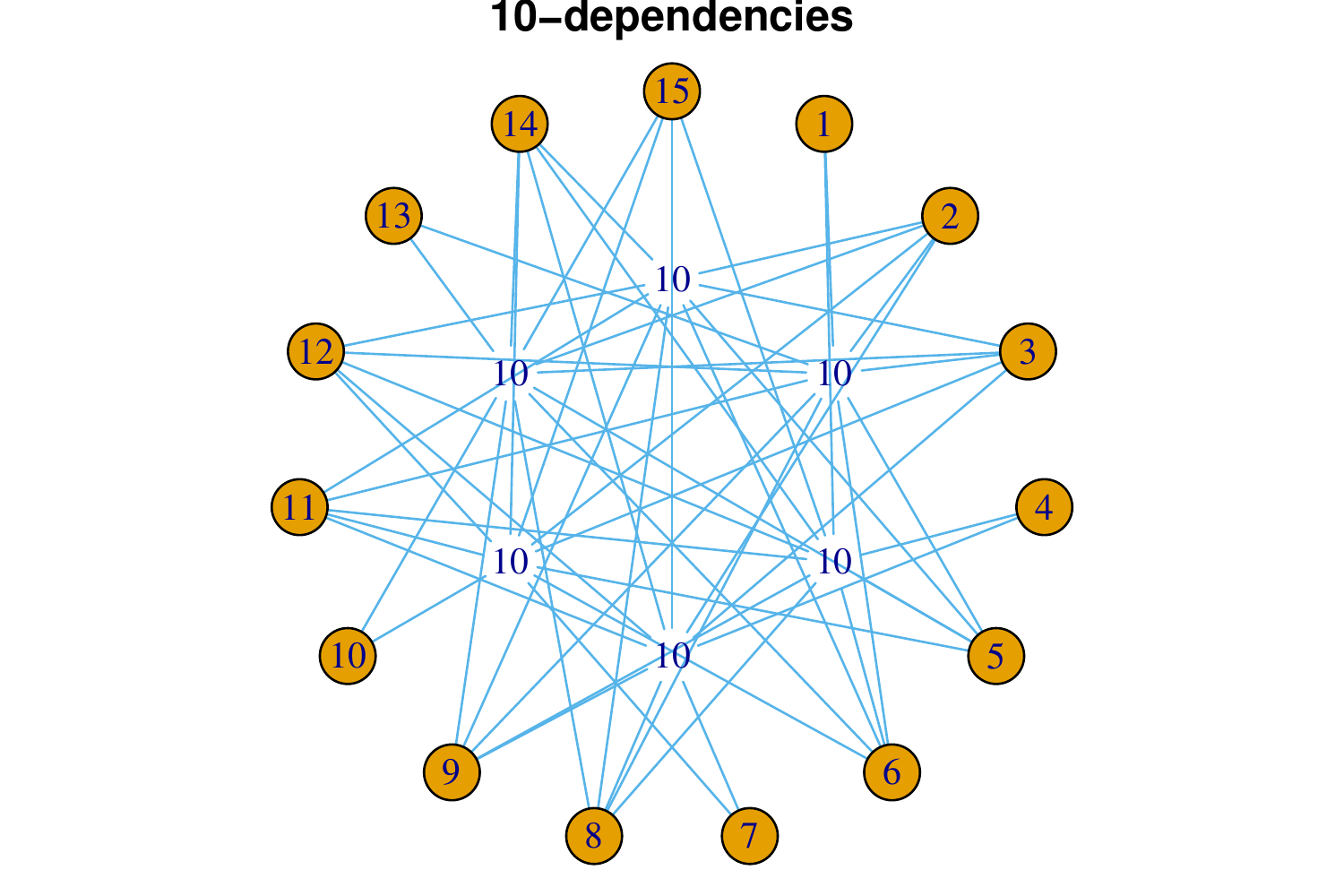}
\quad
\caption{The full dependence structure of Ex.\ \ref{ex:ring} (ring dependence structure).}	\label{fig:full-ring}
\end{figure}

\end{example}

\subsection{Empirical studies of properties of distance multivariance} \label{sec:ex-properties}

Note that in the papers introducing distance multivariance \cite{BoetKellSchi2018,BoetKellSchi2019} only two very elementary examples are contained. Thus simultaneously to illustrating the new measures and methods  provided in the current paper we also provide the first detailed empirical study of distance multivariance. The following aspects of multivariance, total multivariance and $m$-multivariance are discussed: the empirical size of the tests (Example \ref{ex:emp-size}), the dependence of the distribution of the test statistic on marginal distributions, sample size, dimension and the choice of $\psi$ (Example \ref{ex:variousm}), the computational complexity (Example \ref{ex:comp-time}), the moment conditions (Example \ref{ex:mom}) and the statistical curse of dimensions (Example \ref{ex:averaging}). The section closes with a generalization of total multivariance (Example \ref{ex:lambda}).

\begin{example}[Empirical size]\label{ex:emp-size} \begin{figure}[H] \centering
\includegraphics[width = 0.3\textwidth]{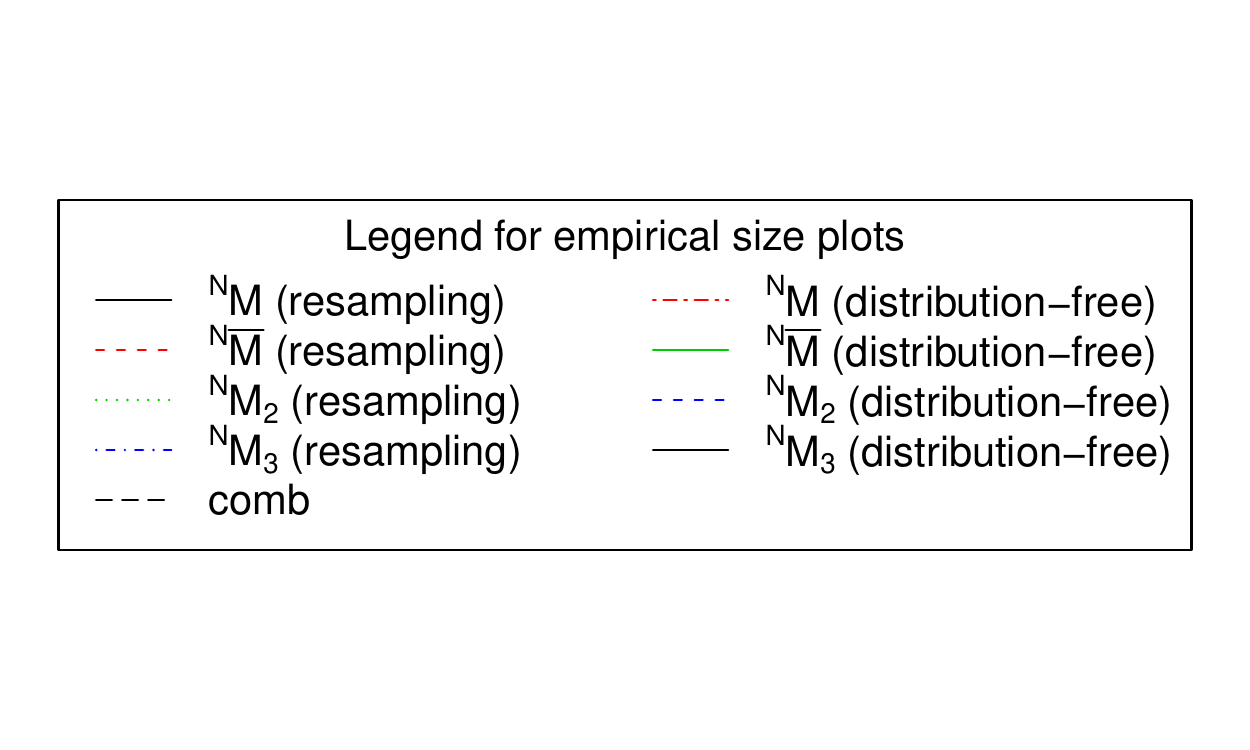}
\includegraphics[width = 0.3\textwidth]{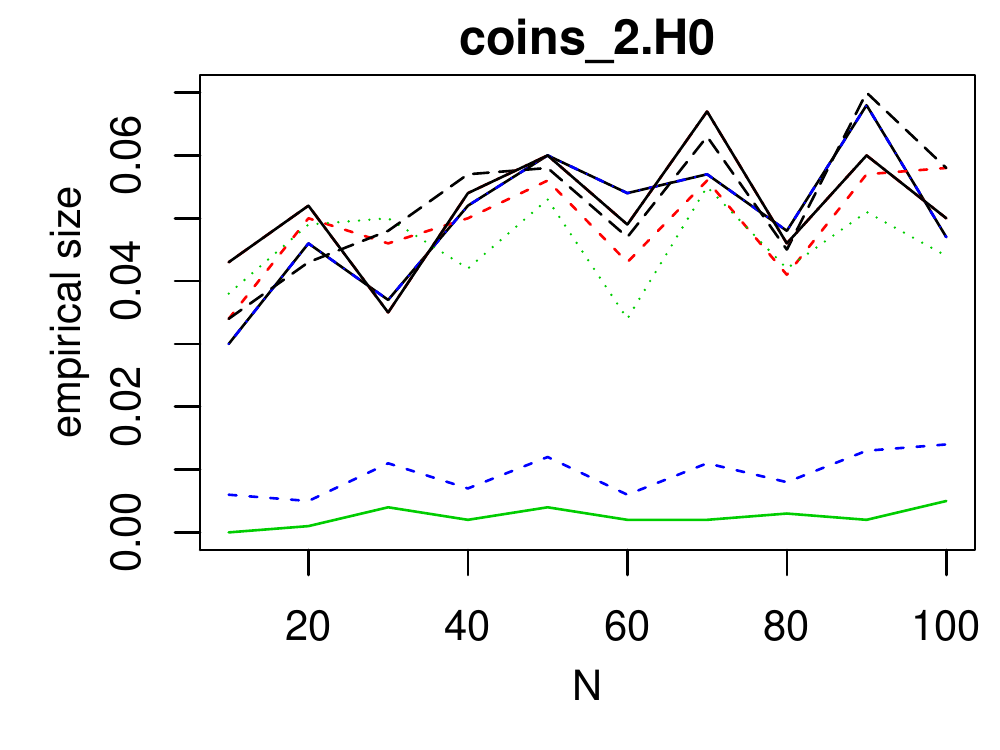}
\includegraphics[width = 0.3\textwidth]{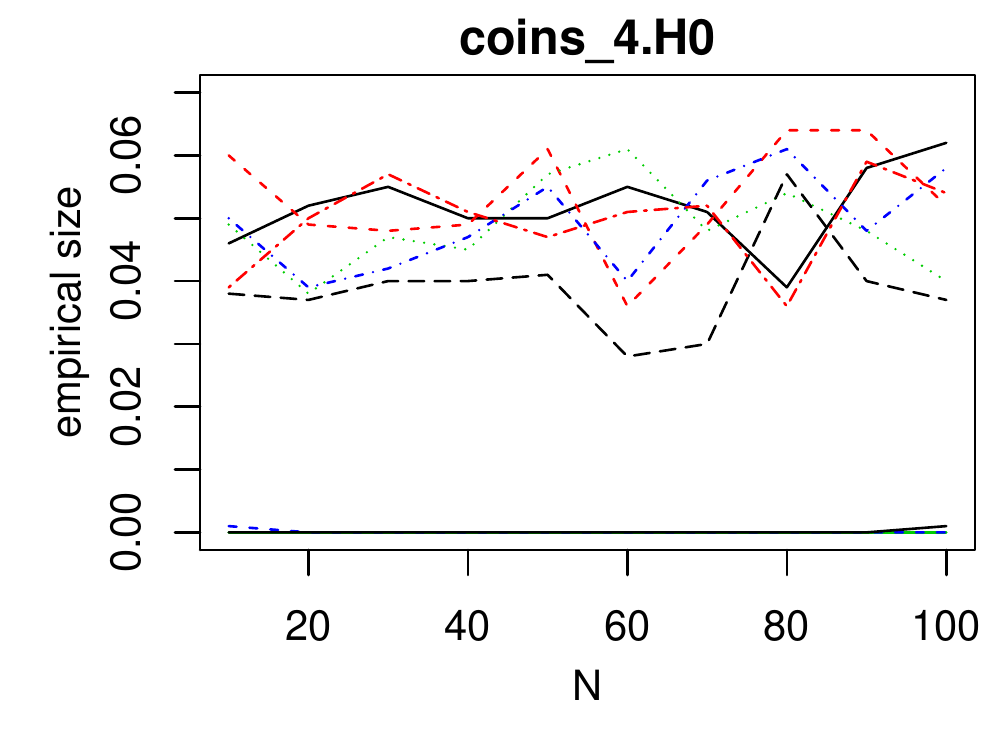}
\includegraphics[width = 0.3\textwidth]{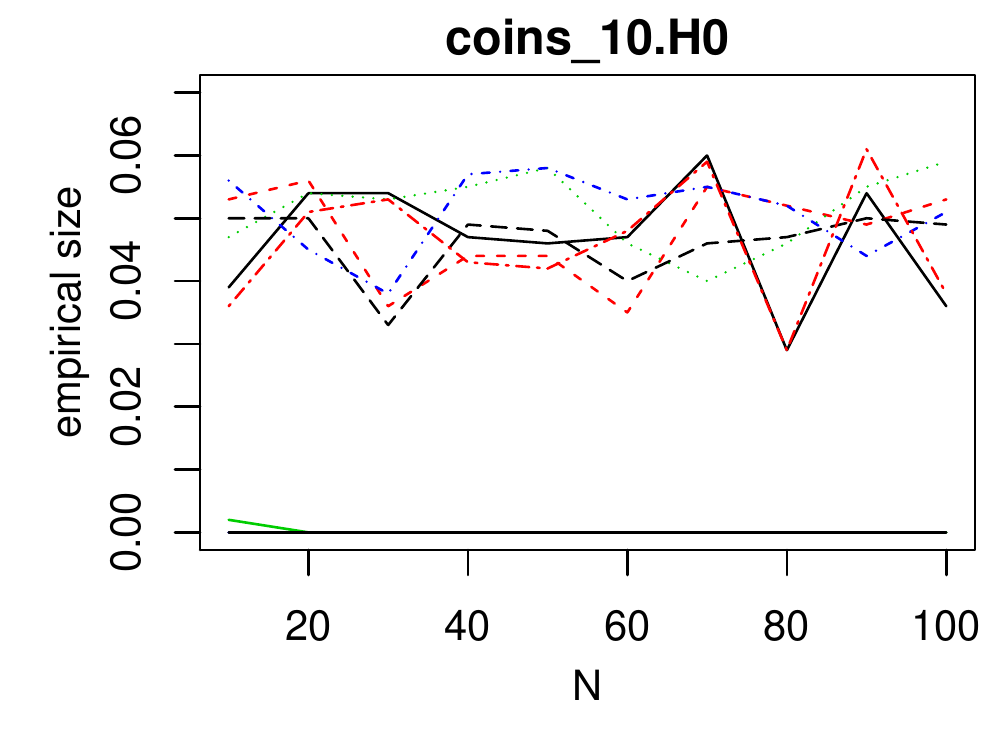}
\includegraphics[width = 0.3\textwidth]{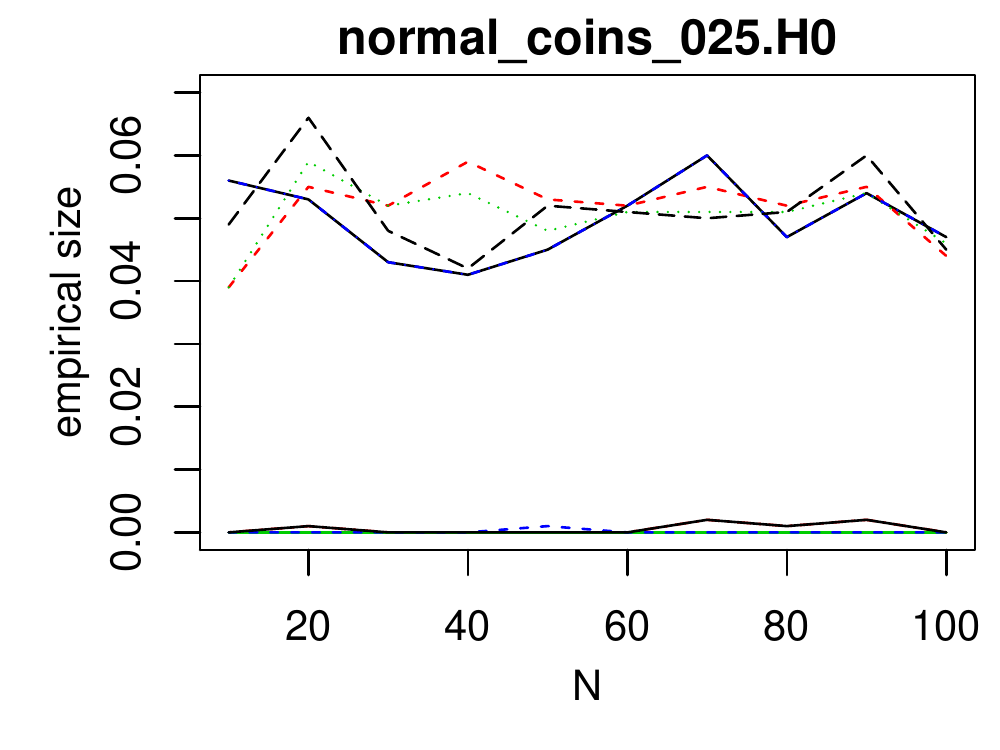}
\includegraphics[width = 0.3\textwidth]{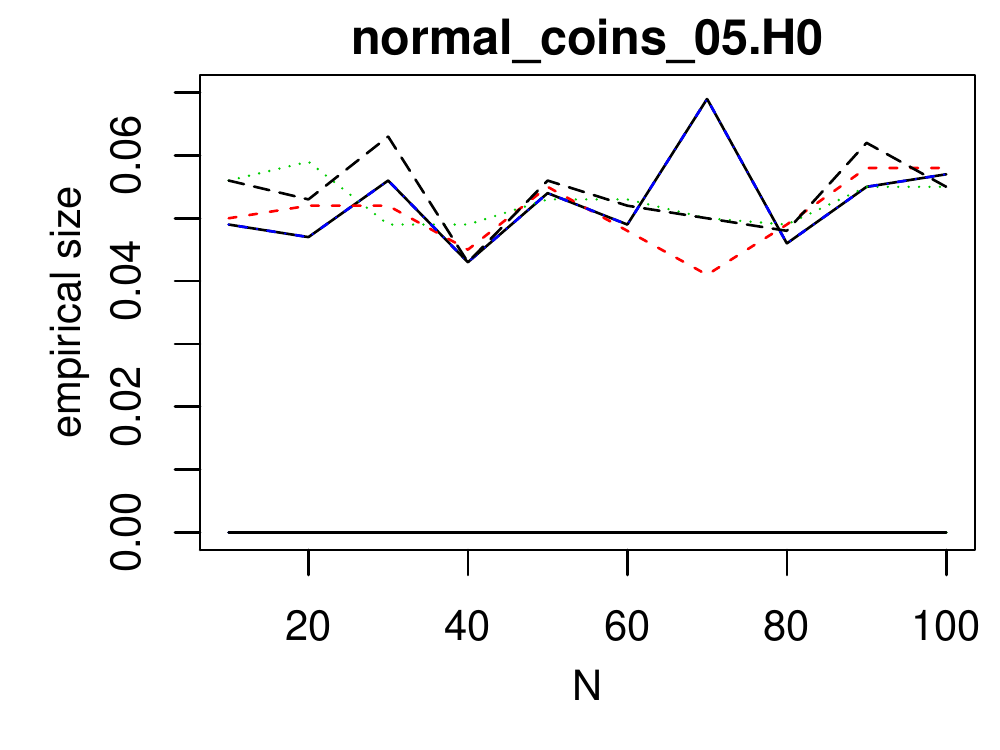}
\includegraphics[width = 0.3\textwidth]{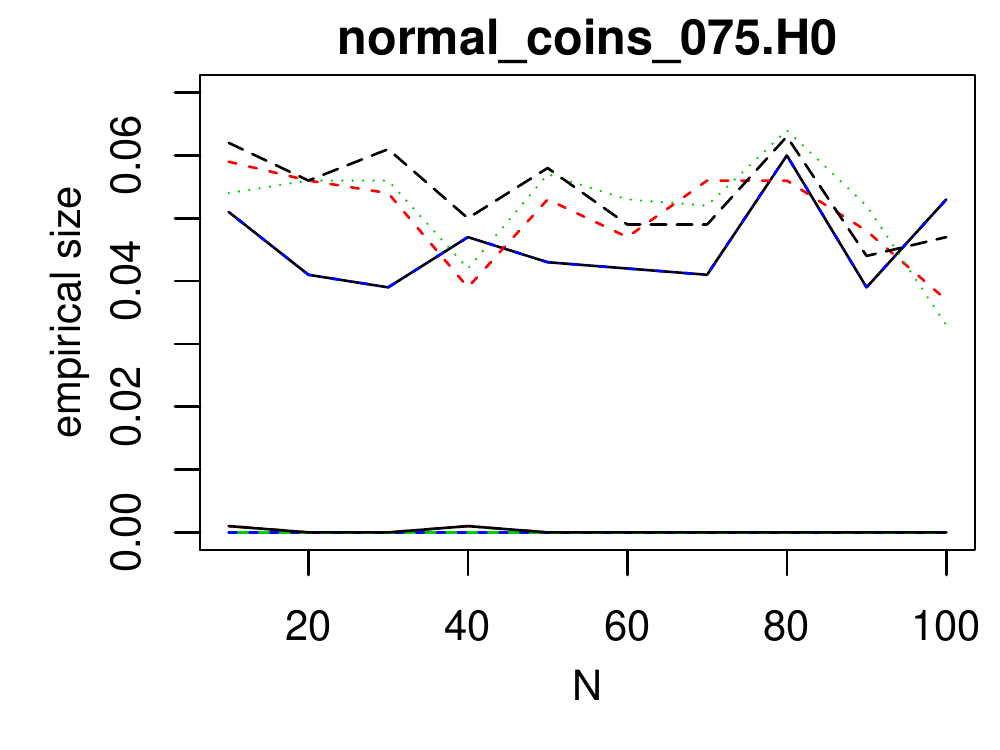}
\includegraphics[width = 0.3\textwidth]{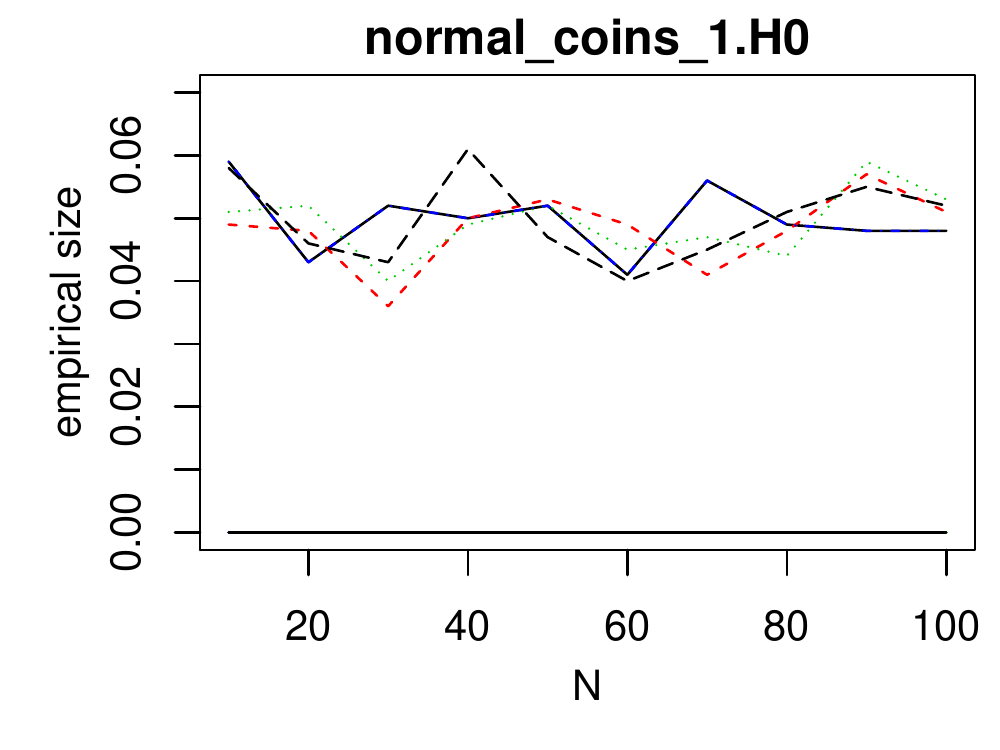}
\includegraphics[width = 0.3\textwidth]{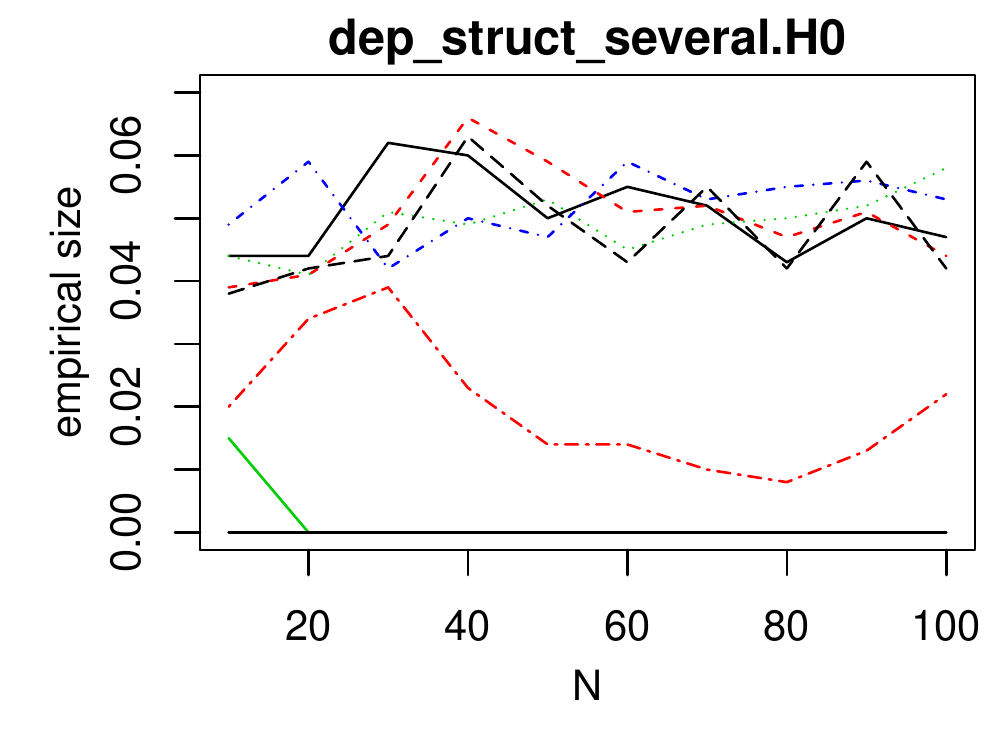}
\includegraphics[width = 0.3\textwidth]{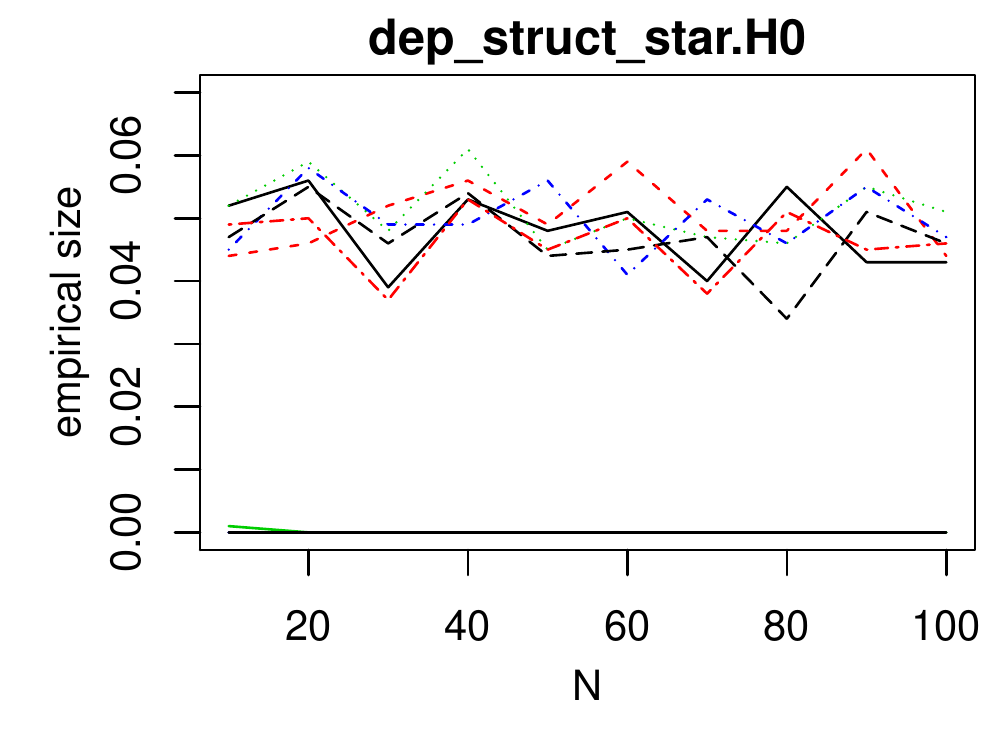}
\includegraphics[width = 0.3\textwidth]{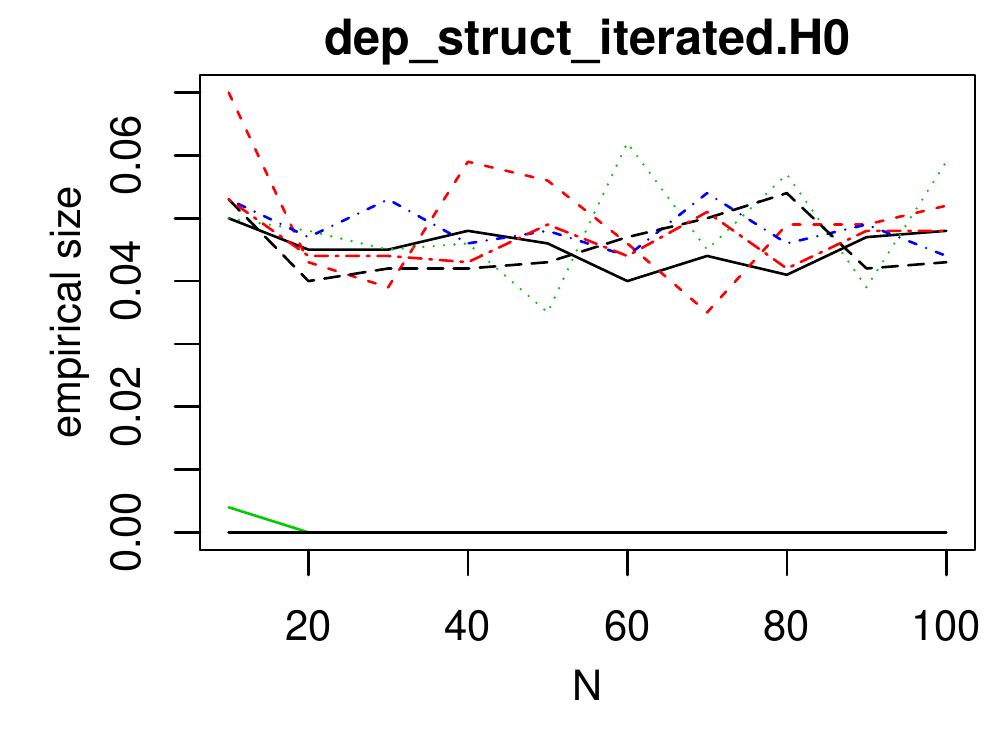}
\includegraphics[width = 0.3\textwidth]{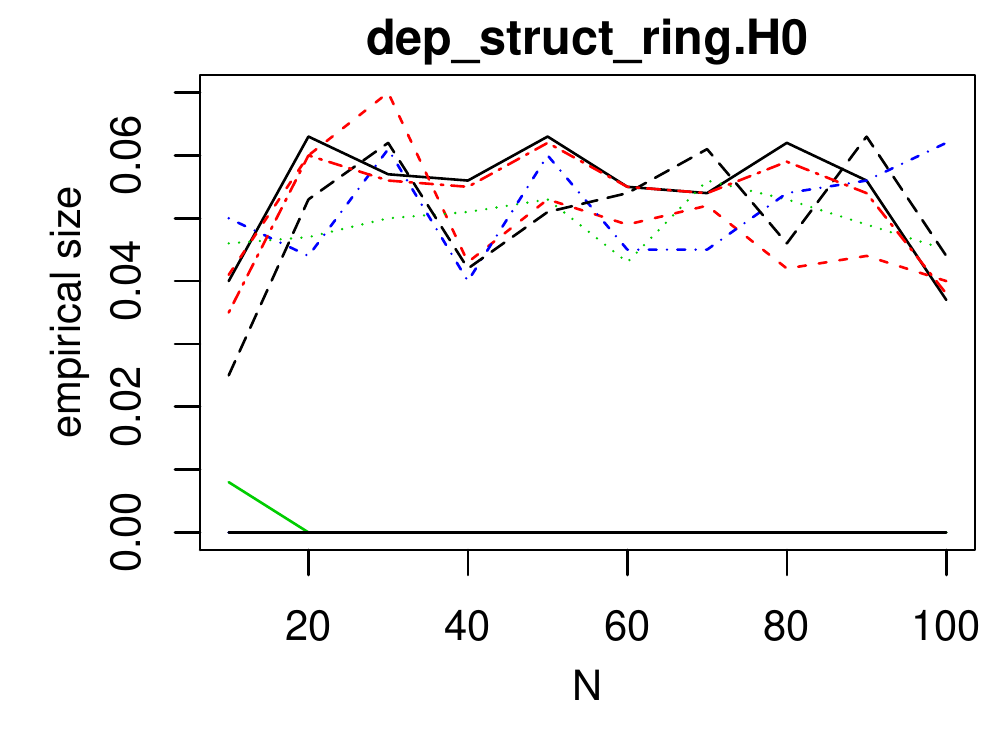}
\caption{The empirical size of the tests for Examples\ \ref{ex:2coins} to \ref{ex:ring} (Ex.\ \ref{ex:emp-size}).}	\label{fig:emp-size}
\end{figure}
Here we consider the same settings as in the previous examples but with $H_0$ data, i.e., the marginal distributions remain as in the examples but the components are now independent. In Figure \ref{fig:emp-size} the empirical sizes are depicted. The resampling methods have (as expected for a sharp test) an empirical size close to 0.05. For Bernoulli marginals also the distribution-free method for multivariance is close to 0.05. In the other cases (and for $m$- and total multivariance) the tests are conservative.

\end{example}

Next we analyze the effect of various parameters on the distribution of the test statistic.

\begin{example}[Influence of sample size, marginal distributions and $\psi_i$]\label{ex:variousm}
The distribution of the test statistic $N \cdot \hN \Mskript^2$ under the hypothesis of independence depends on the marginal distributions of the random variables and also on the number of variables $n$ as Figure \ref{fig:mdist-marginal} illustrates (see also Figure \ref{fig:mdist-dim}). The empirical distributions are based on 3000 samples each.
\begin{figure}[H]\centering
\includegraphics[width = 0.49\textwidth]{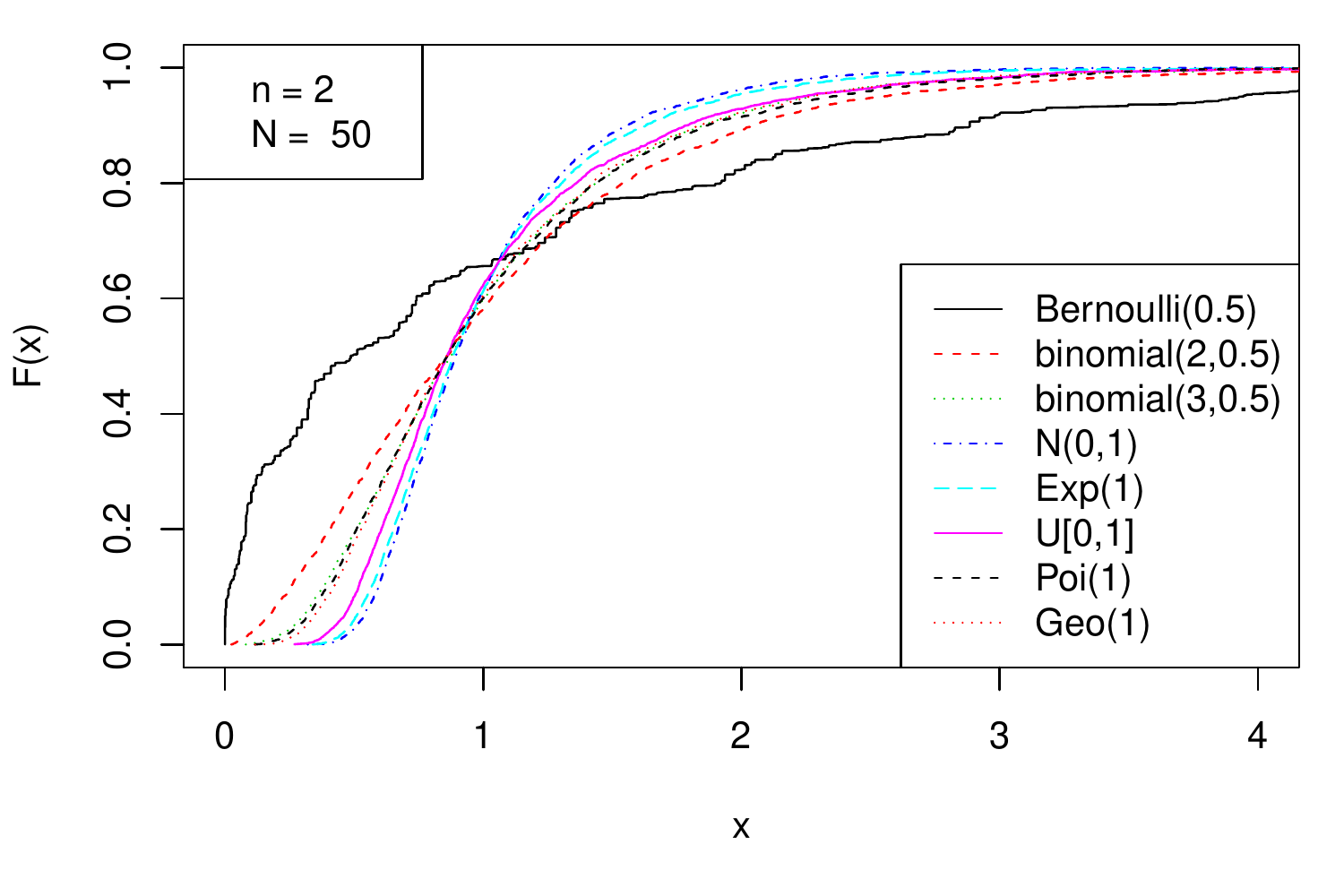}
\includegraphics[width = 0.49\textwidth]{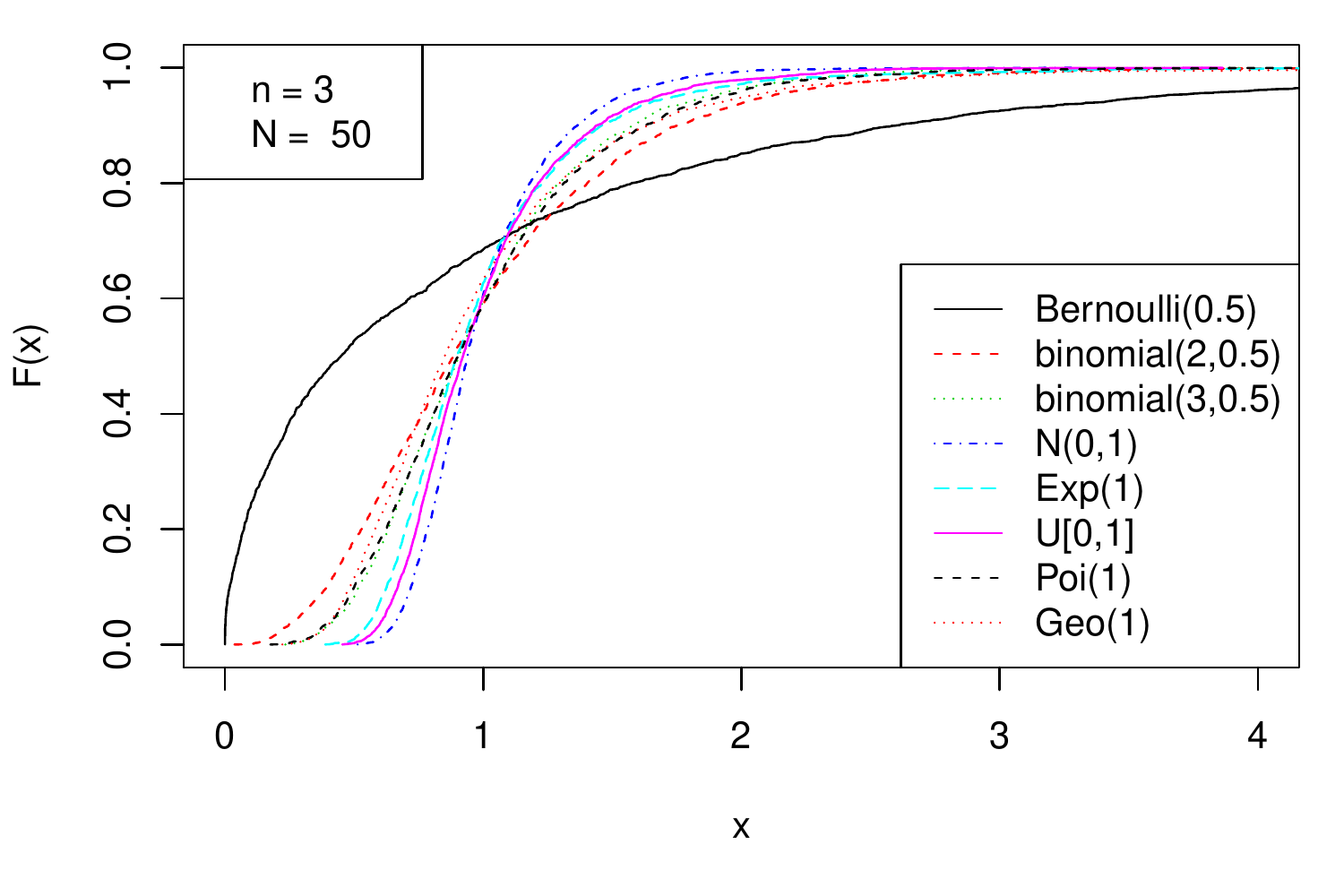}

\includegraphics[width = 0.49\textwidth]{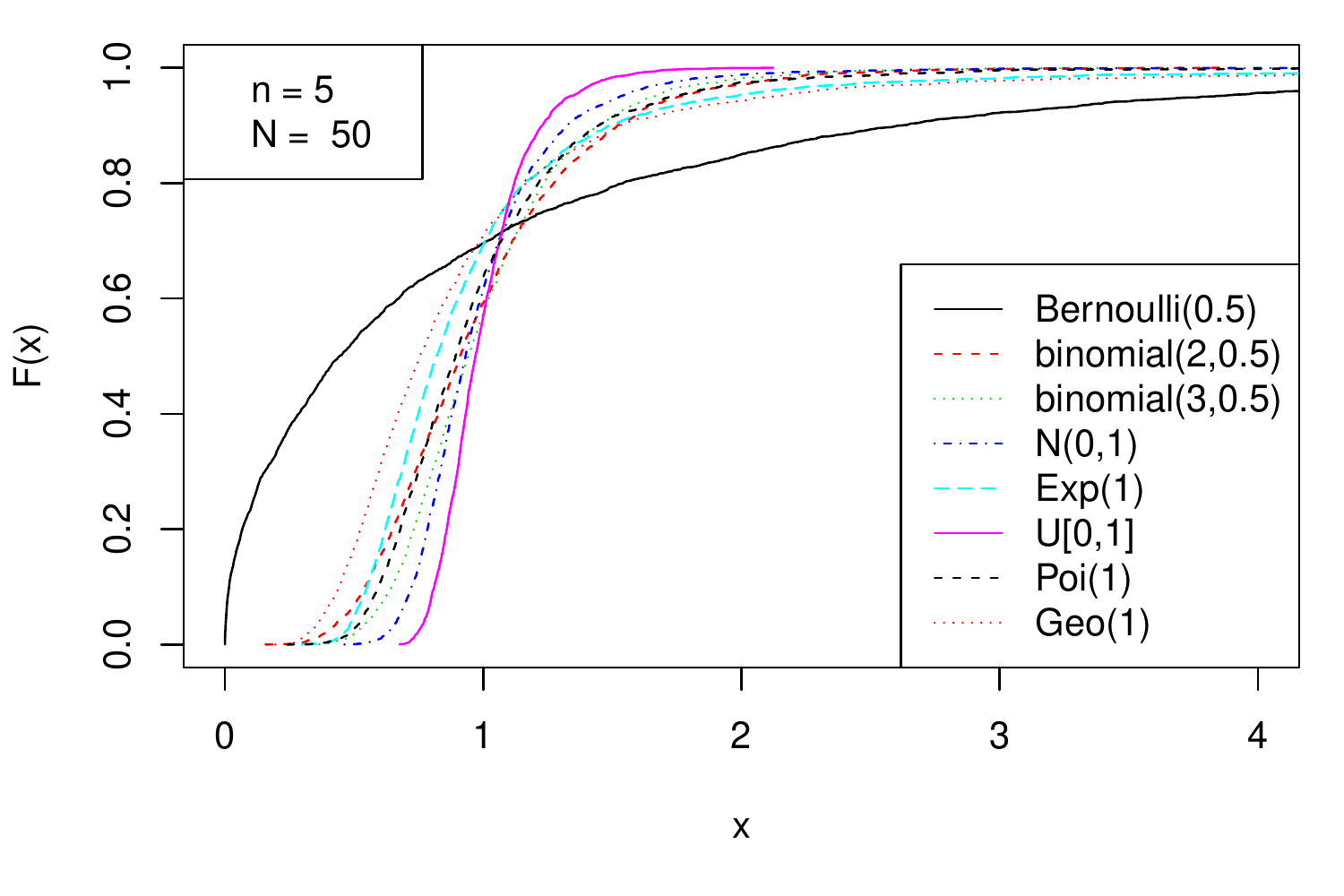}

\caption{Empirical distribution of $N \cdot \hN M_\rho^2(X_1,\ldots,X_n)$ for i.i.d.\  $X_i$ with various distributions (Ex.\ \ref{ex:variousm}).}	\label{fig:mdist-marginal}
\end{figure}

Moreover the distribution also clearly depends on the choice of the reference measure $\rho$ or equivalently (see \eqref{eq:psi} and Remark \ref{rem:choosepsi}) on the distances $\psi_i$. For Figure \ref{fig:mdist-psi} we used $\psi_i(x_i) = |x_i|^\alpha$ with $\alpha \in \{0.5,1,1.5\}$, and the plots show that in general for $\alpha = 1.5$ the upper tail of the distribution of the test statistic comes closer to the distribution-free limit which is the $\chi^2_1$-distribution. Note that the  $\chi^2_1$-distribution is matched in the case of Bernoulli distributed random variables, in this case the choice of $\psi_i$ has no effect on the empirical distribution of the test statistic, since $\psi_i(0) = 0$ and $\psi_i(1)=1$ for all $\alpha$.

\begin{figure}[H]\centering
\includegraphics[width = 0.49\textwidth]{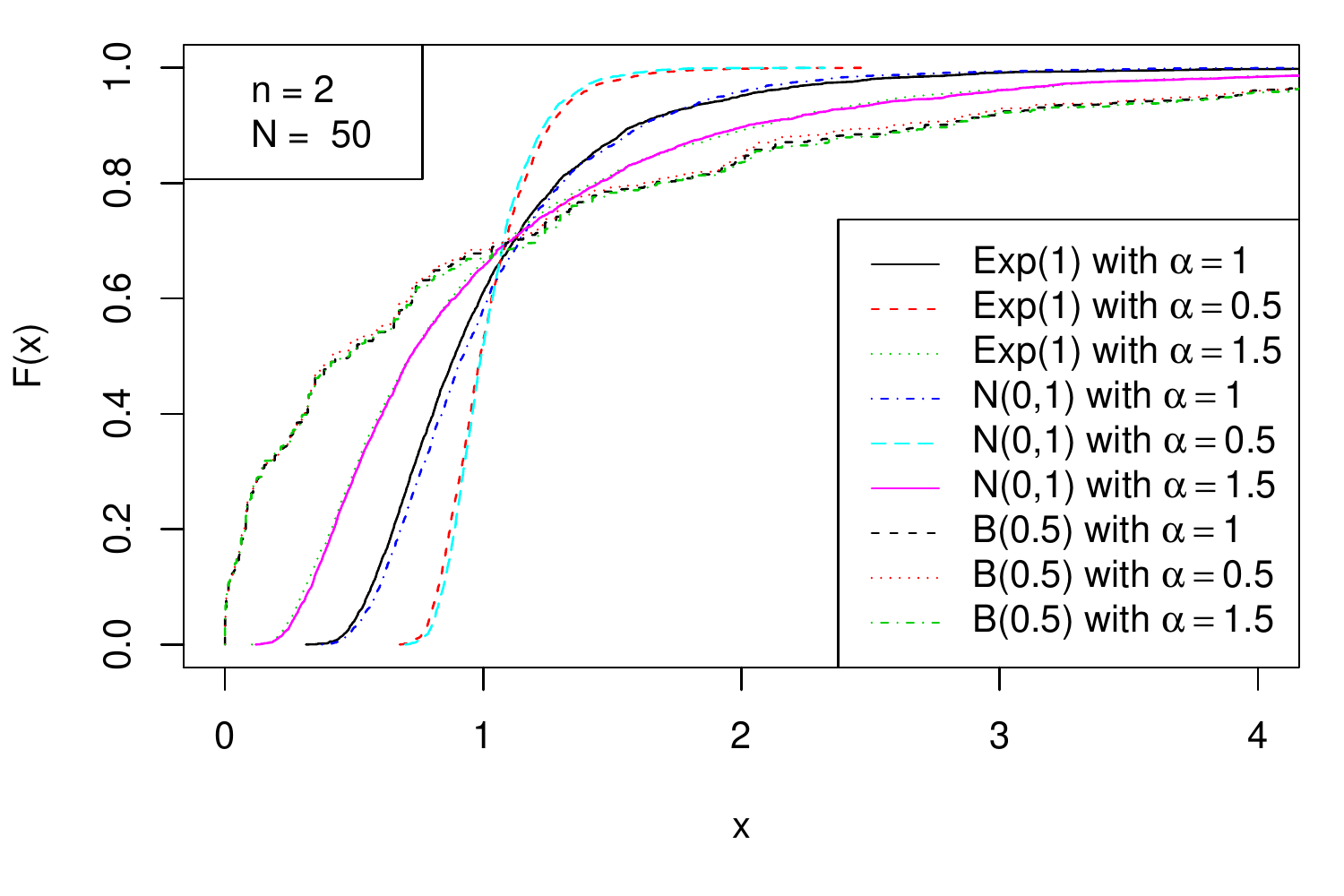}
\includegraphics[width = 0.49\textwidth]{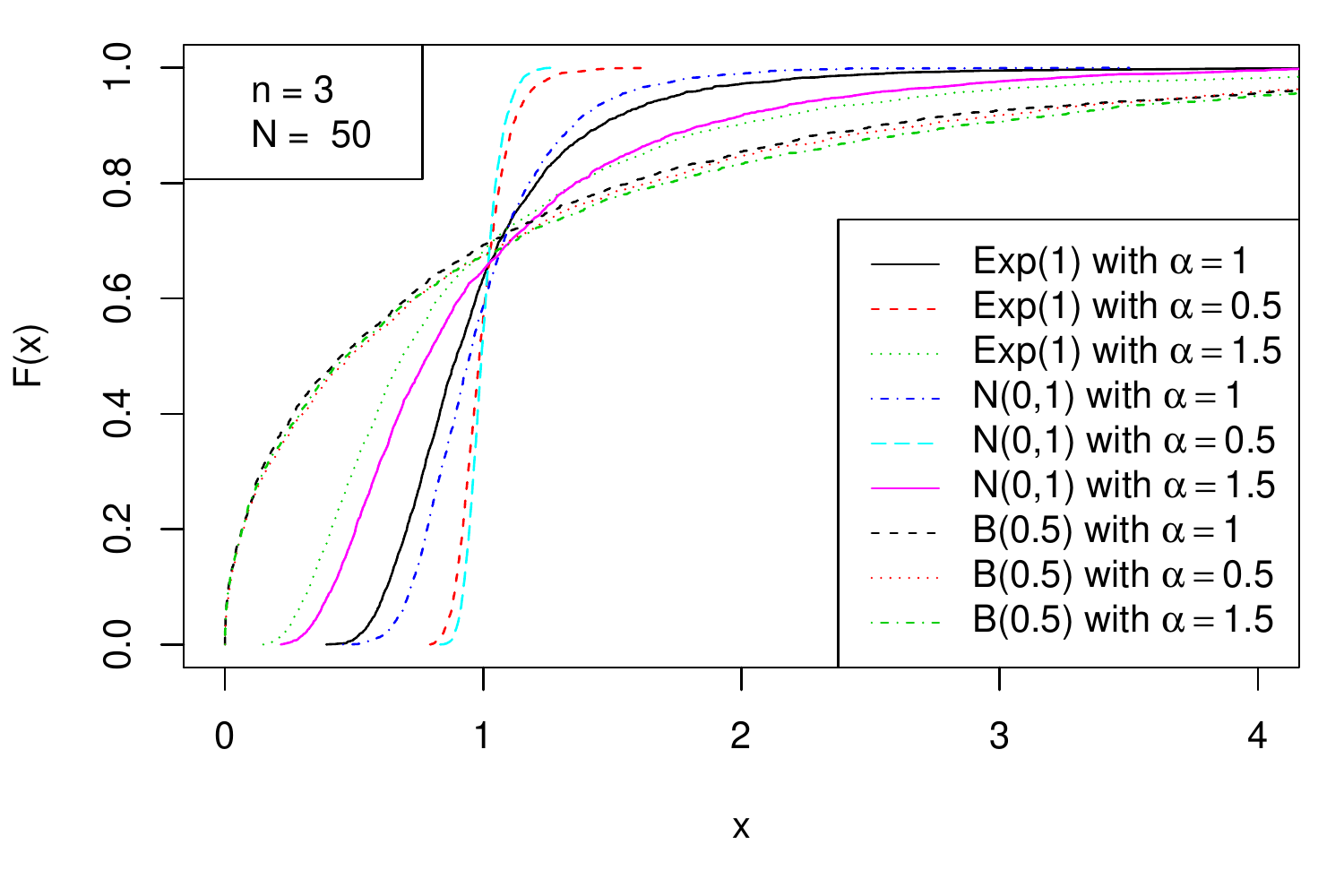}

\includegraphics[width = 0.49\textwidth]{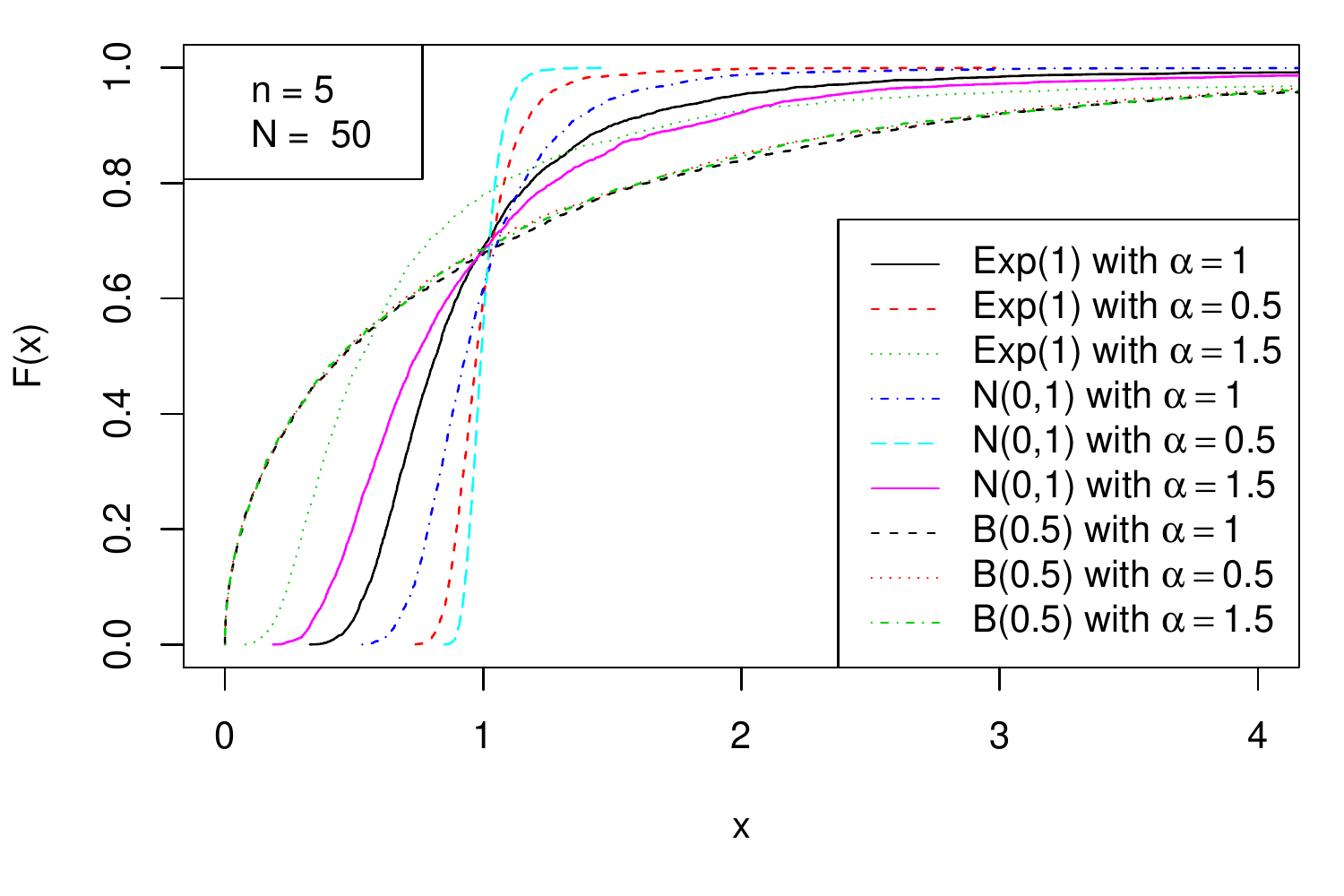}
\caption{Empirical distribution of $N \cdot \hN M_\rho^2(X_1,\ldots,X_n)$ for i.i.d.\  $X_i$ with various distributions and for $\psi_i (x) = |x|^\alpha$ (Ex.\ \ref{ex:variousm}).}	\label{fig:mdist-psi}
\end{figure}
\begin{figure}[H]\centering
\includegraphics[width = 0.49\textwidth]{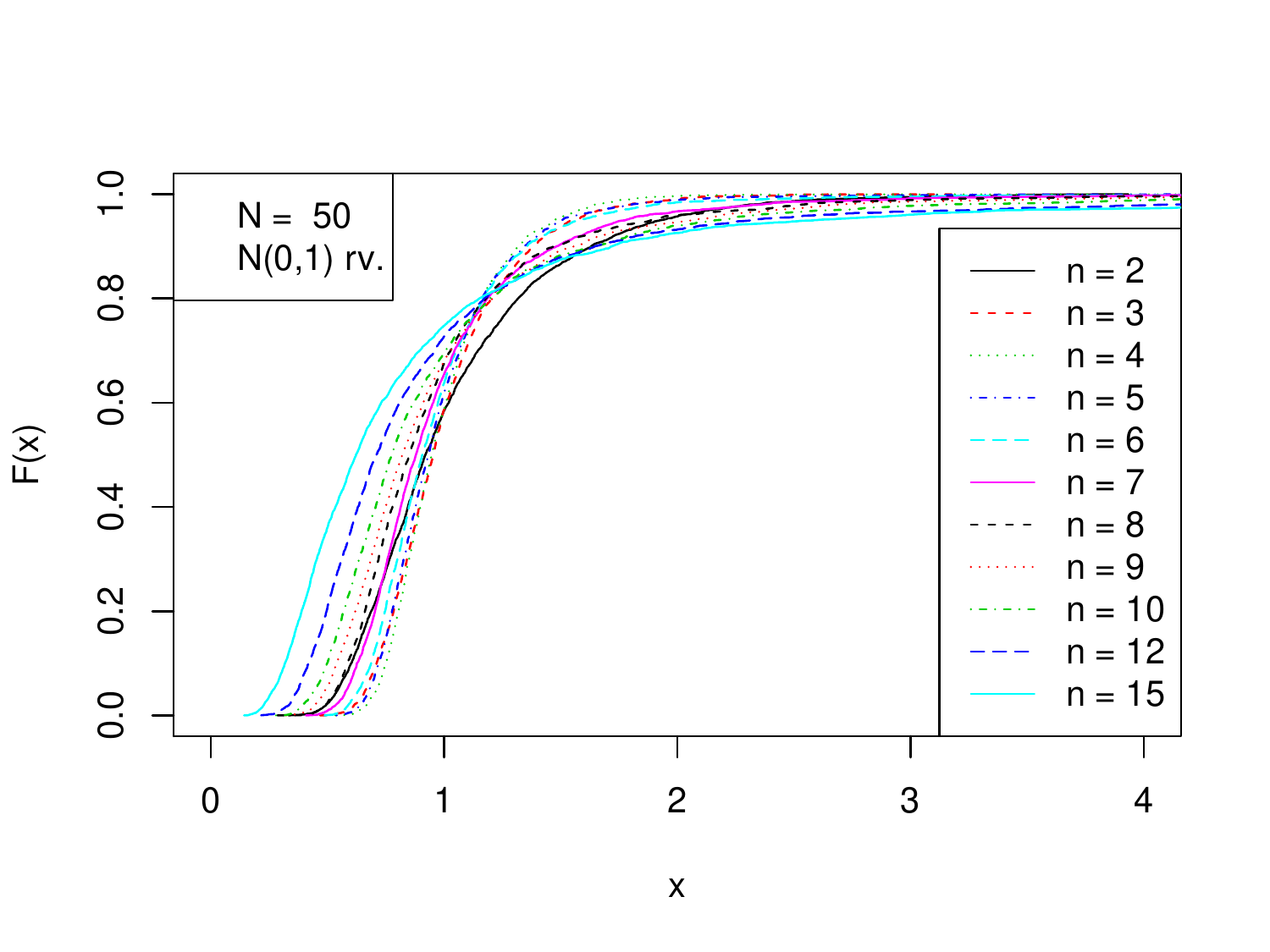}
\includegraphics[width = 0.49\textwidth]{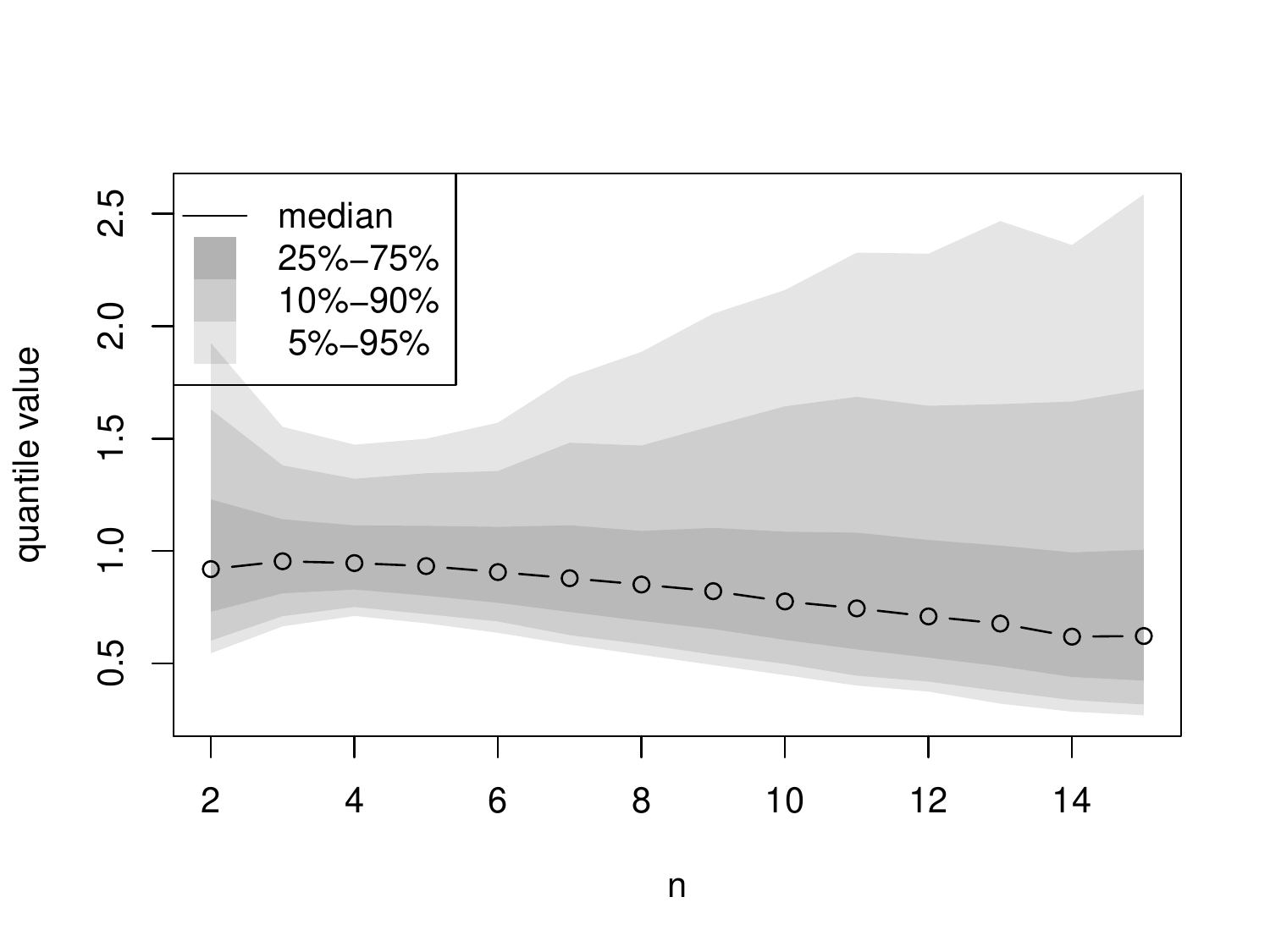}
\caption{The dependence of distribution $N \cdot \hN M_\rho^2(X_1,\ldots,X_n)$ for i.i.d.\  r.v.\ on the number of variables  (Ex.\ \ref{ex:variousm}).}
\label{fig:mdist-dim}
\end{figure}

\begin{figure}[H]\centering
\includegraphics[width = 0.49\textwidth]{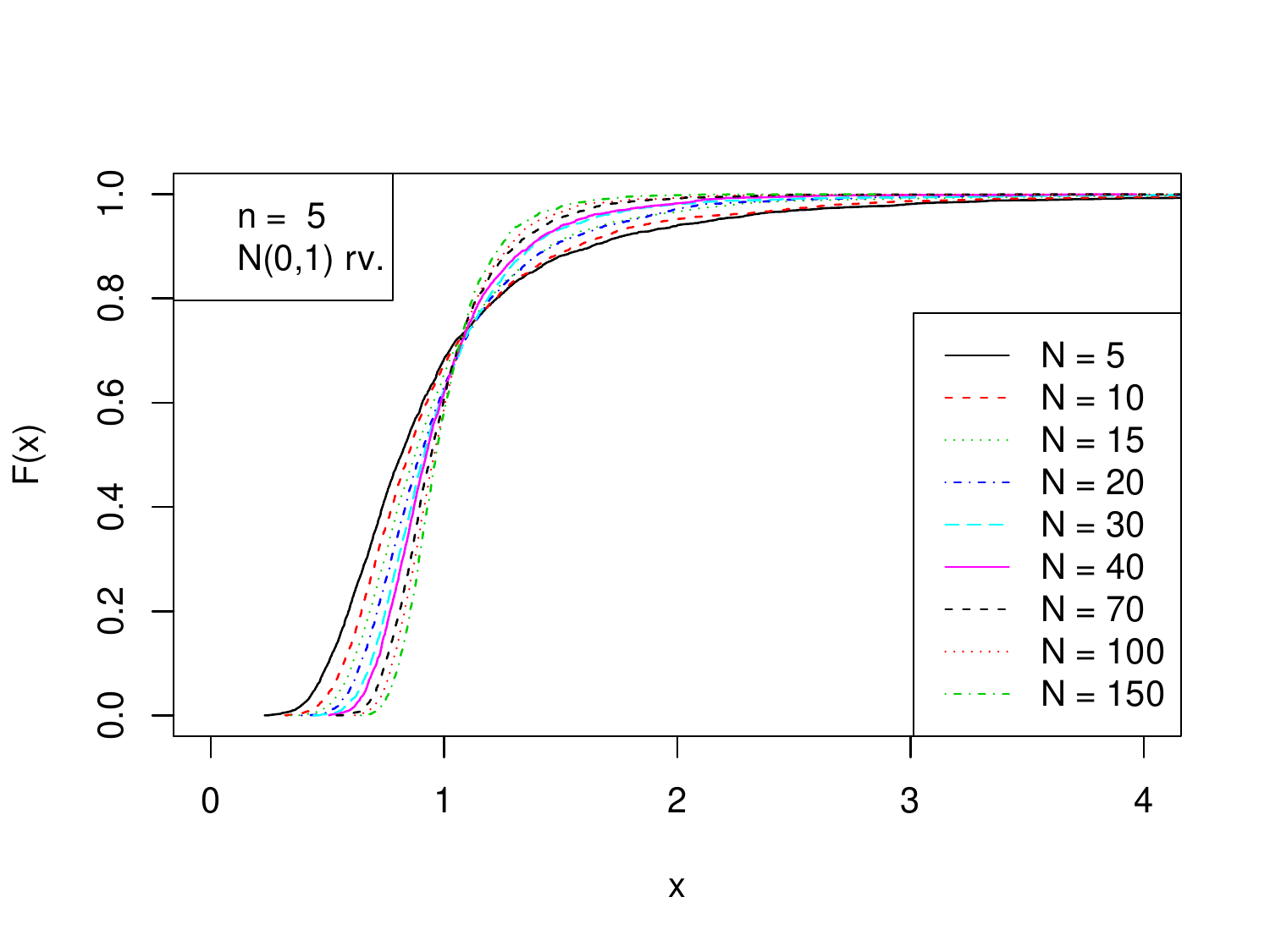}
\includegraphics[width = 0.49\textwidth]{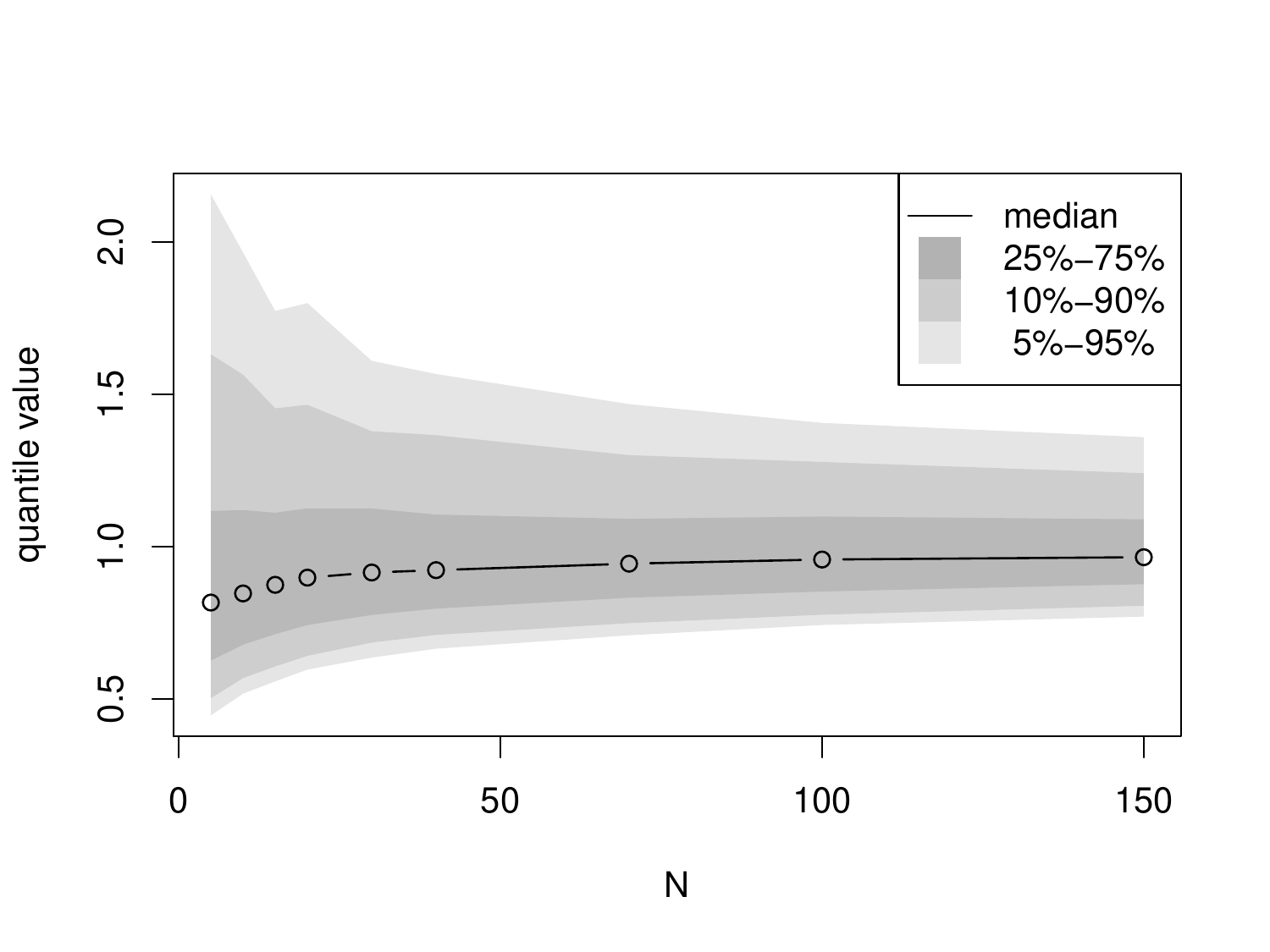}
\caption{The dependence of distribution $N \cdot \hN M_\rho^2(X_1,\ldots,X_n)$ for i.i.d.\  r.v.\ on the sample size  (Ex.\ \ref{ex:variousm}).}
\label{fig:mdist-samplesize}
\end{figure}
For independent normally distributed random variables the dependence of the test statistic on the number of variables $n$ is depicted in Figure \ref{fig:mdist-dim}, and the dependence on the sample size $N$ is illustrated in Figure \ref{fig:mdist-samplesize}. Roughly, the distribution spreads with the number of variables and shrinks to a limiting distribution (as stated in Theorem \ref{thm:convergence}) with increasing sample size.

\end{example}

\begin{example}[Computational complexity] \label{ex:comp-time}
To illustrate that the theoretical complexity $O(n N^2)$ is met by the computations, we computed distance multivariance for various values of $N$ and $n$ (using i.i.d.\  normal samples). In Figure \ref{fig:comp-time} the median of the computation time of 1000 repetitions for each combination of $n\in \{2,3,\ldots,10\}$ and $N\in\{10,20,\ldots,100\}$ is depicted. The linear growth in the dimension $n$ and the non-linear (quadratic) growth in the number of variables $N$ is clearly visible. 
\begin{figure}[H]\centering
		\includegraphics[width = 0.49\textwidth]{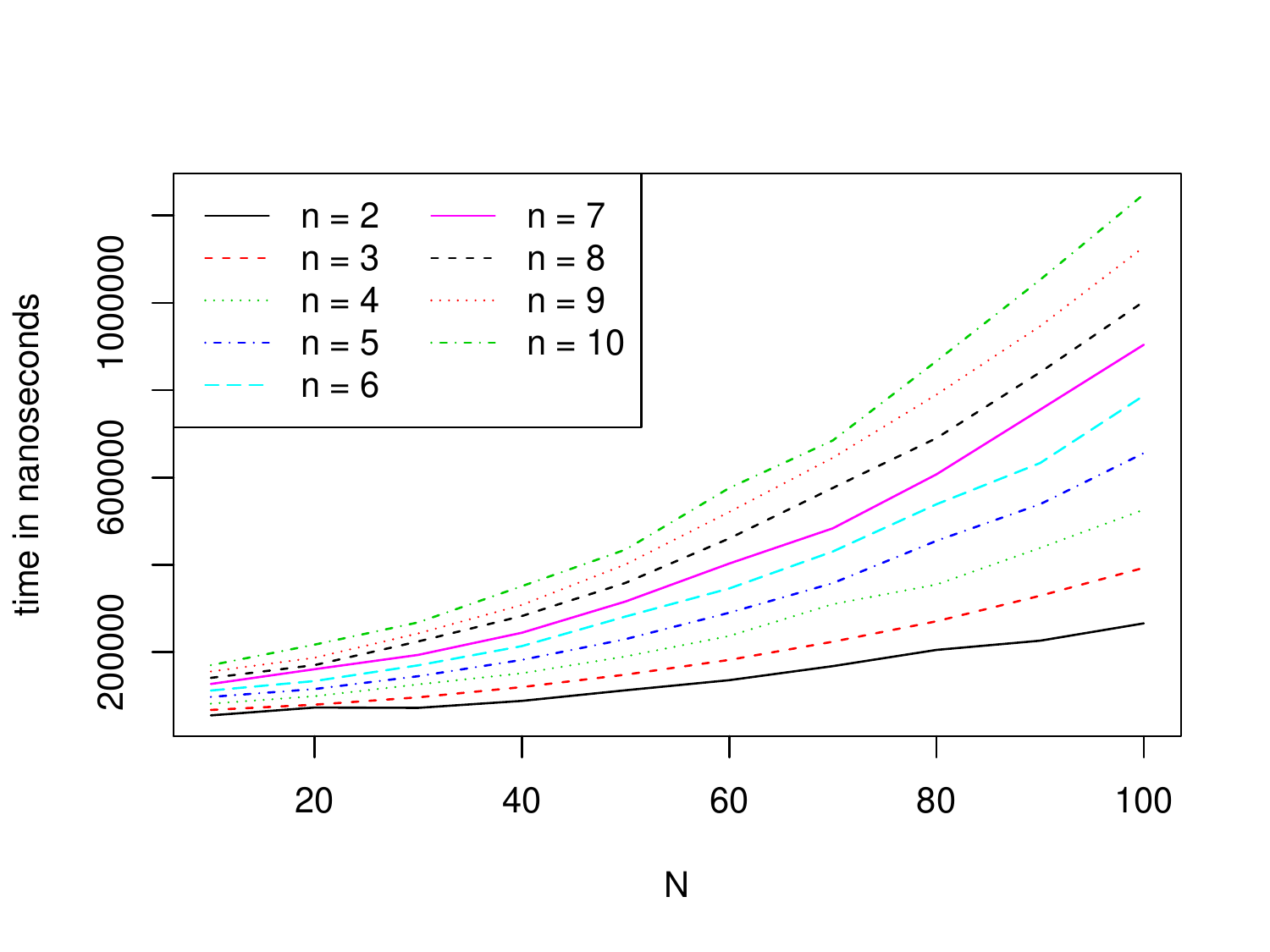}
		\includegraphics[width = 0.49\textwidth]{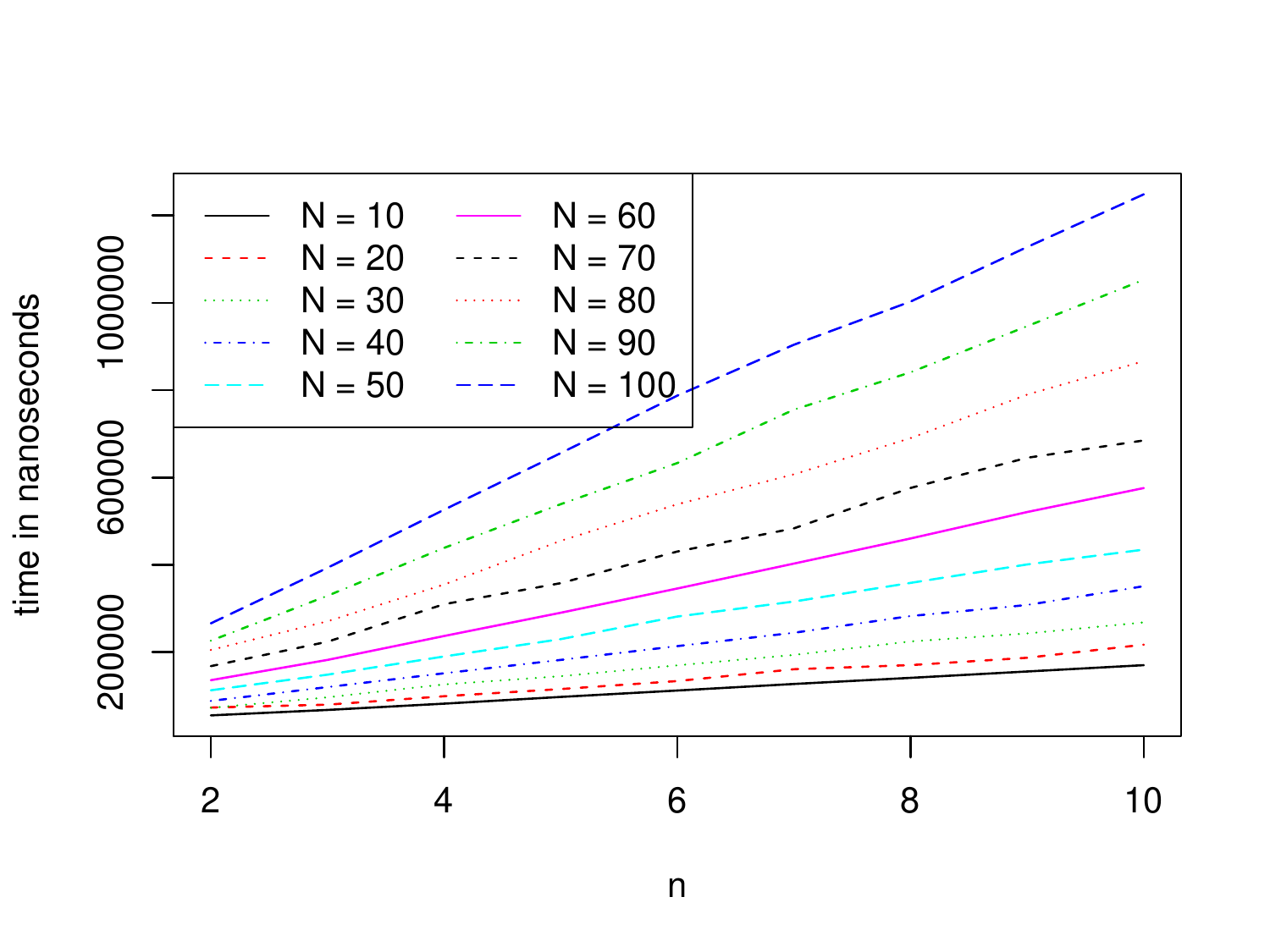}
		\caption{The computation time of multivariance dependent on sample size $N$ and dimension $n$ (Ex.\ \ref{ex:comp-time}).}
		\label{fig:comp-time}
	\end{figure}
\end{example}

\begin{example}[Infinite moments -- cf.\ Remark \ref{rem:choosepsi}] \label{ex:mom} 
Similar to Example \ref{ex:2coinsnorm} let $(Y_1,Y_2,Y_3)$ be the random variables corresponding to the events of $n = 2$ coins in Example \ref{ex:ncoins} and $Z_1,Z_2,Z_3$ be independent Cauchy distributed random variables. Now set $X_i := Y_i + r Z_i^3$ for $i=1,2,3$ and some fixed $r\in\R$ (here we only use $r= 0.001$). Note that $\E(|X_i|^\frac{1}{3})= \infty$, thus clearly the moment condition \eqref{eq:mom1-log} does not hold for the standard $\psi_i(.)=|.|$. Now we compare three methods: a) we don't care (thus we use the standard method); b) we use $\psi_i(\cdot)=\ln(1+\frac{|\cdot|^2}{2})$ which increases slowly enough such that the moments exist; c) we consider the bounded random variables $\arctan(X_i)$ instead of $X_i$ (cf.\ Remark \ref{rem:choosepsi}.\ref{rem:boundedrv}). The results are shown in Figure \ref{fig:mom-cauchy}. It turns out that method a) is not reliable, method b) works reasonably. In our setup method c) works best, but recall that this method destroys the translation and scale invariance of the test statistic, thus already if we shift our data it might not work anymore.
\begin{figure}[H]
\includegraphics[width = 0.3\textwidth]{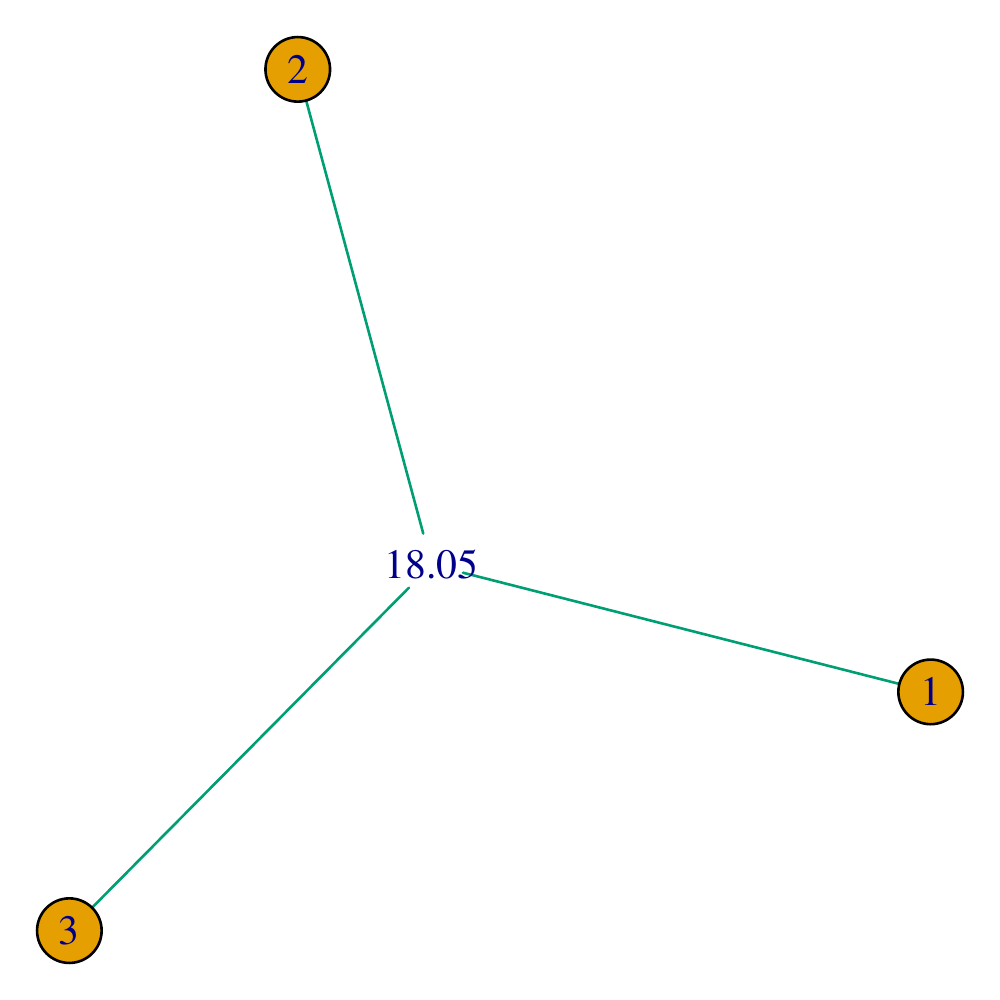}
\quad\centering
\begin{tabular}[b]{r|cc|cc|cc}
&\multicolumn{6}{c}{resampling}\\
& \multicolumn{2}{c|}{$\psi_i(\cdot)=|\cdot|$}& \multicolumn{2}{c|}{$\psi_i(\cdot)=\ln(1+\frac{|\cdot|^2}{2})$}& \multicolumn{2}{c}{$\arctan(X_i)$}\\
 $N$ & $\hN \Mskript$ & $\hN \overline{\Mskript}$ & $\hN \Mskript$ & $\hN \overline{\Mskript}$ & $\hN \Mskript$ & $\hN \overline{\Mskript}$ \\
 \hline
10 & 0.254 & 0.148 & 0.260 & 0.178 & 0.509 & 0.229 \\ 
  20 & 0.244 & 0.169 & 0.266 & 0.158 & 0.919 & 0.693 \\ 
  30 & 0.169 & 0.106 & 0.262 & 0.137 & 0.990 & 0.913 \\ 
  40 & 0.139 & 0.102 & 0.237 & 0.148 & 1.000 & 0.987 \\ 
  50 & 0.106 & 0.075 & 0.200 & 0.115 & 1.000 & 0.996 \\ 
  60 & 0.080 & 0.060 & 0.199 & 0.085 & 1.000 & 0.998 \\ 
  70 & 0.077 & 0.060 & 0.208 & 0.099 & 1.000 & 1.000 \\ 
  80 & 0.087 & 0.064 & 0.225 & 0.102 & 1.000 & 1.000 \\ 
  90 & 0.071 & 0.064 & 0.207 & 0.087 & 1.000 & 1.000 \\ 
  100 & 0.067 & 0.057 & 0.241 & 0.093 & 1.000 & 1.000 \\ 
 \hline	\end{tabular}
\caption{Multivariance for samples of a distribution with infinite expectation (Ex.\ \ref{ex:mom}).}\label{fig:mom-cauchy}
\end{figure}
\end{example}

\begin{example}[(total and $m$-)multivariance -- statistical curse of dimensions] \label{ex:averaging}
Let $X_1,\ldots, X_n$ be independent random variables and set $Y_1:=X_2$. Then (due to the independence of the $X_i$)
\begin{equation}
\overline{M}(Y_1,X_2,\ldots,X_n)-\overline{M}(X_1,\ldots,X_n) = M(Y_1,X_2) - 0 = M(Y_1,X_2)> 0.
\end{equation}
But the corresponding difference of the estimators might be negative, as a direct calculation shows. Empirically we study this setting with $X_i$ i.i.d.\  Bernoulli random variables. The empirical power of the independence test with resampling for $\hN \Mskript_2$ and $\hN \overline{\Mskript}$ is shown in Figure \ref{fig:averaging} for increasing $n$ and various sample sizes. As expected the decrease of power is rapid for total multivariance and at least not as bad for $2$-multivariance.

The resampling method was used, since the distribution-free test is not sharp in this setting. It is for univariate Bernoulli marginals only sharp for multivariance but not for total or $m$-multivariance ($m<n$).
\begin{figure}[H]\centering
\includegraphics[width = 0.49\textwidth]{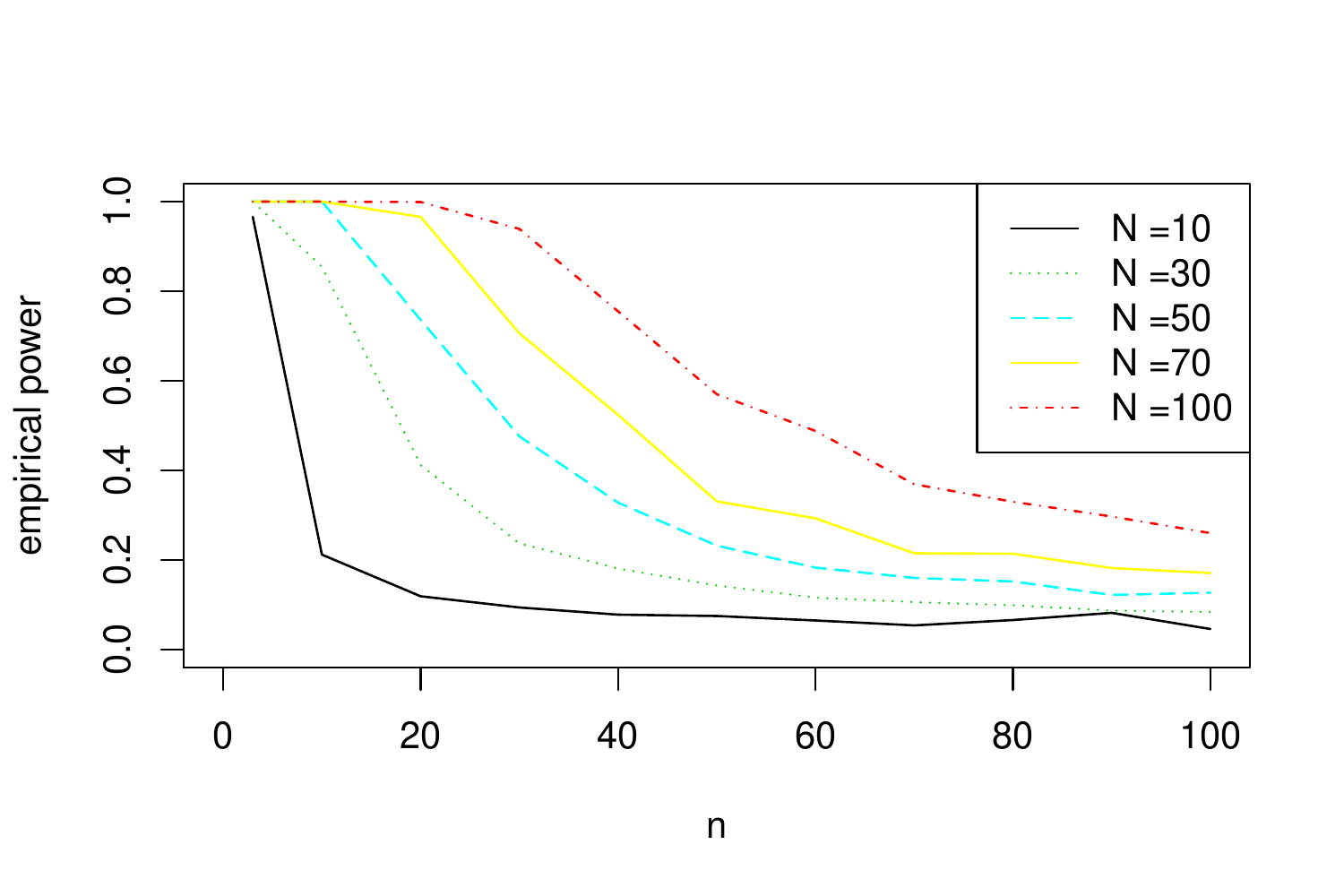}
\includegraphics[width = 0.49\textwidth]{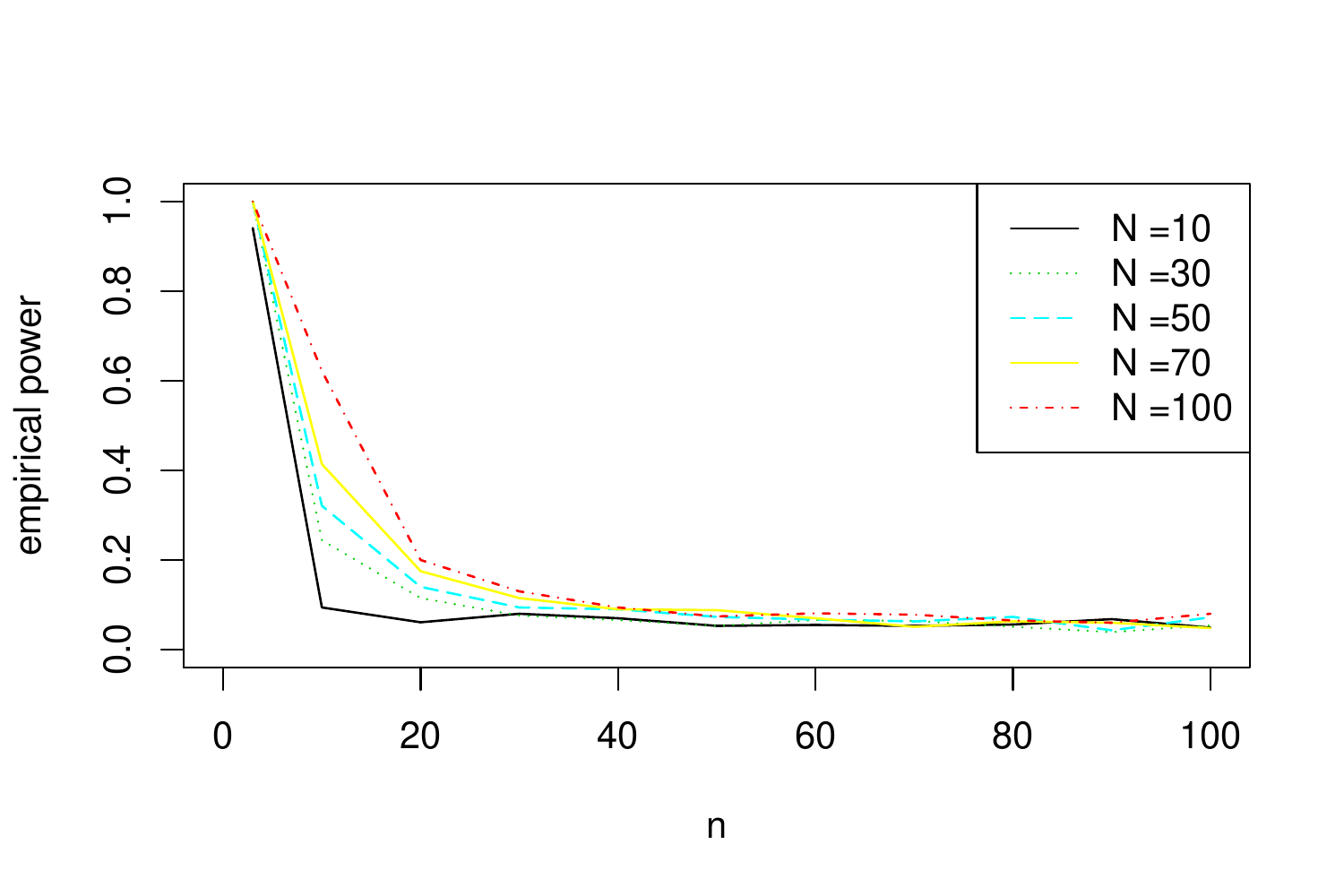}
\caption{The curse of dimension for $\hN\Mskript_2$ (left) and $\hN\overline{\Mskript}$ (right) using the resampling rejection level (Ex.\ \ref{ex:averaging}).}	\label{fig:averaging}
\end{figure}
\end{example}

We close this section with an extension of total multivariance which introduces a further parameter to tune the power of the tests. Recall that we assumed, in order to avoid distracting constants, in Section \ref{sec:compare} that the kernels of HSIC (and the related measures) satisfy $k_i(x_i,x_i) = 1$. Without this assumption additional constants appear naturally. As a special case one is led to the following dependence measure, which was incidentally suggested before \cite{Kellp} and it is for $\psi_i(x_i)=|x_i|$ a special case of the joint distance covariance developed in \cite{ChakZhan2019}.

\begin{example}[total distance multivariance with parameter $\lambda$] \label{ex:lambda} Let $\lambda > 0$ and define \textbf{$\bm{\lambda}$-total multivariance}
	\begin{equation}
	\overline{\Mskript_\rho}^2(\lambda;X_1,\ldots,X_n) := \sum_{\substack{1\leq i_1< \ldots < i_m \leq n\\2 \leq m \leq n}} M_{\otimes_{k=1}^m \rho_{i_k}}^2(X_{i_1},\ldots,X_{i_m}) \lambda^{n-m} 
	\end{equation}
	and its sample version 
	\begin{equation}
	\hN \overline{M}^2(\lambda;\bm{x}^{(1)},\ldots,\bm{x}^{(N)}) :=  \left[\frac{1}{N^2} \sum_{j,k=1}^N (\lambda+(A_1)_{jk}) \cdot \ldots \cdot (\lambda+(A_n)_{jk})\right] - \lambda^n.
	\end{equation}
	Thus one puts the weight $\lambda^{n-k}$ on the multivariance of each $k$-tuple for $k = 2,\ldots,n.$ Therefore with $\lambda<1$ the $n$-tuple gets the biggest weight, with $\lambda>1$ the $2$-tuples (i.e., pairwise dependence) get the biggest weight. This might be used to improve the detection rate of total multivariance as Figure \ref{fig:lambda-coins} and Figure \ref{fig:lambda-several} show. If the random variables are $(n-1)$-independent then clearly the detection improves when $\lambda$ gets closer to 0, Figure \ref{fig:lambda-coins}. If some lower order dependence is present then some optimal $\lambda$ seems to exist,  Figure \ref{fig:lambda-several}, but a priori its value seems unclear.
	
	\begin{figure}[H]
\includegraphics[width = 0.3\textwidth]{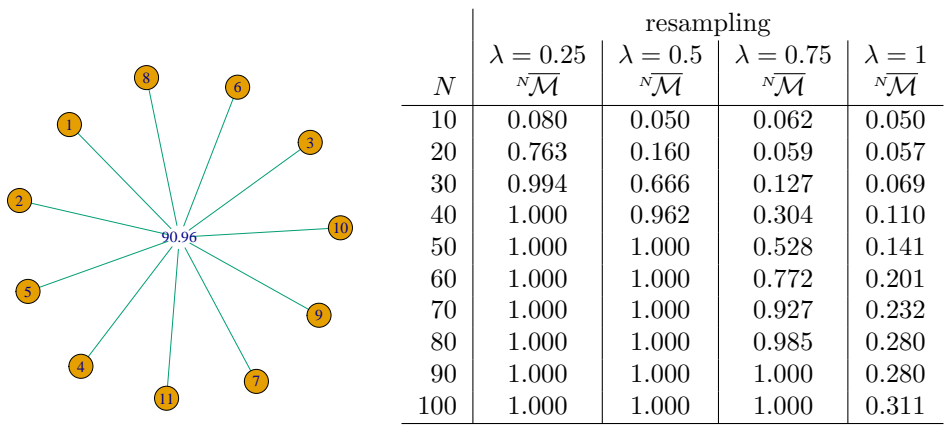}
		\quad\centering
		\begin{tabular}[b]{r|c|c|c|c}
			& \multicolumn{4}{c}{resampling}\\
			&$\lambda = 0.25$ &$\lambda = 0.5$ &$\lambda = 0.75$ &$\lambda = 1$ \\
			$N$ & $\hN \overline{\Mskript}$ & $\hN \overline{\Mskript}$ & $\hN \overline{\Mskript}$ & $\hN \overline{\Mskript}$ \\ 
			\hline
10 & 0.080 & 0.050 & 0.062 & 0.050 \\ 
   20 & 0.763 & 0.160 & 0.059 & 0.057 \\ 
   30 & 0.994 & 0.666 & 0.127 & 0.069 \\ 
   40 & 1.000 & 0.962 & 0.304 & 0.110 \\ 
   50 & 1.000 & 1.000 & 0.528 & 0.141 \\ 
   60 & 1.000 & 1.000 & 0.772 & 0.201 \\ 
   70 & 1.000 & 1.000 & 0.927 & 0.232 \\ 
   80 & 1.000 & 1.000 & 0.985 & 0.280 \\ 
   90 & 1.000 & 1.000 & 1.000 & 0.280 \\ 
   100 & 1.000 & 1.000 & 1.000 & 0.311 \\ 
			\hline
		\end{tabular}
		\caption{Empirical power of $\lambda$-total multivariance for the events of 10 coins, compare to Figure \ref{fig:ncoins} (Ex.\ \ref{ex:lambda}).}
		\label{fig:lambda-coins}
	\end{figure}		
	
	\begin{figure}[H]
			\includegraphics[width = 0.3\textwidth]{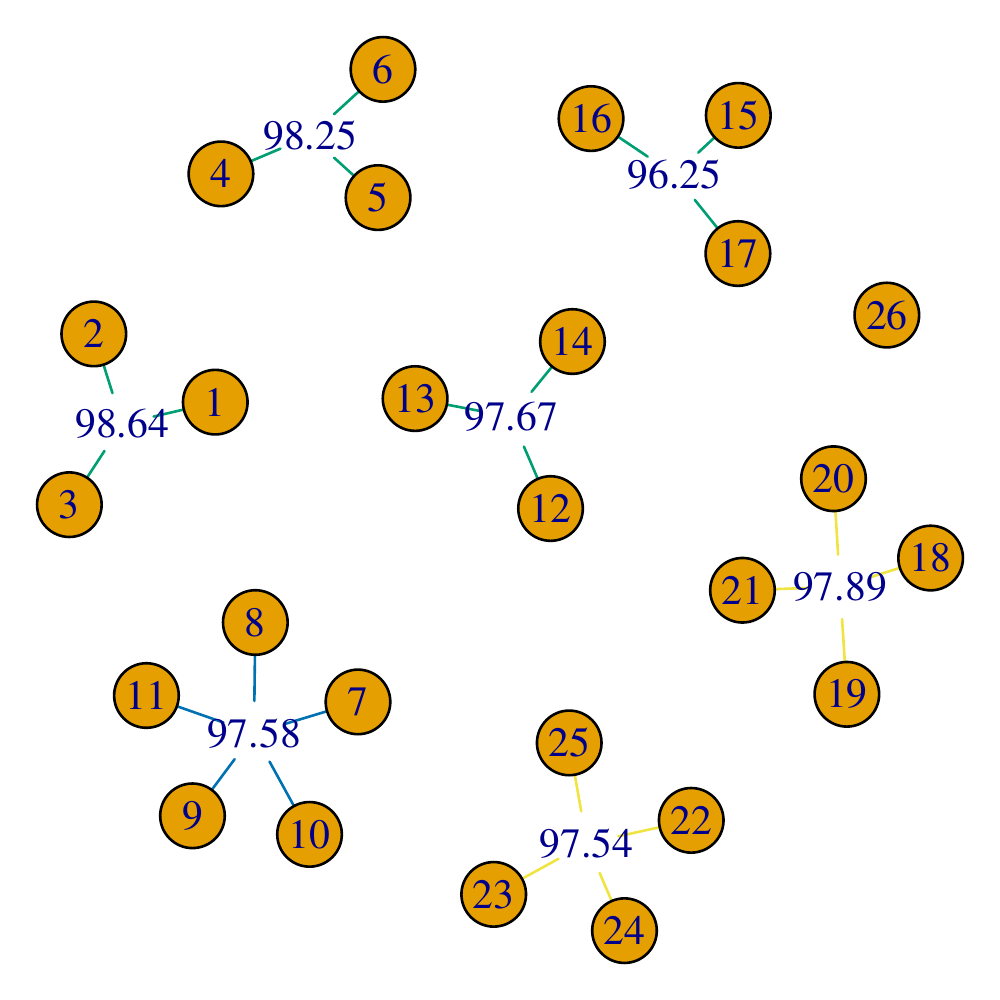}
		\quad\centering
		\begin{tabular}[b]{r|c|c|c|c}
			& \multicolumn{4}{c}{resampling}\\
			&$\lambda = 1$ &$\lambda = 2$ &$\lambda = 4$ &$\lambda = 8$ \\
			$N$ & $\hN \overline{\Mskript}$ & $\hN \overline{\Mskript}$ & $\hN \overline{\Mskript}$ & $\hN \overline{\Mskript}$ \\ 
			\hline
 10 & 0.051 & 0.041 & 0.041 & 0.051 \\ 
   20 & 0.034 & 0.035 & 0.056 & 0.104 \\ 
   30 & 0.031 & 0.042 & 0.093 & 0.147 \\ 
   40 & 0.052 & 0.043 & 0.115 & 0.156 \\ 
   50 & 0.053 & 0.048 & 0.181 & 0.202 \\ 
   60 & 0.043 & 0.055 & 0.263 & 0.272 \\ 
   70 & 0.040 & 0.057 & 0.424 & 0.343 \\ 
   80 & 0.050 & 0.063 & 0.547 & 0.430 \\ 
   90 & 0.046 & 0.069 & 0.693 & 0.501 \\ 
   100 & 0.047 & 0.080 & 0.785 & 0.562 \\ 
			\hline
		\end{tabular}	
		\caption{Empirical power of $\lambda$-total multivariance for the dependence in Example \ref{ex:several}, compare to Figure \ref{fig:several}  (Ex.\ \ref{ex:lambda}).}
		\label{fig:lambda-several}
	\end{figure}		
\end{example}



\end{document}